\documentclass[11pt,reqno]{amsart}
\usepackage[english]{babel}
\usepackage{mathtools,amsmath,amssymb}
\usepackage{mathrsfs,dsfont}
\usepackage{enumitem}
\usepackage{soul,xcolor}
\usepackage{tikz-cd}
\usepackage{stmaryrd}
\usepackage{slashed}
\usepackage[pdfencoding=auto, psdextra]{hyperref}
\hypersetup{colorlinks=true,allcolors=[rgb]{0,0,0.6}}
\usepackage[hmargin=2.7cm,vmargin=2.6cm]{geometry}

\newtheorem{thm}{Theorem}[section]
\newtheorem{prop}[thm]{Proposition}
\newtheorem{lemma}[thm]{Lemma}
\newtheorem{cor}[thm]{Corollary}
\theoremstyle{definition}
\newtheorem{defn}[thm]{Definition}
\newtheorem{notation}[thm]{Notation}

\theoremstyle{remark}
\newtheorem{remark}[thm]{Remark}
\numberwithin{equation}{section}

\def\N{\mathbb{N}}

\def\R{\mathds{R}}
\def\C{\mathds{C}}

\def\Imm{\Im\mathrm{m}}
\def\Ree{\Re\mathrm{e}}
\def\ci{\dot\iota}
\def\Tr{\mathrm {Tr}}

\begin{document}

\title[QED $\overset{\hslash\to 0}{\longrightarrow} $ CED]{Towards a derivation of Classical ElectroDynamics of charges and fields from QED}

\author{Zied Ammari}
\address{Univ Rennes, [UR1], CNRS, IRMAR - UMR 6625, F-35000 Rennes, France.}
\email{zied.ammari@univ-rennes1.fr}

\author{Marco Falconi}
\address{Politecnico di Milano \\ D-Mat\\ Piazza Leonardo da Vinci \\ 20133 Milano \\ Italy}
\email{marco.falconi@polimi.it}

\author{Fumio Hiroshima}
\address{Faculty of Mathematics\\ Kyushu University\\ Motooka 744\\ Nishiku, Fukuoka 819-0385\\ Japan}
\email{hiroshima@math.kyushu-u.ac.jp}

\date{February 2, 2022}


\begin{abstract}
The purpose of this article is twofold:
  \begin{itemize}
  \item On one hand, we rigorously derive the Newton--Maxwell equation in the
    Coulomb gauge from first principles of quantum electrodynamics in
    agreement with the formal Bohr's correspondence principle of quantum
    mechanics.
  \item On the other hand, we establish the global
   well-posedness of the Newton--Maxwell system on energy-spaces under weak
    assumptions on the charge distribution.
  \end{itemize}
  Both results improve the state of the art, and are obtained by
  incorporating semiclassical and measure theoretical techniques. One of the novelties is the use of quantum propagation properties in order to build global solutions of
  the Newton--Maxwell equation.
\end{abstract}
\maketitle
\tableofcontents

\section{Introduction}
\label{sec.intro}

The dynamics of charged particles in interaction with the electromagnetic
field is a topic that has been at the heart of physical and mathematical
investigations for the most part of the last two centuries, and it is still
of great interest for a large community of scientists, \cite{MR0436782}. In the current state
of the art, we are still far from having a complete understanding of such
dynamics, especially from a mathematical perspective: the interaction carries
singularities of various nature both at the classical and quantum levels,
that are quite difficult to deal with (see, \emph{e.g.},~\cite{MR3837508,MR1410064}). It is well known that, from a
classical standpoint, an atomic system is not stable:   according to Larmor's formula electrons would lose
energy  by radiation when accelerated by the nucleus' electrostatic force,
thus collapsing ! In addition, the charge distribution of a given particle
cannot be concentrated on a single point.   Otherwise, it would yield an infinite electrostatic energy caused by the self-interaction of the particle with its own field; and consequently leading to known pathological behaviors of the Abraham--Lorentz radiation reaction  force.
On the other hand, if one passes to the quantum theory (that was indeed developed in good part to address the issues
raised by classical electrodynamics), a stable atomic model can be
defined. However, the relativistic nature of the quantum electromagnetic
field introduces new divergences (both infrared and ultraviolet) that require
a so-called renormalization procedure to be dealt with.  Unfortunately,  renormalization
 for quantum electrodynamics, both with relativistic lepton
fields and nonrelativistic point charges, has not yet been put on firm
mathematical grounds (see \cite{spohn2004dcp}, for a more detailed
discussion).
 In the absence of such Lorentz covariant fundamental theory, one can alternatively state that the classical theory of electrodynamics is not appropriate for describing the behavior of an electron at a distance less than its Compton wavelength; and thus justifying the introduction of extended (non-point) charge distributions.

\medskip

The aim of this paper is to study the mathematical interplay between the quantum and
classical features of electrodynamics, in the case of extended and
nonrelativistic charges in interaction with the electromagnetic
field. Physically, we are thinking about molecules, ions or atoms (possibly
different among each other) interacting both among themselves, with a pair
potential, and with the electromagnetic field. From a classical standpoint,
the dynamics is governed by the Newton--Maxwell system of equations, a coupled
PDE-ODE nonlinear system, while at the quantum level the linear evolution is
generated by the so-called Pauli--Fierz Hamiltonian,  \cite{MR1639713}. We rigorously derive the
classical dynamics starting from the quantum one in the limit $\hslash\to 0$ (here
$\hslash$ is, as usual in semiclassical analysis, a dimensionless parameter
measuring how strong the quantum effects are in the system). Thus, verifying
the validity of Bohr's correspondence principle. In addition, we prove global
well-posedness of the classic Newton--Maxwell system for ``rough'' charge
distributions.

Bohr's correspondence principle states that ``in the limit
  of large quantum numbers'' (\emph{i.e.}\ at scales where the noncommutative
  nature of observables becomes irrelevant) the physical description of a
  quantum system shall become the one dictated by its classical
  analogue. While such commonly accepted physical principle is nowadays
  mathematically understood to a good extent for particle systems, less is
  known for field theories. For the electrodynamics of extended charges,
  there are partial results available, limited only to coherent initial
  quantum states of minimal uncertainty, with no initial correlation between
  the particle's and field's subsystems~\cite{Knowles2009LimitingDI}.

We extend such analysis in two directions, in Theorem
  \ref{thm:main-2}: On one hand, we relax the assumptions on the charge
  distributions of the particles, and on the other hand, we prove the
  correspondence principle for a very general class of quantum states,
  including highly correlated and incoherent ones. Moreover, our approach is
  different from the one in \cite{Knowles2009LimitingDI} and uses the more
  general concept of Wigner measures. This allows us to formulate an
  essentially complete mathematical picture of the correspondence principle
  for such  fundamental interaction.

  \medskip  A second element of novelty is that we use the
    semiclassical features of quantum dynamics \emph{to prove results on the
      classical dynamics} in a way that was not, up to our knowledge,
    explored before in QFT. In fact, we prove by infinite dimensional semiclassical analysis and measure theoretical techniques, global
    well-posedness for the Newton-Maxwell system, under weak
    regularity assumptions on the particles' charge distributions (see Theorem \ref{thm:gflow} and Theorem \ref{thm:main-1}). The
    regularity of the charge distribution affects the fixed point estimates
    for such nonlinear evolution system, and it plays a crucial role for the
    ODE part of the system: below a certain regularity threshold
    ($\dot{H}^{-s}\cap \dot{H}^{s+1}$, where $0\leq s < \frac{d}{2}$ is the Sobolev
    index of the electromagnetic field) the vector field of interaction is
    not Lipschitz, and thus it is not possible to close a standard fixed
    point argument as in \cite{falconi2014jmp}. While on the PDE part a lack
    of regularity may be compensated by finer (spacetime) estimates, for the
    ODE part there is less space of maneuver (see \S~\ref{stab.had} for a discussion).  Instead, we take a ``quantum
    detour'', exploiting the well-defined quantum dynamics to define the
    classical one; we find this approach very natural from a physical
    perspective (classical theories are only an approximation of quantum
    ones), and even if mathematically it could seem at first sight
    surprising, it reflects the regularizing character of the quantum
    dynamics (think, \emph{e.g.}, of Coulomb potentials). We also
    prove some new interesting regularity propagation estimates, inherited
    from the corresponding quantum ones (see Proposition
    \ref{prop.propag.NM}). Moreover, we
    conjecture that the assumptions
    \eqref{eq:hypo1}-(\hyperref[eq:hypo2]{$A_2^{(1)}$}) used in this article
    are thresholds for the Hadamard stability  of the Newton--Maxwell
    equation in the space $X^1$ defined in \eqref{spX1}.

\medskip
 One of the main tools in our investigation  is \emph{Wigner measures}. This mathematical object is  of course well-known in homogenization and in semiclassical analysis over finite dimensional spaces, see for instance \cite{MR2839299,MR1627111,MR1438151,MR1251718,MR1069518}. In recent years, these measures have been extended on the one hand to many-body theory in the mean filed scaling \cite{ammari2008ahp,MR3379490}, and on the other to interacting quantum fields  \cite{ammari2014jsp,ammari2017sima}. Our analysis here  expands the previous investigations to  systems of quantum particles coupled to quantized fields.  In addition, we provide here several improvements of those techniques summarized in Appendices~\ref{appx} and \ref{appy}.

\medskip

Let us conclude this general introduction by providing some references. In addition to
\cite{falconi2014jmp} there are quite some works devoted to the Newton--Maxwell
system of equations \cite{appel2001ap, MR1217117,bambusi1996lmp, bauer2013cpde,
  bauer2001ahp, imaykin2004ahp, imaykin2004mon, imaykin2011jmp,
  kiessling1999pla, komech2000cpde}. Whereas the Pauli--Fierz dynamics and nonrelativistic quantum electrodynamics have attracted a great interest from the mathematical physics community, see, \emph{e.g.},~\cite{MR700181,MR1996868,MR2361881,MR2863851,MR1682684,MR3201223,MR1878286,MR2077252,MR2502527,MR1856401,MR2498761}.
   Apart from the already mentioned paper
\cite{Knowles2009LimitingDI}, let us mention some other works where the classical
limit of the Pauli--Fierz Hamiltonian has been studied in a different scaling, and with a different physical interpretation
\cite{correggi2019jst, correggi2020arxiv, correggi2021jems, leopold2019ahp,
  leopold2020sima}.

\medskip

In the rest of this introduction we will define some notations,  present the
mathematical setting of the problem and conclude  with the statement of our
main results (Theorems~\ref{thm:main-1}
and~\ref{thm:main-2}). In Section~\ref{sec.NM}, we focus on the classical
Newton--Maxwell system and its uniqueness properties. Section~\ref{sec.PF} is
devoted to the quantum system, in particular we provide some  new uniform estimates
for the Pauli--Fierz Hamiltonian that are crucial to our analysis. In Section~\ref{sec.CH}, we study the classical limit $\hslash\to 0$ of the quantum dynamics and derive the characteristic equation satisfied by its Wigner measures.  Finally, Section~\ref{sec.proof} is devoted to the proof of
Theorems~\ref{thm:main-1} and~\ref{thm:main-2}.  The Appendix~\ref{appx} contains abstract results -- of independent  interest -- concerning:
\begin{itemize}
\item Equivalence between characteristic and Liouville equations.
\item Probabilistic representation of measure-valued solutions of Liouville equations.
\item Construction of generalized global flow for abstract initial value problems.
\end{itemize}
The measure theoretical techniques presented in Appendix~\ref{appx} are crucial in the proof of our main results. It
puts in a conclusive form several ideas that appeared previously in \cite{MR4134153,MR3721874,rouffort2018general}. While
Appendix~\ref{appy} establishes convergence results towards Wigner measures under optimal regularity assumptions on quantum states.

\bigskip
{\bf Acknowledgments:} M.F.\ acknowledges support from ``Istituto
Nazionale di Alta Matematica (INdAM)'' through the ``Progetto Giovani GNFM
2020: Emergent Features in Quantum Bosonic Theories and Semiclassical
Analysis'', and from the European Research Council (ERC) under the European
Union’s Horizon 2020 research and innovation programme (ERC CoG UniCoSM,
grant agreement n.724939). F.H.\ acknowledges support from  JSPS KAKENHI 20K20886 and JSPS KAKENHI 20H01808.

\bigskip
{\bf Notations and main results:}
In this paragraph  we first introduce the precise mathematical formulation of the quantum and classical models
of electrodynamics that we shall consider and then state our main results. Recall that classical
electrodynamics is governed by the Newton--Maxwell equation while quantum electrodynamics is described
by the Pauli--Fierz Hamiltonian. Our main contributions are summarized in Theorem \ref{thm:main-1} where
existence of unique global solutions of the Newton--Maxwell equation is established and in Theorem \ref{thm:main-2}
where the above solutions are derived as the classical limit $\hslash\to 0$ of
quantum Pauli--Fierz dynamics.

\subsubsection*{Newton--Maxwell equation:}
Consider $n$ extended classical particles in the configuration space $\R^d$,
with dimension $d\geq 3$, interacting with an electromagnetic field.  The
$i^{\text{th}}$ particle has an assigned mass $m_i>0$ and a charge
distribution\footnote{The total charge of the $i^{\text{th}}$ particle is
  given,  up to a dimensional constant, by $\int_{\R^d}^{}\varphi_i(x)
  \mathrm{d}x\in \mathbb{R}$. Therefore, our discussion includes as well the
  electrodynamics of charges that are globally neutral, but with a nontrivial
  charge distribution, like for example a water molecule.}  $\varphi_i:\R^d\to \R$; and
their dynamics is completely characterized by their momenta
$p_i=(p_i^\nu)_{\nu=1,\dots,d}\in\R^d$, and positions
$q_i=(q_i^\nu)_{\nu=1,\dots,d}\in\R^d$. On the other hand, it is convenient for our
purposes to describe the electromagnetic field in the Coulomb gauge by a
complex vector field
$$
\alpha =(\alpha_\lambda)_{\lambda=1,\dots,d-1}:\R^d\longrightarrow \C^{d-1}\,.
$$
Thus the Hamiltonian of the Newton--Maxwell system takes the form
\begin{equation}
\label{eq:ENM}
H(p,q,\alpha)=
\sum_{i=1}^n \frac{1}{2m_i} \big(p_i-A_i(q_i,\alpha)\big)^2 +V(q)+
\sum_{\lambda=1}^{d-1} \int_{\R^d}  \; \bar{\alpha}_\lambda(k) \,|k| \, \alpha_\lambda(k)\,\mathrm{d}k\,.
\end{equation}
Here $p=(p_1,\cdots,p_n)\in\R^{dn}$, $q=(q_1,\cdots,q_n)\in\R^{dn}$ and $V$ is the smeared
$d$-dimensional Coulomb potential
\begin{equation*}
V(q)=\sum_{1\leq i<j\leq n} w_{ij}(q_i-q_j) \,,
\end{equation*}
such that
\begin{equation}
\label{eq:w}
w_{ij}=g\, \varphi_i*\frac{1}{|\cdot|^{d-2}}*\varphi_j\,,
\end{equation}
where $g>0$ is a suitable (dimensional) constant. In this work we will require
minimal mathematical assumptions on the potential $V$ (that therefore could
be different from the electrostatic potential above), and on the Fourier
transform of the $\varphi_i$, see hypotheses \eqref{eq:hypo0}-\eqref{eq:hypo1} and
\eqref{eq:hypo2} below. The electromagnetic vector potential (smeared by the
charge distribution $\varphi_i$) is defined by
\begin{equation*}
A_i(q_i,\alpha)= \sum_{\lambda=1}^{d-1} \int_{\R^d} \, \frac{\epsilon_\lambda(k)}{\sqrt{2|k|}}
\, \bigl( \bar{\chi}_i(k)\,\alpha_\lambda(k) \,e^{2\pi\ci k\cdot q_i} +\chi_i(k)\,\bar{\alpha}_\lambda(k) \,e^{-2\pi\ci k\cdot q_i}\bigr) \,\mathrm{d}k\, \in\R^d\,,
\end{equation*}
such that
\begin{equation*}
\{\epsilon_\lambda(k),\tfrac{k}{\lvert k  \rvert_{}^{}}\}_{\lambda=1,\dots,d-1}
\end{equation*}
is an O.N.B of $\R^d$ and
\begin{equation}
\label{eq:chi}
\chi_i(\cdot)=\mathscr{F}^{-1} \bigl[\varphi_i\bigr](\cdot)= \int_{\R^d}  \, e^{2\pi\ci x \, \cdot \,(\,\cdot \,)} \, \varphi_i(x)\,\mathrm{d}x\,.
\end{equation}
Let us recall that the scalar potential is not a dynamical variable in the
Coulomb gauge. Moreover, we denote the components of the vector potential by
\begin{equation*}
A_i(q_i,\alpha)=\big(A^\nu_i(q_i,\alpha)\big)_{\nu=1,\dots,d}\;.
\end{equation*}
The Newton--Maxwell equations of motion read as:
\begin{subequations}
\label{eq:Newton-Maxwell}
\begin{align}
       \partial_t p_i &=\frac{1}{m_i} \sum_{\nu=1}^d \big(
       p_i^\nu-A^\nu_i(q_i,\alpha)\big) \, \nabla_{q_i}A^\nu_i(q_i,\alpha)- \nabla_{q_i} V(q)\;,         \label{eq:NMp} \\
        \partial_t q_i &=\frac{1}{m_i} \, \big(p_i-A_i(q_i,\alpha)\big)\;, \label{eq:NMq}\\
      \ci \partial_t \alpha_\lambda(k)&=|k| \alpha_\lambda(k)-\sum_{i=1}^n \frac{1}{m_i} \, \frac{\chi_i(k)}{\sqrt{2|k|}}
        \, \big(p_i-A_i(q_i,\alpha)\big) \cdot \epsilon_\lambda(k) \,e^{-2\pi\ci k\cdot q_i}\;.
        \label{eq:NMfa}
\end{align}
\end{subequations}
Consider respectively the following weighted $L^2$ spaces $\dot{\mathfrak{H}}^{\sigma}$ and
$\mathfrak{H}^\sigma$ with the norms
\begin{eqnarray}
\displaystyle
\|\alpha\|_{\dot{\mathfrak{H}}^\sigma}&= &\left(\sum_{\lambda=1}^{d-1} \int_{\R^d}  \, |k|^{2\sigma} \, |\alpha_\lambda(k)|^2\,\mathrm{d}k\right)^{1/2},\\
\|\alpha\|_{\mathfrak{H}^\sigma}&= & \left(\sum_{\lambda=1}^{d-1} \int_{\R^d}  \, (1+|k|^{2})^{\sigma} \, |\alpha_\lambda(k)|^2\,\mathrm{d}k\right)^{1/2},
\end{eqnarray}
and define the functional inner product spaces $\dot X^\sigma=\R^{dn}\times \R^{dn}\times
\dot{\mathfrak{H}}^\sigma$ and $X^\sigma=\R^{dn}\times \R^{dn}\times \mathfrak{H}^\sigma$ endowed respectively with the
norms\footnote{Let us remark that with the above notation, the $p_i$s and
  $q_i$s are $d$-dimensional vectors.}
\begin{equation}
\label{spX1}
\|u\|^2_{\dot X^\sigma}= \sum_{i=1}^{n} \bigl(|p_i|^2+|q_i|^2\bigr) +\|\alpha\|^2_{\dot{\mathfrak{H}}^\sigma}\,, \quad
\|u\|^2_{X^\sigma}=\sum_{i=1}^{n} \bigl(|p_i|^2+|q_i|^2\bigr)+\|\alpha\|^2_{\mathfrak{H}^\sigma}\,,
\end{equation}
with $u=(p,q,\alpha)$ in $\dot X^\sigma$ or $X^\sigma$. The spaces $X^\sigma$ with the above
inner product are always complete, and thus Hilbert, while $\dot X^{\sigma}$ are
complete only for $\sigma< \frac{d}{2}$. Through out the article, we will use indifferently the canonical identifications,
$$
\dot X^\sigma\equiv \C^{dn}\times \dot{\mathfrak{H}}^\sigma \qquad \text{ and } \qquad X^\sigma\equiv \C^{dn}\times \mathfrak{H}^\sigma\,,
$$
implemented by the complex structure
$$
z\equiv q+\ci p\,,
$$
where $\ci$ denotes the  imaginary unit. Note that $\mathfrak{H}^0=L^2 (\mathds{R}^d, \mathds{C}^{d-1}  )$ and $X^0=\dot{X}^0=\mathbb C^{dn}\times \mathfrak{H}^0$.

\subsubsection*{Pauli--Fierz Hamiltonian:}
The formal quantization of the above Newton--Maxwell system gives the
so-called Pauli--Fierz model of non-relativistic quantum electrodynamics that
we recall below.  The Hilbert space of the quantized particles-field system
is
\begin{equation*}
  \mathscr{H} =L^2 (\mathds{R}_x^{dn}, \mathds{C} )\otimes\Gamma_{\mathrm{s}}\bigl(L^2 (\mathds{R}_k^d, \mathds{C}^{d-1}  )\bigr) \,,
\end{equation*}
where $\Gamma_{\mathrm{s}}\bigl( \mathfrak{H}^0 \bigr)$ denotes the symmetric Fock space over $\mathfrak{H}^0$.
Let $\hat{p}=(\hat p_1,\dots,\hat p_n)$ and $\hat{q}=(\hat q_1,\dots,\hat q_n)$ denote the particles momenta  and position operators,
\begin{equation*}
  \hat{p}_i=-\ci\hslash\nabla_{x_i} \,,\quad \quad \hat{q}_i=x_i\; .
\end{equation*}
The $\hslash$-scaled annihilation-creation operators for the field are defined for any $f=\bigl(f_1,\dots,f_{d-1}\bigr)\in L^2 (\mathds{R}^d, \mathds{C}^{d-1} )$ as
\begin{equation}
\label{eq:an-cr1}
  \hat{a}(f)=\sum_{\lambda=1}^{d-1}\int_{\mathds{R}^d}  \,\bar{f}_\lambda(k)\,\hat{a}_\lambda(k) \,\mathrm{d}k
  \; ,\qquad \hat{a}^{*}(f)=\sum_{\lambda=1}^{d-1}\int_{\mathds{R}^d}  \,{f}_\lambda(k)\,\hat{a}^*_\lambda(k)\, \mathrm{d}k\; ,
\end{equation}
with  $\hat{a}_\lambda(k)$ and $\hat{a}_\lambda^{*}(k)$ are the annihilation-creation operator-valued distributions satisfying the $\hslash$-scaled canonical commutation
relations:
\begin{equation}
\label{eq:an-cr2}
  [\hat{a}_\lambda(k),\hat{a}_\lambda^{*}(k')]=\hslash\,\delta(k-k')\,\delta_{\lambda,\lambda'}\; .
\end{equation}
With these notations the field Hamiltonian is given by
\begin{equation*}
  \hat{H}_{\mathrm{f}}=\sum_{\lambda=1}^{d-1}\int_{\R^d}\;\lvert k  \rvert \,
  \hat{a}_\lambda^{*}(k)\hat{a}_\lambda(k)\,\mathrm{d}k\;,
\end{equation*}
 and the smeared quantum electromagnetic vector potential is defined as
\begin{equation}
\label{eq:vectpot}
  A_i(\hat q_i,\hat a)=\sum_{\lambda=1}^{d-1} \int_{\R^d} \, \frac{\epsilon_\lambda(k)}{\sqrt{2|k|}}
\, \bigl(\bar{\chi}_i(k) \,\hat{a}_\lambda(k) \,e^{2\pi\ci k\cdot \hat q_i} + \chi_i(k)\,\hat{a}^{*}_\lambda(k) \,e^{-2\pi\ci k\cdot \hat q_i}\bigr)\,\mathrm{d}k\;.
\end{equation}
Here, $V$ and $\chi_i$ are respectively the same potential and form factor as for the
Newton--Maxwell system.  In particular, $\chi_i$ is the Fourier transform of the
$i^{\text{th}}$ particle's charge distribution $\varphi_i$ and
${\bigl(\epsilon_\lambda(k)\bigr)}_{\lambda=1,\dotsc,d-1}$ are the above defined polarization vectors satisfying,
for almost all \ $k\in \mathbb{R}^d$ and for all $\lambda,\lambda'\in \{1,\dotsc, d-1\}$, the identities:
\begin{equation}
\label{eq:Coulgauge}
k\cdot \epsilon_\lambda(k)=0\,\qquad \text{and } \qquad \epsilon_\lambda(k)\cdot \epsilon_{\lambda'}(k)=\delta_{\lambda,\lambda'}\,.
\end{equation}
The Pauli--Fierz Hamiltonian of $n$-particles interacting with the quantized
electromagnetic field takes then the form
\begin{equation}
  \label{eq:PF}
  \hat{H}_{\hslash}=\sum_{i=1}^n {\frac{1}{2m_i}} {\bigl(\hat{p}_i-A_i(\hat q_i,\hat a)\bigr)}^2+V(\hat q)+\hat{H}_{\mathrm{f}} \;.
\end{equation}
It is known that the above Pauli--Fierz Hamiltonian is self-adjoint under convenient assumptions on the
form factors $\chi_i$ and the potential $V$ (see Section \ref{sec.PF} for more details).
For convenience, let us also define the non-interacting  Hamiltonian,
\begin{equation*}
  \hat{H}_{\hslash}^0=\sum_{i=1}^n {\frac{1}{2m_i}} \hat{p}_i^2+ \hat{H}_{\mathrm{f}}\; ,
\end{equation*}
the  $\hslash$-scaled number operator
\begin{equation}
\label{eq:Nop}
\hat{N}_\hslash=\sum_{\lambda=1}^{d-1}\int_{\R^d}\;
  \hat{a}_\lambda^{*}(k)\hat{a}_\lambda(k)\,\mathrm{d}k\;,
\end{equation}
and the weighted modulus momentum operator defined for $\sigma\geq 0$ as
\begin{equation}
\label{eq:dgk}
{\rm d}\Gamma(|k|^{2\sigma})=\sum_{\lambda=1}^{d-1}\int_{\R^d}\; |k|^{2\sigma}\;
  \hat{a}_\lambda^{*}(k)\hat{a}_\lambda(k)\,\mathrm{d}k\;.
\end{equation}

\subsubsection*{Wigner measures:}

In order to establish the Bohr's correspondence principle between quantum and
classical electrodynamics it is useful to use the very general concept of
Wigner (or semiclassical) measures, and not to restrict to coherent states
only. The semiclassical techniques of Wigner measures are a well known tool
that efficiently relates the quantum and classical states while taking the
effective  limit $\hslash\to0$. In particular, Wigner measures have been extensively
studied in finite dimensions (see, \emph{e.g.},~\cite{MR1251718} and
references therein contained). Moreover, such concept has been extended to
infinite dimensional phase spaces in \cite{ammari2008ahp} and applied to some
fundamental examples of many-body and quantum field theories in
\cite{ammari2014jsp,ammari2017sima,ammari2011jmpa,MR3379490}.

\medskip
In the sequel, we denote by
$\mathfrak{P}(X^\sigma)$ the set of all Borel probability measures over $X^\sigma$ and recall that
a density matrix  on $\mathscr H$ is a normalized non-negative trace class operator on
$\mathscr{H}$ called sometimes a quantum normal state.
\begin{defn}[Wigner measures]
\label{def:wigner}
      A Borel probability measure  $\mu\in\mathfrak{P}(X^0)$ is a Wigner measure of a family of density matrices
      $\bigl(\varrho _{\hslash}\bigr)_{
      \hslash \in (0,1)}$ on  the Hilbert space $\mathscr H$ if and only if there exists a countable subset
      $\mathscr{E }\subset  (0,1)$ with $0\in \overline{\mathscr{E }}$ (the closure of  ${\mathscr{E }}$)  such
    that for any $\xi=(p,q,\alpha) \in X^0$:
    \begin{equation}
\label{eq:wigner}
\lim_{\hslash \to 0,\hslash \in \mathscr{E  }} \Tr\big[ \varrho_{\hslash} \;\mathcal{W}\big(\pi q,-\pi p,\sqrt{2}\pi \alpha\big)\big] =
\int_{u\in X^0} e^{2\pi\ci  \Ree\langle \xi, u\rangle_{X^0}}\; \mathrm{d}\mu(u)\,.
\end{equation}
Here, $\mathcal{W}(\cdot)$ depends
on the parameter $\hslash$ and denotes the Weyl--Heisenberg operator defined according
to \eqref{eq:weyl-Heis}, \eqref{eq:weyl} and \eqref{eq:repWeyl}.
\end{defn}
\begin{notation}[Wigner measures set]
The above definition extends to any family of density matrices $\bigl(\varrho _{\hslash}\bigr)_{
      \hslash \in \mathscr{I}}$ indexed  by an arbitrary subset $\mathscr{I}$ of {the interval} $(0,1)$ such that $0\in\overline{\mathscr{I}}$ {(the closure of $\mathscr{I}$).} In such case, we simply denote
      the collection of all the Wigner (probability) measures of $\bigl(\varrho _{\hslash}\bigr)_{
      \hslash \in \mathscr{I}}$ by
      \begin{equation}
      \label{not.Mj}
      \mathscr{M}(\varrho _{\hslash}, \hslash\in \mathscr{I}) \,.
      \end{equation}
\end{notation}
Note that in general the set of Wigner measures $\mathscr{M}(\varrho _{\hslash}, \hslash \in
(0,1 ))$ is not empty if a mild assumption on the density matrices is assumed (see Proposition \ref{wig}). {Moreover, by extracting subsequences one can always choose a family of quantum normal states with a single Wigner measure, \emph{i.e.} there exists $\mathscr{I}$ such that \eqref{not.Mj} is a singleton.}
\medskip
Now, the convergence of quantum electrodynamics towards classical electrodynamics when
$\hslash\to 0$ can be reformulated as the formal commutative diagram:

\begin{equation*}
  \begin{tikzcd}[row sep=1.5cm, column sep=4.5cm]
    \varrho_{\hslash} \arrow[r, mapsto,"e^{-\frac{i}{\hslash}t \hat{H}_{\hslash}}\;(\,\cdot\,) \; e^{\frac{i}{\hslash}t \hat{H}_{\hslash}}"] \arrow[d, "\hslash\to 0" left]& \varrho_{\hslash}(t)\arrow[d,"\hslash\to 0" right]\\
    \mu_0 \arrow[r, mapsto, "(\Phi_t)\,_{\sharp}\, (\,\cdot \,)" swap]&\mu_t
  \end{tikzcd}
\end{equation*}

Indeed, consider a family of density matrices $(\varrho _{\hslash})_{\hslash\in(0,1)}$ at time
$t=0$; then its time evolution satisfies
$$
\varrho _{\hslash }(t)=e^{-
      \frac{i}{\hslash }t \,\hat H_\hslash}\,\varrho _{\hslash}\,e^{\frac{i}{\hslash } t \, \hat H_\hslash}\,.
$$
In order to link the quantum and classical dynamics as $\hslash\to 0$, it is enough
to show that if the family of states $(\varrho _{\hslash})_{\hslash\in(0,1)}$ admits a single
Wigner measure $\mu_0$, \emph{i.e.}
$$
 \mathscr{M}(\varrho _{\hslash}, \hslash\in (0,1))=\{\mu_0\} \,,
$$
then the family $\bigl(\varrho _{\hslash }(t)\bigr)_{\hslash\in(0,1)}$ admits a
unique Wigner measure at any time $t\in\R$ given by the  push-forward measure
\begin{equation}
  \label{int.eqt}
  \mu_t= (\Phi _t)\,_{\sharp}\, \mu_{0}\;,
\end{equation}
with $\Phi_t$ the flow that solves the classical Newton--Maxwell equation. In other words, the
right hand side of \eqref{int.eqt} denotes the image measure or pushforward measure
defined for all Borel sets $B$  of $X^\sigma$ as
$$
(\Phi_t)\,_\sharp\,\mu_0(B)=\mu_0\big((\Phi_t)^{-1}(B)\big)\,.
$$
It is worth noting that such Bohr's correspondence principle scheme was successfully proved to be
true for the Nelson model with Yukawa interaction, in a different scaling (in
which the particles are many but remain quantum), both with and without
ultraviolet cutoffs in \cite{ammari2014jsp} and \cite{ammari2017sima},
respectively.

\subsubsection*{Main results:}
The two ingredients to be chosen in the quantum and classical systems
\eqref{eq:ENM}-\eqref{eq:PF} are the potential $V$ and the form factors
$\chi_i$s. Although $V$ and $\chi_i$s are physically related through the charge
distributions $\varphi_{i}$s, we prefer to consider a more general context by
picking the following hypotheses, for all $i=1,\dotsc,n$:
\begin{equation}
\label{eq:hypo0}\tag{$A_0$}
        V\in \mathscr{C}_b^2(\R^{dn},\R)\,,
\end{equation}
\begin{equation}
\label{eq:hypo1}\tag{$A_1$}
        |\,\cdot \,|^{-1}\,\chi_i(\cdot )\in L^2 (\R^d, \C)\,,
\end{equation}
and
\begin{equation}
\label{eq:hypo2}\tag{$A_2^{(\sigma)}$}
        |\,\cdot \,|^{\frac 3 2-\sigma}\,\chi_i(\cdot )
        \in L^2 (\R^d, \C)\,.
\end{equation}
The first condition means that $V$ is a bounded $\mathscr{C}^2$ function with
its first and second derivatives bounded. In our analysis the relevant range
for the parameter $\sigma$ in the assumption \eqref{eq:hypo2} is the interval
$[\frac{1}{2},1]$.  Notice that for $\sigma,\sigma'\in[\frac{1}{2},1]$ such that $\sigma\geq  \sigma'$,
one has
$$
(A_2^{(\sigma')})\; \Rightarrow \; (A_2^{(\sigma)})\,.
$$
In particular, (\hyperref[eq:hypo2]{$A_2^{(1)}$}) holds true if we assume \eqref{eq:hypo2} for some  $\sigma\in[\frac 1 2, 1[$.
Moreover, by interpolation if $\chi_i$ satisfies \eqref{eq:hypo1} and \eqref{eq:hypo2}, then for all
$-1\leq \lambda\leq \frac 3 2-\sigma$,
  $$
  \lvert \,\cdot \, \rvert_{}^{\lambda}\,\chi_i(\cdot )\in L^2(\mathds{R}^d,\mathds{C})\,.
  $$

Our first result concerns the flow of the Newton--Maxwell system of equations,
and reads as follows.
\begin{thm}[Generalized global flow]
\label{thm:main-1}
Let $\sigma\in[\frac 1 2, 1]$ and assume that \eqref{eq:hypo0}, \eqref{eq:hypo1} and \eqref{eq:hypo2}
are satisfied. Then for any initial condition $u_0\in X^\sigma$ there exists a unique global strong solution
$u(\,\cdot\,)\in \mathscr{C}(\R,X^\sigma)\cap \mathscr{C}^1(\R, X^{\sigma-1})$ of the Newton--Maxwell equation \eqref{eq:Newton-Maxwell}. Moreover, the generalized global flow
\begin{equation}
\label{eq:Gflowmap}
\begin{aligned}
\Phi_t:  X^\sigma & \rightarrow & X^\sigma\\
u_0 &\mapsto & u(t)
\end{aligned}
\end{equation}
is Borel measurable.
\end{thm}
The above theorem is a consequence of  a uniqueness property for the Newton--Maxwell solutions, proved in Proposition \ref{s1:prop2},  combined with the existence of global solutions derived from the Pauli--Fierz quantum dynamics by means of the general theory in Appendix \ref{appx}.
 Its proof, provided in Section \ref{sec.proof},  is quite different from standard nonlinear fixed point arguments which require the charge distribution to be more regular. We also prove, by the same method, further quantitative propagation estimates satisfied by the Newton--Maxwell flow in Proposition \ref{prop.propag.NM}.

\medskip

Let now $(\varrho_\hslash)_{\hslash\in (0,1)}$ be a family of density matrices on the total
Hilbert space $\mathscr{H}$ of the particles-field quantum system
(\emph{i.e.} $\varrho_\hslash$ is a normalized non-negative trace class operator on
$\mathscr{H}$). The main assumptions on the family of states $(\varrho_\hslash)_{\hslash\in
  (0,1)}$ are:
\begin{align}
\label{eq:Ass0}\tag{$S_0^{(\delta)}$}
&\exists C_0>0, \;\forall \hslash\in (0,1), \qquad  \Tr\big[ \varrho_\hslash \; (\hat{H}_\hslash^0+1)^{\delta}\big] \leq C_0,\\
\label{eq:Ass1}\tag{$S_1^{(\delta)}$}
&\exists C_1>0, \;\forall \hslash\in (0,1), \qquad
\Tr\big[ \varrho_\hslash \; (\hat{N}_\hslash+\hat{q}^2+1)^\delta\big] \leq C_1\,,\\
\intertext{and}
\label{eq:Ass2}\tag{$S_2^{(\sigma)}$}
&\exists C_2>0, \;\forall \hslash\in (0,1), \qquad  \Tr\big[ \varrho_\hslash \; {\rm d}\Gamma(|k|^{2\sigma})\big] \leq C_2\,.
\end{align}

Under  such assumptions, we can prove the Bohr's correspondence principle.
\begin{thm}[Classical limit]
\label{thm:main-2}
Let $\delta\in(0,1]$, $\sigma\in[\frac 1 2, 1]$ and assume that \eqref{eq:hypo0}, \eqref{eq:hypo1} and \eqref{eq:hypo2}
 are satisfied. Let $(\varrho_\hslash)_{\hslash\in (0,1)}$  be a family of density matrices on $\mathscr{H}$ satisfying the assumptions
  \eqref{eq:Ass0}, \eqref{eq:Ass1} and \eqref{eq:Ass2}. Assume that for some probability measure $\mu_0\in\mathfrak{P}(X^0)$
\begin{equation}
\label{eq:wigsingle}
\mathcal{M}(\varrho_\hslash, \hslash\in (0,1))=\{\mu_0\}\,.
\end{equation}
Then for all times $t\in\R$,
$$
\mathcal{M}\big(e^{-\ci \frac t \hslash \hat H_\hslash}\,\varrho_\hslash \, e^{\ci \frac t \hslash \hat H_\hslash},
\; \hslash\in (0,1)\big)=\{\mu_t\}\,,
$$
with $\mu_t\in\mathfrak{P}(X^0)$ satisfying the following statements:
\begin{itemize}
\item [(i)] $\mu_t$ is a Borel probability measure on $X^\sigma$,
\item [(ii)] $\mu_t$ is the push-forward measure of $\mu_0$ by  the generalized flow $\Phi_t$  of
the Newton--Maxwell equation constructed in Theorem \ref{thm:main-1},
\emph{i.e.} for all Borel subsets $B$ of $X^\sigma$, and all $t\in\R$,
\begin{equation}
\label{eq:immeas}
\mu_t(B)=(\Phi_t)\,_\sharp\,\mu_0(B)=\mu_0\big((\Phi_t)^{-1}(B)\big)\,.
\end{equation}
\end{itemize}
Moreover, in the general case where \eqref{eq:wigsingle} is not assumed the statement below holds true for all times,
    \begin{equation}
    \label{eq:genres}
       \mathcal{M}\big(e^{-\ci \frac t \hslash \hat H_\hslash}\,\varrho_\hslash \, e^{\ci
      \frac t \hslash \hat H_\hslash}, \; \hslash\in (0,1)\big)=\big\{(\Phi_t)\,_{\sharp}\,\mu_0, \; \mu_0\in
    \mathcal{M}(\varrho_\hslash, \hslash\in (0,1)) \big\}\,.
    \end{equation}
\end{thm}

\begin{remark}
  \label{rem:2}
  $\phantom{i}$
  \begin{enumerate}
  \item The case $\sigma=1$ corresponds to the weakest requirement for the form
    factors $\chi_i$ while for $\sigma=\frac{1}{2}$ the assumptions on the quantum states
    \eqref{eq:Ass2} are the weakest.
  \item The statement \eqref{eq:genres} can be complemented by the fact that given $\mu_0\in
    \mathcal{M}(\varrho_\hslash, \hslash\in (0,1))$ and the associated family $(\varrho_\hslash)_{\hslash\in
      \mathscr{E}}$ that converges to it, then $\bigl(\varrho_\hslash(t)\bigr)_{\hslash\in
      \mathscr{E}}$ converges to $(\Phi_t)\,_{\sharp}\,\mu_0$ for all $t\in \R$ (convergence here is according to \eqref{eq:wigner} and Definition \ref{def:wigner}).
  \item Most of the analysis in this article  is devoted to the case $\delta=1$ and \eqref{eq:wigsingle}. The proof
  of  Theorem \ref{thm:main-2} in the general case is given only at the end of Section \ref{sec.proof} (\S-\ref{proof:impro} Proof of Theorem \ref{thm:main-2} (bis)).
  \item Examples of density matrices satisfying the above hypothesis \eqref{eq:Ass0}--\eqref{eq:Ass2} include  coherent states, product states or superpositions (see Section \ref{sec.CH}  and \cite{ammari2008ahp}).
        \end{enumerate}
\end{remark}

\section{Newton--Maxwell equation}
\label{sec.NM}

In this section, we prove a uniqueness result for solutions of the
Newton--Maxwell equation \eqref{eq:Newton-Maxwell} on the spaces $X^\sigma$ with
$\sigma\in[\frac{1}{2},1]$. Note that from now on we will refer only to $V$ and $\chi_i$,
$i=1,\dotsc,n$, satisfying the hypothesis \eqref{eq:hypo0}, \eqref{eq:hypo1} and
\eqref{eq:hypo2}, and no more to the charge distributions $\varphi_i$s. In
particular, the following lemma justifies the assumption \eqref{eq:hypo0} on
the potential $V$ by showing it to be true for the electrostatic potentials
between the charges.

\begin{lemma}
  Let $\sigma\in[\frac{1}{2},1]$ and assume that \eqref{eq:hypo1} and \eqref{eq:hypo2}
  are satisfied. Then the potentials $w_{ij}$, $i<j=1,\dotsc,n$, defined by
  \eqref{eq:w}, belong to $\mathscr{C}_b^2(\R^d)$.
\end{lemma}
\begin{proof}
Firstly, the expression \eqref{eq:w} is well defined as a tempered distribution. Hence,
$$
w_{ij}= \mathscr{F} \bigl[\mathscr{F}^{-1} [w_{ij}]\bigr]= g \mathscr{F} \biggl[ \frac{\mathscr{F}^{-1}[\varphi_i]\mathscr{F}^{-1}[\varphi_j]}{|\cdot|^2}\biggr]= g \mathscr{F}\biggl[\frac{\chi_i\chi_j}{|\cdot|^2}\biggr],
$$
where $\mathscr{F}$ denotes here the Fourier transform on
$\mathscr{S}'(\R^d)$, and $\mathscr{F}^{-1}$ its inverse (see also \eqref{eq:chi}). Now, the assumption
\eqref{eq:hypo1} yields that $w_{ij}\in \mathscr{C}_\infty(\R^d)$ as the Fourier
transform of an $L^1$ function. Repeating the same argument for the
derivatives one proves the result.
\end{proof}

Let us now prove some preliminary lemmas, concerning the
electromagnetic vector potential and its derivative.

\begin{lemma}
\label{s1:lem1}
Let $\sigma\in[\frac{1}{2},1]$ and assume that \eqref{eq:hypo1} and \eqref{eq:hypo2}
are satisfied. Then
\begin{itemize}
\item [\textrm{(i)}] $\forall (q,\alpha)\in \R^{dn}\times L^2$, $\forall i\in \{1,\dotsc,n\}$,
 $$|A_i(q,\alpha)|\leq  \sqrt{2(d-1)}\,\big\|\frac{\chi_i}{\sqrt{|\cdot|}}\big\|_{L^2} \, \|\alpha\|_{L^2},$$

\item [\textrm{(ii)}] $\forall (q,\alpha)\in \R^{dn}\times\dot{\mathfrak{H}}^{1/2}$, $\forall i\in \{1,\dotsc,n\}$,
$$|A_i(q,\alpha)|\leq  \sqrt{2(d-1)}\,\big\|\frac{\chi_i}{|\cdot|}\big\|_{L^2} \, \|\alpha\|_{\dot{\mathfrak{H}}^{1/2}}.$$
\end{itemize}
\end{lemma}
\begin{proof}
  By the Cauchy-Schwarz inequality (both w.r.t.\ $\lambda$ and $k$):
\begin{eqnarray*}
|A_i(q,\alpha)| &\leq &\left| \sum_{\lambda=1}^{d-1} \int_{\R^d} \, \frac{1}{\sqrt{2|k|}}
\, \big(\bar{\chi}_i(k)\epsilon_\lambda(k) \,\alpha_\lambda(k) \,e^{2\pi\ci k\cdot q} +\mathrm{h.c.}\big)\,\mathrm{d}k \right|\\
&\leq  & 2\biggl(\sum_{\lambda=1}^{d-1}  \big\|\frac{\chi_i}{\sqrt{2|\cdot|}}\big\|^2_{L^2}\biggr)^{\frac{1}{2}} \,
\|\alpha\|_{L^2} \quad \text{ or } \quad 2\biggl(\sum_{\lambda=1}^{d-1}  \big\|\frac{\chi_{i}}{\sqrt{2}|\cdot|}\big\|^2_{L^2}\biggr)^{\frac{1}{2}} \,
 \|\alpha\|_{\dot{\mathfrak{H}}^{1/2}}\,.
\end{eqnarray*}
\end{proof}

\begin{lemma}
\label{s1:lem2}
Let $\sigma\in[\frac{1}{2},1]$ and assume that \eqref{eq:hypo1} and \eqref{eq:hypo2}
are satisfied. Then for all $\nu=1,\dots,d$:
\begin{itemize}
\item [\textrm{(i)}]  $\forall (q,\alpha)\in \R^{dn}\times L^2$, $\forall i\in \{1,\dotsc,n\}$,
$$|\nabla_q A^\nu_i(q,\alpha)|\leq\sqrt{2(d-1)}\, \big\||\cdot|^{1/2} \chi_i\big\|_{L^2} \, \|\alpha\|_{L^2}\,,$$

\item [\textrm{(ii)}]  $\forall(q,\alpha)\in\R^{dn}\times\dot{\mathfrak{H}}^{1/2}$, $\forall i\in \{1,\dotsc,n\}$,
$$|\nabla_q A^\nu_i(q,\alpha)|\leq  \sqrt{2(d-1)}\,\big\|\chi_i\big\|_{L^2} \, \|\alpha\|_{\dot{\mathfrak{H}}^{1/2}}\,.$$
\end{itemize}
\end{lemma}
\begin{proof}
By the Cauchy-Schwarz inequality:
\begin{eqnarray*}
|\nabla_q A^\nu_i(q,\alpha)|&\leq& \left|  \sum_{\lambda=1}^{d-1} \int_{\R^d}   \, \frac{k}{\sqrt{2|k|}} \, \big(\bar{\chi}_i(k)\epsilon^\nu_\lambda(k) \, \alpha_\lambda(k) e^{2\pi\ci k\cdot q}+\mathrm{h.c.}\big)\,\mathrm{d}k\right|\\
&\leq  & 2\biggl(\sum_{\lambda=1}^{d-1} \tfrac{1}{2}\big\||\cdot|^{1/2} \chi_i\big\|^2_{L^2}\biggr)^{\frac{1}{2}}  \|\alpha\|_{L^2}
\quad \text{ or } \quad 2\biggl(\sum_{\lambda=1}^{d-1} \tfrac{1}{2}\big\|\chi_i\big\|^2_{L^2}\biggr)^{\frac{1}{2}} \, \|\alpha\|_{\dot{\mathfrak{H}}^{1/2}}.
\end{eqnarray*}
\end{proof}

\begin{lemma}
\label{s1:lem3}
Let $\sigma\in[\frac{1}{2},1]$ and assume that \eqref{eq:hypo1} and \eqref{eq:hypo2}
are satisfied. Then for all $q_1,q_2\in\R^{dn}$ and $\nu=1,\dots,d$:
\begin{itemize}
\item [\textrm{(i)}]  $\forall \alpha_1,\alpha_2\in L^2$, $\forall i\in \{1,\dotsc,n\}$,
$$
\big|A_i(q_1,\alpha_1)-A_i(q_2,\alpha_2)\big| \lesssim  \big\|\frac{\chi_i}{\sqrt{|\cdot|}}\big\|_{L^2}
\,\|\alpha_1-\alpha_2\|_{L^2}+ |q_1-q_2| \, \big\|\sqrt{|\cdot|}\chi_i\big\|_{L^2} \,
\|\alpha_2\|_{L^2}\,.
$$

\item [\textrm{(ii)}] $\forall\alpha_1,\alpha_2\in \mathfrak{H}^{\sigma}$, $\forall i\in \{1,\dotsc,n\}$,
$$
\big|\nabla_{q_1}A^\nu_i(q_1,\alpha_1)-\nabla_{q_2}A^\nu_i (q_2,\alpha_2)\big| \lesssim
\big\|\sqrt{|\cdot|}\chi_i\big\|_{L^2}
\,\|\alpha_1-\alpha_2\|_{L^2}+ |q_1-q_2| \;\; \big\| \,|\cdot|^{\frac 3 2-\sigma}\,\chi_i\big\|_{L^2} \,
\|\alpha_2\|_{\dot{\mathfrak{H}}^{\sigma}}\,.
$$
\end{itemize}
\end{lemma}
\begin{proof} Estimating the left hand side of \textit{(i)}, the Cauchy-Schwarz inequality yields:
\begin{eqnarray*}
&&\hspace{-.3in}\big|A_i(q_1,\alpha_1)-A_i(q_2,\alpha_2)\big| \\
&&\hspace{.8in}\lesssim
\sum_{\lambda=1}^{d-1} \int_{\R^d}  \, \left| \frac{\chi_i(k)}{\sqrt{|k|}}\right| \, \left(
\left| \alpha_{1,\lambda}(k)-\alpha_{2,\lambda}(k)  \right| +
\left| \big( e^{2\pi\ci k\cdot q_1}-e^{2\pi\ci k\cdot q_2}\big) \alpha_{2,\lambda}(k)\right|\right) \,\mathrm{d}k\\
&&\hspace{.8in}\lesssim
\sum_{\lambda=1}^{d-1} \int_{\R^d}  \, \left| \frac{\chi_i(k)}{\sqrt{|k|}}\right| \, \left(
\left|\alpha_{1,\lambda}(k)-\alpha_{2,\lambda}(k)  \right| +
\, |k| \, |q_1-q_2|\, |\alpha_{2,\lambda}(k)|\right)\,\mathrm{d}k\\
&&\hspace{.8in}\lesssim   \big\|\frac{\chi_i}{\sqrt{|\cdot|}}\big\|_{L^2}
\,\|\alpha_1-\alpha_2\|_{L^2}+ |q_1-q_2| \, \big\|\sqrt{|\cdot|}\chi_i\big\|_{L^2} \|\alpha_2\|_{L^2}.
\end{eqnarray*}
A similar argument gives
\begin{eqnarray*}
&&\hspace{-.5in}\big|\nabla_{q_1}A^\nu_i(q_1,\alpha_1)-\nabla_{q_2}A^\nu_i(q_2,\alpha_2)\big| \\
&&\hspace{1in}\lesssim
\sum_{\lambda=1}^{d-1} \int_{\R^d}  \,  \sqrt{|k|} \left|\chi_i(k)\right|  \left(
\left| \alpha_{1,\lambda}(k)-\alpha_{2,\lambda}(k)  \right| +
 |k| \, |q_1-q_2|\, \left|\alpha_{2,\lambda}(k)\right|\right)\,\mathrm{d}k\\
 &&\hspace{1in}\lesssim  \big\|\sqrt{|\cdot|}\chi_i\big\|_{L^2}
\,\|\alpha_1-\alpha_2\|_{L^2}+ |q_1-q_2| \;\; \big\| \,|\cdot|^{\frac 3 2 -\sigma}\,\chi_{i}\big\|_{L^2} \,
\|\alpha_2\|_{\dot{\mathfrak{H}}^{\sigma}} \,.
\end{eqnarray*}
\end{proof}

The Newton--Maxwell system \eqref{eq:Newton-Maxwell} is a semi-linear PDE of the form
\begin{equation}
\label{eq:NMc}
\left\{
\begin{aligned}
&\frac{d}{dt} u(t)=\mathcal{L}( u(t))+ F(u(t))&,\\
&u(0)  = u_0 \in X^\sigma,& \\
\end{aligned}
\right.
\end{equation}
where $t\to u(t)=(p(t),q(t),\alpha(t))$ is a solution,  $\mathcal{L}(u)=(0,0,-i|k|\alpha)$ is a linear operator defined
for all $u=(p,q,\alpha)\in \dot X^{1}$, and $F$ is the ``nonlinearity" given by
\begin{subequations}
\label{eq:F}
\begin{align}
       \bigg(F(u)\bigg)_{p_i}&=\frac{1}{m_i} \sum_{\nu=1}^d \big(
       p_i^\nu-A^\nu_i(q_i,\alpha)\big) \, \nabla_{q_i}A^\nu_i(q_i,\alpha)- \nabla_{q_i} V(q),         \label{eq:Fp} \\
        \bigg(F(u)\bigg)_{q_i}&=\frac{1}{m_i} \, \big(p_i-A_i(q_i,\alpha)\big), \label{eq:Fq}\\
      \bigg(F(u)\bigg)_{\alpha_\lambda}(k)&=\ci\sum_{i=1}^n \frac{1}{m_i} \, \frac{\chi_i(k)}{\sqrt{2|k|}}
        \, \big(p_i-A_i(q_i,\alpha)\big) \cdot \epsilon_\lambda(k) \,e^{-2\pi\ci k\cdot q_i}.
        \label{eq:Ffa}
\end{align}
\end{subequations}
Here the notations $\left( \,\cdot\, \right)_{p_i}$, $\left( \,\cdot\, \right)_{q_i}$  and  $\left( \,\cdot\, \right)_{\alpha_\lambda}$ refer
to the components $p_i$, $q_i$ and $\alpha_\lambda$ respectively. We note also that $F$ contains a linear part in \eqref{eq:Fq}.

\medskip
For convenience, given a (finite) collection of functions
$\{f_i,\dotsc,f_n\}$, all belonging to the same Banach space $X$, we adopt the
following notation:
\begin{equation*}
  \lVert f  \rVert_X^{}= \max_{i\in \{1,\dotsc,n\}}\lVert f_i  \rVert_X^{}\; .
\end{equation*}
The vector field $F$ representing the Newton--Maxwell nonlinearity has the
following property.

\begin{prop}
\label{s1:prop1}
Let $\sigma\in[\frac{1}{2},1]$ and assume that \eqref{eq:hypo0}, \eqref{eq:hypo1}
and \eqref{eq:hypo2} are satisfied. Then the nonlinearity $F:X^{\sigma}\to X^{\sigma}$ is
a continuous, and bounded on bounded sets, vector field.
\end{prop}
\begin{proof}
 One checks, thanks to Lemmas~\ref{s1:lem1} and~\ref{s1:lem2} that $F$ maps
 $X^{1/2}$ into $X^1$ and is bounded on bounded sets, with the
 estimates below holding true:
\begin{subequations}
\label{eq:Fbd}
\begin{align}
\bigg(F(u)\bigg)_{p_i}&\lesssim (|p|+\|\alpha\|_{L^2}) \,\big\|\frac{\chi_i}{\sqrt{|\cdot|}}\big\|_{L^2} \;
\big\|\sqrt{|\cdot|} \chi_i\big\|_{L^2} \; \|\alpha\|_{L^2}+\sup_q|\nabla_q V(q)|\\
\bigg(F(u)\bigg)_{q_i}&\lesssim  |p|+\big\|\frac{\chi_i}{\sqrt{|\cdot|}}\big\|_{L^2} \; \|\alpha\|_{L^2}\\
\left\|\bigg(F(u)\bigg)_{\alpha_\lambda}\right\|_{\dot{\mathfrak{H}}^1} &\lesssim
\big\|\sqrt{|\cdot|} \chi\big\|_{L^2} \; \bigg(|p|+ \big\|\frac{\chi}{\sqrt{|\cdot|}}\big\|_{L^2} \, \|\alpha\|_{L^2}\bigg)\\
\left\|\bigg(F(u)\bigg)_{\alpha_\lambda}\right\|_{L^2} &\lesssim
\big\|\frac{\chi}{\sqrt{|\cdot|}}\big\|_{L^2} \; \bigg(|p|+ \big\|\frac{\chi}{\sqrt{|\cdot|}}\big\|_{L^2} \, \|\alpha\|_{L^2}\bigg).
\end{align}
\end{subequations}
In order to prove continuity of $F$ for  $\sigma\in[\frac{1}{2},1]$, it is enough
to show the continuity of the following three maps over $\R^{dn}\times \mathfrak{H}^{\sigma}$, for any $i=1,\dotsc,n$:
\begin{itemize}
\item [(a)] $ (q,\alpha)\to A^\nu_i(q,\alpha)\in\R$
\item [(b)] $ (q,\alpha)\to \nabla_q A^\nu_i(q,\alpha)\in\R$
\item [(c)] $ (q,\alpha)\to \frac{\chi_i(k)}{\sqrt{|k|}} e^{-2\pi\ci k\cdot q}\in \mathfrak{H}^\sigma$.
\end{itemize}
  Indeed, continuity of the maps (a) and (b) is a consequence of Lemma \ref{s1:lem3} while the map (c) is continuous thanks to dominated convergence and to the bounds,
\begin{align}
\left\|\frac{\chi_i(k)}{\sqrt{|k|}} e^{-2\pi\ci k\cdot q_1}-\frac{\chi_i(k)}{\sqrt{|k|}} e^{-2\pi\ci k\cdot q_2}\right\|_{L^2}^2 &\leq
\int_{\R^d} \;  \left(\frac{|\chi_i(k)|}{\sqrt{|k|}}  \,\bigg| e^{-2\pi\ci k\cdot q_1}-e^{-2\pi\ci k\cdot q_2}\bigg|\right)^2\,\mathrm{d}k\,,
\intertext{and}
\left\|\frac{\chi_i(k)}{\sqrt{|k|}} e^{-2\pi\ci k\cdot q_1}-\frac{\chi_i(k)}{\sqrt{|k|}} e^{-2\pi\ci k\cdot q_2}\right\|_{\dot{\mathfrak{H}}^\sigma}^2 &\leq
\int_{\R^d} \;  \left(|k|^{\sigma-1/2} |\chi_i(k)| \,\bigg|
e^{-2\pi\ci k\cdot q_1}-e^{-2\pi\ci k\cdot q_2}\bigg|\right)^2\,\mathrm{d}k\,,
\end{align}
since $\sigma-\frac 1 2 \leq \frac 3 2 -\sigma$ for $\sigma\in [\frac{1}{2},1]$.
\end{proof}

Let $I$ be an open  interval containing  the origin. We are interested in  \emph{strong} solutions $u(\,\cdot\,)$ of
the Newton--Maxwell equation \eqref{eq:Newton-Maxwell} or equivalently   \eqref{eq:NMc} such that
$$
u(\,\cdot\,)\in \mathscr{C}(I, X^\sigma)\cap\mathscr{C}^1(I, X^{\sigma-1}),
$$
with $\sigma\in[\frac{1}{2},1]$ and \eqref{eq:NMc} is satisfied for all $t\in I$. In particular, these solutions verify the  following Duhamel formula  for all $t\in I$,
\begin{equation}
\label{eq:Duh}
u(t)=\mathcal{D}_0(t) u(0)+\int_0^t \,\mathcal{D}_0(t-s) F(u(s)) \, \mathrm{d}s\,
\end{equation}
where  $\mathcal{D}_0(\cdot)$ is the free field flow defined as
\begin{equation}
\label{eq:D0}
\mathcal{D}_0(t)(p,q,\alpha)=(p,q,e^{-\ci t |\cdot|} \alpha).
\end{equation}

\begin{prop}
\label{s1:prop2}
Let $\sigma\in[\frac{1}{2},1]$ and assume that \eqref{eq:hypo0}, \eqref{eq:hypo1}
and \eqref{eq:hypo2} are satisfied. Consider  an open interval $I$ containing  the origin and let
$u_1,u_2\in \mathscr{C}(I,X^\sigma)$  two strong  solutions of the Newton--Maxwell equation \eqref{eq:NMc} such that
$$
u_1(0)=u_2(0).
$$
Then $u_1(t)=u_2(t)$ for all $t\in I$.
\end{prop}
\begin{proof}
For any bounded interval $J\subset I$, note that
$$
M=\max_{i=1,2}\left(\sup_{J}\|u_i(t)\|_{X^\sigma}\right)<\infty.
$$
Using the Duhamel formula \eqref{eq:Duh}, one finds that for all $t\geq 0$,
$$
\|u_1(t)-u_2(t)\|_{X^0}\leq \int^t_0  \|F(u_1(s))-F(u_1(s))\|_{X^0} \; \mathrm{d}s.
$$
The following estimate holds true (we drop the $s$-dependence in the right hand side):
\begin{eqnarray*}
&\left| \bigg( F(u_1(s))-F(u_2(s))\bigg)_{p_i}\right|& \\ &\lesssim & \hspace{-.5in}\sum_{\nu=1}^{d}\left(\left|p^\nu_{1,i}-p^\nu_{2,i}\right|
+\left|A^\nu_i(q_{1,i},\alpha_1)-A^\nu_i(q_{2,i},\alpha_2)\right|\right)  \,\left|\nabla_{q_{2,i}}A^\nu_i(q_{2,i},\alpha_2)\right| \\
&& \hspace{-.5in} + \left(|p^\nu_{2,i}|+|A^\nu_i(q_{2,i},\alpha_2)|\right) \,
\left|\nabla_{q_{1,i}}A^\nu_i(q_{1,i},\alpha_1)-\nabla_{q_{1,i}}A^\nu_i(q_{1,i},\alpha_1)\right| \\
&& \hspace{-.5in} +\left|\nabla_{q_{1,i}} V(q_1)-\nabla_{q_{2,i}} V(q_2)\right|\\
&\overset{\text{Lemma }\ref{s1:lem1}-\ref{s1:lem3} }{\lesssim}& |p_1(s)-p_2(s)|+|q_1(s)-q_2(s)|
+\|\alpha_1(s)-\alpha_2(s)\|_{L^2}\,.
\end{eqnarray*}
Similarly,
\begin{eqnarray*}
\left| \bigg( F(u_1(s))-F(u_2(s))\bigg)_{q_i}\right|& \lesssim & \sum_{\nu=1}^{d} \left|p^\nu_{1,i}-p^\nu_{2,i}\right|
+\left|A^\nu_i(q_{1,i},\alpha_1)-A^\nu_i(q_{2,i},\alpha_2)\right| \\
&\overset{\text{Lemma }\ref{s1:lem1}-\ref{s1:lem3} }{\lesssim}& |p_1(s)-p_2(s)|+|q_1(s)-q_2(s)|+\|\alpha_1(s)-\alpha_2(s)\|_{L^2}
\end{eqnarray*}
and
\begin{eqnarray*}
\left\| \bigg( F(u_1(s))-F(u_2(s))\bigg)_{\alpha_\lambda}\right\|_{L^2}& \lesssim & \sum_{i=1}^{n} \left|p_{1,i}-p_{2,i}\right|
+\left|A_i(q_{1,i},\alpha_1)-A_i(q_{2,i},\alpha_2)\right| \\
&& + |q_{1,i}-q_{2,i}| \\
&\overset{\text{Lemma }\ref{s1:lem1}-\ref{s1:lem3} }{\lesssim}&
|p_1(s)-p_2(s)|+|q_1(s)-q_2(s)| \\ && +\|\alpha_1(s)-\alpha_2(s)\|_{L^2}\,.
\end{eqnarray*}
Hence,
\begin{eqnarray*}
  \|u_1(t)-u_2(t)\|_{X^0}&\lesssim& \int^t_0
  |p_1(s)-p_2(s)|+|q_1(s)-q_2(s)|+\|\alpha_1(s)-\alpha_2(s)\|_{L^2}\; \mathrm{d}s\\
  &\leq& C\; \int^t_0
  \|u_1(s)-u_2(s)\|_{X^0}\;\mathrm{d}s\,,
\end{eqnarray*}
and so Gr\"onwall's inequality yields the result.
\end{proof}

\paragraph{\emph{Local well-posedness:}}
 \label{stab.had} We indicate here the main difficulty in proving the local well-posedness in $X^1$ of the Newton--Maxwell equation by means of a standard fixed point argument.
 Take $(p_0,q_0,\alpha_0)\in X^1$ an initial data, $T>0$ a given time and $R$ a positive real number. Consider the natural complete metric space
 $$
 E=\{u(\cdot)\in \mathscr{C}([0,T],X^1) : \sup_{[0,T]} \|u(s)-(p_0,q_0,e^{-\ci s|\cdot|} \alpha_0)\|_{X^1}\leq R\}\,.
 $$
So, the question is to prove the existence of  a unique fixed point for the mapping
\begin{align}
\tau: E &\longrightarrow E \\
u(\cdot)&\longrightarrow \tau(u)(\cdot)\,,
\end{align}
defined by
$$
\tau(u)(t)=(p_0,q_0,e^{-\ci t|\cdot|} \alpha_0)+\int_0^t \big(F(u(s))_p,F(u(s))_q, e^{-\ci (t-s)|\cdot|} F(u(s))_\alpha\big) \,ds\,,
$$
where $F$ is the vector field given in \eqref{eq:F}. Firstly, one has to choose $T>0$ small enough such that $\tau$ maps $E$ into itself. Secondly,
to prove the contraction property for $\tau$ it is sufficient to establish the following inequality for all $u_1,u_2\in E$,
\begin{equation}
\label{eq.pofix}
\int_0^T \| F(u_1(s))-F(u_2(s))\|_{X^1} \,ds \leq T \,C(R)\, \sup_{[0,T]} \|u_1(s)-u_2(s)\|_{X^1} \,,
\end{equation}
with $C(R)$ is a constant depending essentially on $R$ and $T$ has to be chosen small enough such that $T C(R)<1$. Now, the estimate \eqref{eq.pofix} requires essentially the bounds
$$
\|\frac{\chi_j(\cdot)}{\sqrt{2|\cdot|}} \, \big(e^{-2\pi\ci k\cdot q_{j,1}(s)}- e^{-2\pi\ci k\cdot q_{j,2}(s)}\big) \|_{\dot H^1} \lesssim \sup_{[0,T]} |q_{j,1}(s)-q_{j,2}(s)|\,
$$
to be true while we only have $|\cdot|^{1/2}\chi_j(\cdot)\in L^2$. Hence, the argument can not be closed in this way and probably more refined analysis is needed.  In particular taking advantage of the time oscillation on front of the field component could be helpful; but even in this case it is unclear to us how to close the estimates.

\smallskip
Finally, examining the vector field $F$ in \eqref{eq:F}, one notices that a small variation on the positions of the particles leads to a large variation of the fields in the $L^2$-norm whenever we assume only $|\cdot|^{1/2-\varepsilon}\chi_j(\cdot)\in L^2$ for an arbitrary small $\varepsilon>0$. In some sense, this indicates that the assumptions \eqref{eq:hypo1}-(\hyperref[eq:hypo2]{$A_2^{(1)}$}) are  thresholds for the Hadamard stability of the Newton--Maxwell equation.
We believe that these questions of fixed point and stability are interesting in themselves and relevant for classical electrodynamics and we hope that our work will stimulate some interest on them, see \cite{PhysRevD.15.2850} and \cite{MR2303098}.

\section{Pauli--Fierz Hamiltonians}
\label{sec.PF}

In this section, we review some features of the Pauli--Fierz Hamiltonian and prove uniform estimates that are  useful in the analysis of the classical limit.  Recall that such  Hamiltonian   describes a system of many quantum $d$-dimensional extended charges ($d\geq 3$) interacting with the quantized electromagnetic field in the Coulomb gauge, see \emph{e.g}.~\cite{spohn2004dcp}.  Moreover, it
 is known that the Pauli--Fierz operator $\hat H_\hslash$ in \eqref{eq:PF} is essentially self-adjoint under suitable regularity assumptions on the charge  distributions. For various  approaches to the study of self-adjointness for Pauli--Fierz-type Hamiltonians, see for instance  \cite{falconi2015mpag,hasler2008rmp,hiroshima2000cmp,hiroshima2002ahp,matte2017mpag}.

\begin{thm}[\cite{hiroshima2002ahp}]
  \label{prop:2}
  Assume \eqref{eq:hypo0}, \eqref{eq:hypo1} and
  \emph{(}\hyperref[eq:hypo2]{$A_2^{(1)}$}\emph{)}. Then
  $\hat{H}_{\hslash}$ is self-adjoint on $D(\hat{H}_{\hslash}^0)=D(\hat{p}^2)\cap D(\hat{H}_{\mathrm{f}})$.
\end{thm}

 We recall some useful inequalities on $\hat a^\sharp$ and $\hat H_{\rm f}$. Let $\psi\in D(\hat H_{\rm f}^{m/2})$.  Suppose that $\|f_j/\sqrt{|\cdot|}\|_2<\infty$ for $j=1,\ldots, m$. Then $\psi\in D(\prod_{j=1}^m   \hat a(f_j))\cap D(\prod_{j=1}^m   \hat a^\ast(f_j))$ and there exists
a constant $c>0$ depending only on $\|f_j/\sqrt{|\cdot|}\|_2$ such that
\begin{equation}
\label{eq.a.est}
\|\prod_{j=1}^m   \hat a(f_j)\psi\|\leq c
\|\hat H_{\rm f}^{m/2}\psi\|\,,
\end{equation}
\begin{equation}
\label{eq.ast.est}
\|\prod_{j=1}^m  \hat a^\ast (f_j)\psi\|\leq c
\|(\hat H_{\rm f}+1)^{m/2}\psi\|\,.
\end{equation}

 \begin{lemma}
   \label{s3.lemE}
 Assume \eqref{eq:hypo0}, \eqref{eq:hypo1} and \emph{(}\hyperref[eq:hypo2]{$A_2^{(1)}$}\emph{)}. Then there exist $c,C,\gamma>0$ independent of $\hslash$
 such that for all $\hslash\in (0,1)$ and for all $\psi\in D(\hat{H}_{\hslash}^0)$,
  \begin{align}
  \label{eq:unH}
    c\| (\hat{H}_{\hslash}+\gamma)\psi  \|\leq \| (\hat{H}_{\hslash}^0+1)\psi  \|
    \leq C
    \| (\hat{H}_{\hslash}+\gamma)\psi  \| .
  \end{align}
\end{lemma}
\begin{proof}
We set $A_j=A_j(\hat q_j,\hat a)$, $p_j=\hat p_j$,
$H^0=\hat H_{\hslash}^0$, $H_{\mathrm{f}}=\hat{H}_{\mathrm{f}}$, and $H=\hat H_{\hslash}$ for simplicity, and
we assume that $\psi$ is sufficiently smooth.
By Theorem \ref{prop:2} and the closed graph theorem
we obtain $(\ref{eq:unH})$, but it is not clear how $c$ and $C$ depend on
$\hslash$.
Hence we shall prove the lemma from the scratch.
In the proof below constants $c_j$ are independent of $\hslash\in (0,1]$ and
depend at most on the $L^2$-norms of $\chi_i/|\cdot|$, $\chi_i/\sqrt{|\cdot|}$, $\chi_i$ and
$\sqrt{|\cdot|}\chi_i$. Firstly we assume that $V=0$.
We estimate $\sum_{j=1}^n \|\frac{1}{2m_j} p_j^2\psi \|^2$ and
$\|H_{\rm f}\psi \|^2$.
We have the identity
$(p_j-A_j)^2\psi =p_j^2\psi -2A_j(p_j-A_j)\psi -A_j^2 \psi $ and thus
\begin{align}
\label{h1}
p_j^2\psi =(p_j-A_j)^2\psi +2A_j(p_j-A_j)\psi +A_j^2 \psi\,,
\end{align}
and
\begin{align*}
\|p_j^2\psi \|^2 \leq c_1(\|(p_j-A_j)^2\psi \|^2+\|A_j(p_j-A_j)\psi \|^2+
\|A_j^2 \psi \|^2)\,.
\end{align*}
We obtain
that
\begin{align*}
\|A_j(p_j-A_j)\psi \|^2
&\leq
c_2\langle(p_j-A_j)\psi , (H_{\rm f}+1)(p_j-A_j)\psi \rangle \\
&=
c_2\langle(p_j-A_j)^2\psi , (H_{\rm f}+1)\psi \rangle
+
c_2\langle(p_j-A_j)\psi , -[H_{\rm f},A_j] \psi \rangle.
\end{align*}
Note that $|\langle(p_j-A_j)\psi , -[H_{\rm f},A_j] \psi \rangle|\leq
c_3\|(p_j-A_j)\psi \|\| (H_{\rm f}+1)^{1/2}\psi \|$. Then
\begin{align*}
\|A_j(p_j-A_j)\psi \|^2\leq
c_4
(\|(p_j-A_j)^2\psi \|^2+\|(H_{\rm f}+1)\psi \|^2+\|\psi \|^2)\,.
 \end{align*}
Since from \eqref{eq.a.est}-\eqref{eq.ast.est}
$$\|A_j^2\psi\|\leq c\|(H_{\rm f}+1)\psi\|$$
follows, we conclude that
 \begin{align}
 \label{eq.pf.2}
 \sum_{j=1}^n \|\frac{1}{2m_j} p_j^2\psi \|^2
 \leq
c_5( \sum_{j=1}^n \|\frac{1}{2m_j}(p_j-A_j)^2\psi \|^2+\|(H_{\rm f}+1)\psi \|^2+\|\psi \|^2)\,.
 \end{align}
 Next we estimate
 $\|H_{\rm f}\psi \|^2$.
 We have
 \begin{align*}
 \|H\psi \|^2
 &=
 \|\sum_{j=1}^n \frac{1}{2m_j}(p_j-A_j)^2\psi +H_{\rm f}\psi \|^2\\
 &=
 \|\sum_{j=1}^n \frac{1}{2m_j}(p_j-A_j)^2\psi\|^2
+ \|H_{\rm f}\psi \|^2+2\Ree\langle\sum_{j=1}^n \frac{1}{2m_j}(p_j-A_j)^2\psi ,
 H_{\rm f} \psi \rangle\,.
 \end{align*}
 Since $[p_j-A_j,p_i-A_i]\psi=0$ for $i\neq j$, we have
  \begin{align}
 \label{eq.pf.1}
 &\|\sum_{j=1}^n \frac{1}{2m_j}(p_j-A_j)^2\psi \|^2\nonumber\\
 &=
 \sum_{j=1}^n \|\frac{1}{2m_j}(p_j-A_j)^2\psi \|^2+
 \sum_{i\neq j} \frac{1}{2m_j}\frac{1}{2m_i}
 \langle(p_j-A_j)^2\psi ,
 (p_i-A_i)^2\psi \rangle\\
& \geq
 \sum_{j=1}^n \|\frac{1}{2m_j}(p_j-A_j)^2\psi \|^2\,.\nonumber
 \end{align}
 In addition to this we can see that for arbitrary small $\epsilon>0$ there
 exists $c_7>0$ such that
 \begin{align*}
&\Ree\langle\sum_{j=1}^n \frac{1}{2m_j}(p_j-A_j)^2\psi ,  H_{\rm f} \psi \rangle  \\
&=
\sum_{j=1}^n \frac{1}{2m_j}\langle(p_j-A_j)\psi ,  H_{\rm f} (p_j-A_j)\psi \rangle
+
\sum_{j=1}^n \frac{1}{2m_j}\Ree\langle(p_j-A_j)\psi ,  [H_{\rm f}, -A_j]\psi \rangle  \\
& \geq
-c_6\sum_{j=1}^n \frac{1}{2m_j}\|(p_j-A_j)\psi \|\|(H_{\rm f}+1)^{1/2}\psi \|  \\
& \geq
-\epsilon \sum_{j=1}^n \frac{1}{2m_j}\|(p_j-A_j)^2\psi \|^2
-\epsilon \|(H_{\rm f}+1)\psi \| ^2-c_7\|\psi \|^2.
\end{align*}
Hence there exists $c_8>0$ such that
\begin{align*}
c_8 (\sum_{j=1}^n \|\frac{1}{2m_j}(p_j-A_j)^2\psi \|^2+\|H_{\rm f}\psi \|^2)
\leq
 \|H\psi \|^2+c_9\|\psi \|^2\,.
  \end{align*}
  Then combining \eqref{eq.pf.2} and \eqref{eq.pf.1} yield
  \begin{align*}
  &\|\sum_{j=1}^n \frac{1}{2m_j}p_j^2\psi \|^2+\|H_{\rm f}\psi \|^2
\leq
c_{10}
\sum_{j=1}^n  \| \frac{1}{2m_j} p_j^2\psi \|^2+\|H_{\rm f}\psi \|^2\\
&\leq
c_{11} (\sum_{j=1}^n \|\frac{1}{2m_j}(p_j-A_j)^2\psi \|^2+\|H_{\rm f}\psi \|^2
+\|\psi \|^2)\\
&\leq c_{12}(\|H\psi \|^2+\|\psi \|^2)
\leq
c_{13}\|(H+1)\psi \|^2.
\end{align*}
It finally follows that
$\|(H^0+1)\psi\|\leq c_{14}\|(H+1)\psi\|$.
We can also show the inequality of the right hand side,
but it is rather easier than that of the left hand side.
We have
\begin{align*}
\|H\psi\|\leq \| H^0\psi\|+
\|
\sum_{j=1}^n \frac{1}{m_j}A _jp_j\psi +
\sum_{j=1}^n \frac{1}{2m_j}A _j^2\psi
\|
\end{align*} and
\begin{align*}
\|
\sum_{j=1}^n \frac{1}{m_j}A _jp_j\psi +
\sum_{j=1}^n \frac{1}{2m_j}A _j^2\psi
\|
\leq c_{14}
(\|\sum_{j=1}^n \frac{1}{2m_j}p_j^2\psi\|+\|H_{\rm f}\psi\|+\|\psi\| )\,.
\end{align*}
Then it follows that
$$\|(H+1)\psi\|\leq c_{15}\|(H^0+1)\psi\|$$
and the proof is complete for $V=0$.  In the case of $V\neq 0$, \eqref{eq:unH} can be also derived straightforwardly since $V$ is bounded.
\end{proof}

From Lemma \ref{s3.lemE} we can immediately obtain
a form version of inequalities.
Suppose that $\hat{H}_{\hslash}+a>0$ and $\hat{H}_{\hslash}^0+b>0$ with some $a,b\in{\mathbb R}$.

\begin{cor}
 \label{s4.lemE}
 Assume \eqref{eq:hypo0}, \eqref{eq:hypo1} and \emph{(}\hyperref[eq:hypo2]{$A_2^{(1)}$}\emph{)}. Then there exist $c,C>0$ independent of $\hslash$
 such that for all $\hslash\in (0,1)$ and for all $\psi\in D(\hat{H}_{\hslash}^0)$,
  \begin{align}
  \label{eq:unHs}
    c\langle\psi,  (\hat{H}_{\hslash}+a)\psi \rangle\leq
    \langle\psi,  (\hat{H}_{\hslash}^0+b)\psi\rangle
    \leq C
    \langle\psi, (\hat{H}_{\hslash}+a)\psi\rangle\,.
  \end{align}
In particular
  \begin{align}
  \label{eq:unHs2}
    c\|(\hat{H}_{\hslash}+a)^{1/2}\psi \|
    \leq
    \| (\hat{H}_{\hslash}^0+b)^{1/2}\psi\|
    \leq C
    \| (\hat{H}_{\hslash}+a)^{1/2}\psi\|
  \end{align}
for $\psi\in D((\hat{H}_{\hslash}^0)^{1/2})$.
\end{cor}
\begin{proof}
The proof for \eqref{eq:unHs} is similar but rather easier than that of Lemma \ref{s3.lemE}. To show \eqref{eq:unHs2} we take an approximation.
\eqref{eq:unHs2} is true for $\psi\in D(\hat{H}_{\hslash}^0)$.
Let $\psi\in D((\hat{H}_{\hslash}^0)^{1/2})$. There exists a sequence
$\psi_n\in
D(\hat{H}_{\hslash}^0) $ such that $\psi_n\to\psi$ and
$(\hat{H}_{\hslash}^0)^{1/2}\psi_n\to
(\hat{H}_{\hslash}^0)^{1/2}\psi$ as $n\to\infty$.
Hence
$\hat{H}_{\hslash}^{1/2}\psi_n$
is a Cauchy sequence.
By the closedness of $\hat{H}_\hslash^{1/2}$  we see that
$\psi\in D(\hat{H}_{\hslash}^{1/2})$ and
the inequality of the left  hand side of \eqref{eq:unHs2} follows.
The inequality of the right hand side of \eqref{eq:unHs2} can be similarly proven.
\end{proof}

\subsection{Uniform  estimates}
\label{sec:basic-estimates}

In this paragraph, we use the following shorthand notations:
\begin{equation}
\label{eq:L2}
\mathfrak{H}=L^2(\mathds{R}^d,\mathds{C}^{d-1}\times \mathds{R}^d)\,,\qquad  \mathfrak{H}^0=L^2(\mathds{R}^d,\mathds{C}^{d-1})\,.
\end{equation}
Consider a \emph{separable} Hilbert space $\mathcal{H}$, and bounded
operators ${v}\in \mathscr{L}\bigl(\mathcal{H},\mathcal{H}\otimes \mathfrak{H}\bigr)$
such that there exist a.e.\ defined maps
\begin{equation}
\label{eq:vk}
k\mapsto {v}^{\nu}(k)=\bigl({v}_1^{\nu}(k),\dotsc,{v}_{d-1}^{\nu}(k)\bigr), \quad \forall \nu\in \{1,\dotsc,d\}
\end{equation}
such that for all $\phi\in \mathcal{H}$, and a.a.\ $k\in \mathds{R}^d$:

\begin{eqnarray*}
&&{v}_\lambda^{\nu}(k)\in \mathscr{L}(\mathcal{H});
\\
&& \mathcal{H}\ni {({v}\phi)}_\lambda^{\nu} (k)={v}_\lambda^{\nu}(k) \phi;
\\
&&\sum_{\lambda=1}^{d-1}\sum_{\nu=1}^d\int_{\mathds{R}^d}^{}\lVert {v}_\lambda^{\nu}(k)\phi \rVert_{\mathcal{H}}^2 \mathrm{d}k=\lVert {v}\phi
  \rVert_{\mathcal{H}\otimes \mathfrak{H}}^2\,.
\end{eqnarray*}

In particular, the given representation \eqref{eq:vk} defines an operator ${v}^{\nu}\in  \mathscr{L}(\mathcal{H},\mathcal{H}\otimes \mathfrak{H}^0)$, with
\begin{equation*}
 \lVert {v}^{\nu}  \rVert^2_{\mathscr{L}(\mathcal{H},\mathcal{H}\otimes \mathfrak{H}^0)}=\sum_{\lambda=1}^{d-1}\int_{\mathds{R}^d}^{}\lVert {v}_\lambda^{\nu}(k)  \rVert_{\mathscr{L}(\mathcal{H})}^2  \mathrm{d}k\; ;
\end{equation*}
and
\begin{equation*}
  \lVert {v} \rVert_{\mathscr{L}(\mathcal{H},\mathcal{H}\otimes \mathfrak{H})}^2=\sum_{\nu=1}^d\lVert {v}^{\nu}  \rVert_{\mathscr{L}(\mathcal{H},\mathcal{H}\otimes \mathfrak{H}^0)}^2\; .
\end{equation*}
This leads to the following definition.
\begin{defn}
  \label{def:1}
  Let ${v}\in \mathscr{L}\bigl(\mathcal{H},\mathcal{H}\otimes \mathfrak{H}\bigr)$ be
  as above; then $\hat{a}({v}^{\nu})$
  and $\hat{a}^{*}({v}^{\nu})$, $\nu\in \{1,\dotsc,d\}$, are the closed, densely defined operators on $\mathcal{H}\otimes
  \Gamma_{\mathrm{s}}\bigl(\mathfrak{H}\bigr)$ whose action is defined by
  \begin{equation}
  \label{eq:aa}
    \hat{a}({v}^{\nu})=\sum_{\lambda=1}^{d-1}\int_{\mathbb{R}^d}^{}(v_\lambda^{\nu})^{*}(k)\;
    \hat{a}_\lambda(k)  \mathrm{d}k\; ,\; \hat{a}^{*}({v}^{\nu})=\sum_{\lambda=1}^{d-1}\int_{\mathbb{R}^d}^{}{v}_\lambda^{\nu}(k)\;
    \hat{a}_\lambda^{*}(k)  \mathrm{d}k\; .
  \end{equation}
  Here $({v}_\lambda^{\nu})^{*}(k)\in \mathscr{L}(\mathcal{H})$ is the adjoint of ${v}_\lambda^{\nu}(k)$.
\end{defn}
If ${v}$ and ${w}$ are such that for a.a.\ $k\in \mathds{R}^d$ and all $\lambda,\lambda'\in \{1,\dotsc,d-1\}$ and $\nu\in
\{1,\dotsc,d\}$, $[({v}_\lambda^\nu)^{*}(k),{w}_{\lambda'}^\nu(k')]=0$, then the $d$-dimensional vectors $\hat{a}({v})$ and
$\hat{a}^{*}({w})$ (whose components are $\hat{a}({v}^{\nu})$ and $\hat{a}^{*}({w}^{\nu})$ respectively) satisfy generalized
canonical commutation relations on a dense set, as stated in the following lemma, whose proof is
straightforward.
\begin{lemma}
  \label{lemma:1}
  If for a.a.\ $(k,k')\in \mathds{R}^{2d}$ and all $\lambda,\lambda'\in \{1,\dotsc,d-1\}$ and
  $\nu\in \{1,\dotsc,d\}$,
  $$
  [({v}_\lambda^{\nu})^{*}(k),{w}_{\lambda'}^{\nu}(k')]=0\,,
  $$
  then $\hat{a}({v}^{\nu})$ and $\hat{a}^{*}({w}^{\nu})$ satisfy
  \begin{equation*}
    [\hat{a}({v}),\hat{a}^{*}({w})]:=\sum_{\nu=1}^d[\hat{a}({v}^{\nu}),\hat{a}^{*}({w}^{\nu})]
    =\hslash\, {v}^{*}\cdot {w}=\hslash\sum_{\lambda=1}^{d-1}\sum_{\nu=1}^d\int_{\mathds{R}^d}^{}(v_\lambda^\nu)^{*}(k)\;
    {w}_\lambda^{\nu}(k)  \,\mathrm{d}k\; .
  \end{equation*}
\end{lemma}
Let us adopt the following shorthand notation:
\begin{equation*}
    {\llbracket{v}^{*},{w}\rrbracket=0} \Leftrightarrow
    \text{for a.a.\ } (k,k')\in \mathds{R}^{2d}, \forall \lambda,\lambda'\in \{1,\dotsc,d-1\}, \nu\in \{1,\dotsc,d\}: [({v}_\lambda^\nu)^{*}(k),{w}_{\lambda'}^\nu(k')]=0.
\end{equation*}
In addition, let us denote by $\omega$ the operator of multiplication by $\lvert k \rvert_{}^{}$. Given ${v}\in
\mathscr{L}\bigl(\mathcal{H},\mathcal{H}\otimes \mathfrak{H}\bigr)$, with associated representation
$k\mapsto {v}^{\nu}(k)$ for a.a.\ $k\in \mathds{R}^d$, then $\omega^{\beta}{v}$ has $k\mapsto \lvert k \rvert_{}^{\beta}{v}^{\nu}(k)$ as
associated representation.
\begin{lemma}
  \label{lemma:2}
  Let ${v}\in \mathscr{L}\bigl(\mathcal{H},\mathcal{H}\otimes \mathfrak{H}\bigr)$, with associated
  representation $k\mapsto {v}^{\nu}(k)$ for a.a.\ $k\in \mathds{R}^d$. If $\llbracket{v}^{*},{v}\rrbracket=0$, then, uniformly w.r.t.\
  $\hslash\in (0,1)$: \vspace{1.5mm} {\normalfont{}
    \begin{enumerate}
      \setlength{\itemsep}{1.5mm}
    \item\label{item:1} $\lVert \hat{a}({v}^{\nu}) \,{(\hat{H}_{\mathrm{f}}+1)}^{-\frac{1}{2}} \rVert_{\mathscr{L}(\mathcal{H}\otimes
        \Gamma_{\mathrm{s}}(\mathfrak{H}))}\leq \lVert \omega^{-\frac{1}{2}}{v}^{\nu} \rVert_{\mathscr{L}(\mathcal{H},\mathcal{H}\otimes \mathfrak{H}^0)}\, ,$
    \item\label{item:3} $\lVert \hat{a}^{*}({v}^{\nu})\,{(\hat{H}_{\mathrm{f}}+1)}^{-\frac{1}{2}} \rVert_{\mathscr{L}(\mathcal{H}\otimes
        \Gamma_{\mathrm{s}}(\mathfrak{H}))}\leq \lVert \omega^{-\frac{1}{2}}{v}^{\nu} \rVert_{\mathscr{L}(\mathcal{H},\mathcal{H}\otimes \mathfrak{H}^0)}+\lVert {v}^{\nu} \rVert_{\mathscr{L}(\mathcal{H},\mathcal{H}\otimes \mathfrak{H}^0)}\, .$
    \end{enumerate}}
\end{lemma}
\begin{remark}
  \label{rem:3}
  Analogous estimates hold if we take $\omega\equiv1$, and replace $\hat{H}_{\mathrm{f}}$ with the number operator $\hat{N}_{\hslash}$ defined in \eqref{eq:Nop}.
\end{remark}
\begin{proof}
  Let us start with~\eqref{item:1}. By definition, one obtains
  \begin{equation*}
    \lVert \hat{a}({v}^{\nu})\Psi  \rVert_{}^2=\sum_{n\in \mathds{N}}^{}\lVert \hat{a}({v}^{\nu})\Psi^{(n)}  \rVert_{}^2=\sum_{n\in \mathds{N}}^{}\hslash n \lVert ({v}^{\nu})^{*}\otimes 1^{(n-1)}\Psi^{(n)}  \rVert_{}^2\; .
  \end{equation*}
  Hence multiplying and dividing by $\omega^{\frac{1}{2}}$ (acting on the first variable), using the symmetry
  of $\Psi^{(n)}$, and $\lVert ({v}^{\nu})^{*} \,\omega^{-\frac{1}{2}} \rVert_{\mathscr{L}(\mathcal{H}\otimes \mathfrak{H}^0,\mathcal{H})}^{}=\lVert
  \omega^{-\frac{1}{2}}{v}^{\nu} \rVert_{\mathscr{L}(\mathcal{H},\mathcal{H}\otimes \mathfrak{H}^0)}^{}$ one obtains
  \begin{equation*}
    \begin{split}
      \sum_{\nu=1}^d\lVert \hat{a}({v}^{\nu})\Psi  \rVert_{}^2\leq \sum_{j}^{}\lVert ({v}^{\nu})^{*} \,\omega^{-\frac{1}{2}}  \rVert_{\mathscr{L}(\mathcal{H}\otimes \mathfrak{H}^0,\mathcal{H})}^2\sum_{n\in \mathds{N}}^{}\hslash n \lVert \omega^{\frac{1}{2}}\otimes 1^{(n-1)}\Psi^{(n)}  &\rVert_{}^2\\\leq \lVert \omega^{-\frac{1}{2}}{v}^{\nu} \rVert_{\mathscr{L}(\mathcal{H},\mathcal{H}\otimes \mathfrak{H}^0)}^{} \lVert {(\hat{H}_{\mathrm{f}}+1)}^{\frac{1}{2}}\Psi  &\rVert_{}^2\; .
    \end{split}
  \end{equation*}
  To prove~\eqref{item:3}, one uses the canonical commutation relations:
  \begin{equation*}
    \begin{split}
      \lVert \hat{a}^{*}({v}^{\nu})\Psi  \rVert_{}^2=\langle \Psi  , \hat{a}({v}^{\nu})\hat{a}^{*}({v}^{\nu})\Psi \rangle_{}=\langle \Psi  , \hat{a}^{*}({v}^{\nu})\hat{a}({v}^{\nu})\Psi \rangle_{}+\sum_{\lambda=1}^{d-1}\int_{\mathds{R}^d}^{}\langle \Psi  ,   ({v}_\lambda^{\nu})^{*}(k){v}_\lambda^{\nu}(k) \,\Psi
      \rangle_{}\mathrm{d}k\\\leq \lVert \omega^{-\frac{1}{2}}{v}^{\nu} \rVert_{\mathscr{L}(\mathcal{H},\mathcal{H}\otimes \mathfrak{H}^0)}^{} \lVert {(\hat{H}_{\mathrm{f}}+1)}^{\frac{1}{2}}\Psi  \rVert_{}^2 + \lVert {v}^{\nu}  \rVert_{\mathscr{L}(\mathcal{H},\mathcal{H}\otimes \mathfrak{H}^0)}^2\lVert \Psi  \rVert_{}^2\; .
    \end{split}
  \end{equation*}
\end{proof}
We would like to apply the general estimates of Lemma~\ref{lemma:2} to the Pauli--Fierz Hamiltonian. To this
extent, let us identify
$$
\mathcal{H}\equiv L^2 (\mathds{R}^{dn}, \mathds{C} )$$
and eventually take
$$
{v}_\lambda^{\nu}(k,x_j)\equiv \lvert k
\rvert_{}^{-\frac{1}{2}} \,\epsilon_\lambda^{\nu}(k)\,\chi_i(k)\,e^{-2\pi\ci k\cdot x_j}\,,
$$
acting as a $x_j$-multiplicative operator on $L^2 (\mathds{R}^{dn}, \mathds{C} )$ with $x_j\in\R^d$ is the $j^{th}$ particle position.
\begin{cor}
  \label{cor:1}
  Let ${v},{w}\in \mathscr{L}\bigl(L^2,L^2\otimes \mathfrak{H}\bigr)$ such that
  $\llbracket{v}^{*},{v}\rrbracket=\llbracket{w}^{*},{w}\rrbracket=0$ and let $\|\cdot\|_{\mathscr{L}}=\|\cdot\|_{\mathscr{L}(\mathcal{H},\mathcal{H}\otimes \mathfrak{H})}$.
  Then the following bounds hold uniformly w.r.t.\ $\hslash\in
  (0,1)$: \vspace{1.5mm} {\normalfont{}
    \begin{enumerate}
      \setlength{\itemsep}{1.5mm}
    \item\label{item:4} $\lVert {(\hat{H}_{\mathrm{f}}+1)}^{-\frac{1}{2}}\,\hat{a}^{*}({v})\cdot \hat{a}^{*}({w})\,
      {(\hat{H}_{\mathrm{f}}+1)}^{-\frac{1}{2}} \rVert_{\mathscr{L}(\mathscr{H})}^{}\leq \lVert
      \omega^{-\frac{1}{2}}{v} \rVert_{\mathscr{L}}^{}(\lVert \omega^{-\frac{1}{2}}{w} \rVert_{\mathscr{L}}^{}+\lVert {w}
      \rVert_{\mathscr{L}}^{})$
    \item\label{item:7} $\lVert {(\hat{H}_{\mathrm{f}}+1)}^{-\frac{1}{2}}\,\hat{a}({v})\cdot \hat{a}({w})\,
      {(\hat{H}_{\mathrm{f}}+1)}^{-\frac{1}{2}} \rVert_{\mathscr{L}(\mathscr{H})}^{}\leq (\lVert
      \omega^{-\frac{1}{2}}{v} \rVert_{\mathscr{L}}^{}+\lVert {v} \rVert_{\mathscr{L}}^{})\lVert \omega^{-\frac{1}{2}}{w}
      \rVert_{\mathscr{L}}^{}$

    \item\label{item:8} $\lVert {(\hat{H}_{\mathrm{f}}+1)}^{-\frac{1}{2}}\,\hat{a}^{*}({v})\cdot \hat{a}({w})\,
      {(\hat{H}_{\mathrm{f}}+1)}^{-\frac{1}{2}} \rVert_{\mathscr{L}(\mathscr{H})}^{}\leq \lVert \omega^{-\frac{1}{2}}{v}
      \rVert_{\mathscr{L}}^{}\lVert \omega^{-\frac{1}{2}}{w} \rVert_{\mathscr{L}}^{}$

    \item\label{item:9} $\lVert {(\hat{H}_{\mathrm{f}}+1)}^{-\frac{1}{2}}\,\hat{a}({v})\cdot \hat{a}^{*}({w})\,
      {(\hat{H}_{\mathrm{f}}+1)}^{-\frac{1}{2}} \rVert_{\mathscr{L}(\mathscr{H})}^{}\leq (\lVert \omega^{-\frac{1}{2}}{v}
      \rVert_{\mathscr{L}}^{}+\lVert {v} \rVert_{\mathscr{L}}^{})(\lVert \omega^{-\frac{1}{2}}{w} \rVert_{\mathscr{L}}^{}+\lVert
      {w} \rVert_{\mathscr{L}}^{})$

    \item\label{item:5} $\lVert {(\hat{H}_{\hslash}^0+1)}^{-\frac{1}{2}}\,\hat{{p}}\cdot \hat{a}^{*}({v})\,
      {(\hat{H}_{\hslash}^0+1)}^{-\frac{1}{2}} \rVert_{\mathscr{L}(\mathscr{H})}^{}\leq \lVert \omega^{-\frac{1}{2}}{v}
      \rVert_{\mathscr{L}}^{}+\lVert {v} \rVert_{\mathscr{L}}^{}$

    \item\label{item:6} $\lVert {(\hat{H}_{\hslash}^0+1)}^{-\frac{1}{2}}\,\hat{a}^{*}({v})\cdot \hat{{p}}
      \,{(\hat{H}_{\hslash}^0+1)}^{-\frac{1}{2}} \rVert_{\mathscr{L}(\mathscr{H})}^{}\leq \lVert \omega^{-\frac{1}{2}}{v}
      \rVert_{\mathscr{L}}^{}+\lVert {v} \rVert_{\mathscr{L}}^{}\;.$
    \end{enumerate}}
\end{cor}

\subsection{Regularity properties}
\label{sec:regularity}

In order to understand the behavior  of  the Pauli--Fierz dynamics  as $\hslash\to 0$, it is
convenient to control the expectation of some suitable quantum observables \emph{uniformly in
  $\hslash\in (0,1)$}. Namely, we consider the number operator \eqref{eq:Nop}, the modulus  momentum operator \eqref{eq:dgk} and  the squared position operator $\hat q^2$, and establish the following regularity propagation estimates for  the Pauli--Fierz dynamics.  These estimates prove to be very useful in Section \ref{sec.CH}.

\subsubsection{Number estimates}
\label{sec:number-estimates}

The aim of present subsection is to study the propagation in time of
the photon number operator estimates on a
suitable dense domain.
The commutation relations of $\hat{N}_{\hslash}$ with $\hat{a}_\lambda(f)$ and $\hat{a}_\lambda^{*}(f)$ have a simple form:
\begin{align*}
  &[\hat{N}_{\hslash},\hat{a}_{\lambda}(f)]=-\hslash\, \hat{a}_{\lambda}(f),    \\
  &[\hat{N}_{\hslash},\hat{a}^{*}_{\lambda}(f)]=\hslash\, \hat{a}^{*}_{\lambda}  (f).
\end{align*}
We shall characterize a dense domain $\mathcal{D}\subset D(\hat{N}_{\hslash})$ such that for
all $\psi \in \mathcal{D}$ and for all $t\in {\mathbb R}$,
$$\psi(t)=\hat{U}_{\hslash}(t)\psi = e^{-\ci\frac{t}{\hslash} \hat H_\hslash} \psi \in D(\hat{N}_{\hslash}).$$
To this extent, it is technically
convenient to introduce a bounded and positive, approximation of $\hat{N}_{\hslash}$ that strongly converges
monotonically to it:
\begin{align*}
  \hat{N}_{\hslash}(\delta)=\hat{N}_{\hslash}e^{-\delta \hat{N}_{\hslash}} ,   \delta\geq 0.
\end{align*}
It is straightforward to check that on their dense common domain of definition, $\hat{H}_{\hslash}^0$ and $\hat{N}_{\hslash}(\delta)$
commute for all $\delta\geq 0$:
\begin{align*}
  [\hat{H}_{\hslash}^0,\hat{N}_{\hslash}(\delta)]=0.
\end{align*}
Let $\psi \in D(\hat{H}_\hslash^0)$, and define
\begin{align*}
  M(t)= \| \hat{N}_{\hslash}(\delta)\psi(t)  \| _{}^2.
\end{align*}
The quantity $M(t)$ is finite for every $t\in {\mathbb R}$ and $\delta>0$, since $\hat{N}_{\hslash}(\delta)$, $\delta>0$, is a bounded
operator. In addition, the map $t\mapsto M(t)$ is differentiable, with
\begin{align}
  \label{eq:18}
  \dot{M}(t)=\frac{\ci}{\hslash}\big(\langle \hat{N}_{\hslash}(\delta)\hat{H}_{\hslash}\psi(t),   \hat{N}_{\hslash}(\delta)\psi(t) \rangle_{}-\langle \hat{N}_{\hslash}(\delta)\psi(t),   \hat{N}_{\hslash}(\delta)\hat{H}_{\hslash}\psi(t) \rangle_{}\big).
\end{align}
In light of \eqref{eq:18}, let us introduce the densely defined sesquilinear form
${\rm q}_{\delta,\hslash}:
{\mathscr H}\times{\mathscr H}\to {\mathbb C}$ defined by
\begin{align*}
  {\rm q}_{\delta,\hslash}(\phi,\psi)=\frac{\ci}{\hslash}\big(\langle \hat{N}_{\hslash}(\delta)\hat{H}_{\hslash}\phi,   \hat{N}_{\hslash}(\delta)\psi \rangle_{}-\langle \hat{N}_{\hslash}(\delta)\phi,   \hat{N}_{\hslash}(\delta)\hat{H}_{\hslash}\psi \rangle_{}\big),
\end{align*}
for any $\phi,\psi\in \mathscr{C}_0^{\infty}({\mathbb R}^{dn})\otimes \Gamma_{\mathrm{fin}}$, where $\Gamma_{\mathrm{fin}}$ is the subset of Fock space vectors with a finite number of photons.
Let us omit the explicit dependence on $\delta$ and $\hslash$ of the sesquilinear
form, \emph{i.e.}  ${\rm q}\equiv {\rm q}_{\delta,\hslash}$. It is possible to rewrite ${\rm q}$ by using the
commutator between $\hat{N}_{\hslash}(\delta)$ and $\hat{H}_{\hslash}$ (on $\mathscr{C}_0^{\infty}({\mathbb R}^{dn})\otimes \Gamma_{\mathrm{fin}}$):
\begin{align*}
  {\rm q}(\phi,\psi)=\frac{\ci}{\hslash}\big(\langle [\hat{N}_{\hslash}(\delta),\hat{H}_{\hslash}]\phi,   \hat{N}_{\hslash}(\delta)\psi \rangle_{}-\langle \hat{N}_{\hslash}(\delta)\phi,   [\hat{N}_{\hslash}(\delta),\hat{H}_{\hslash}]\psi \rangle_{}\big).
\end{align*}

\begin{lemma}
  \label{lemma:33}
  Assume \eqref{eq:hypo0}, \eqref{eq:hypo1} and
  \emph{(}\hyperref[eq:hypo2]{$A_2^{(1)}$}\emph{)}. Then there exists $C>0$ such that for all $\delta>0$, $\hslash\in (0,1)$ and
  $\phi,\psi\in \mathscr C_0^{\infty}({\mathbb R}^{dn})\otimes
  \Gamma_{\mathrm{fin}}$:
  \begin{align*}
    \| {\rm q}(\phi,\psi)  \| \leq C\big(\| (\hat{H}_{\hslash}^0+1) \phi \|\| \hat{N}_{\hslash}(\delta)\psi  \|+\| \hat{N}_{\hslash}(\delta)\phi  \|\| (\hat{H}_{\hslash}^0+1)\psi  \|\big).
  \end{align*}
  There exists also a constant $C'>0$ such that
  \begin{align*}
    \| {\rm q}(\phi,\psi)  \| \leq C'\big(\| (\hat{H}_{\hslash}+1) \phi \|\| \hat{N}_{\hslash}(\delta)\psi  \|+\| \hat{N}_{\hslash}(\delta)\phi  \|\| (\hat{H}_{\hslash}+1)\psi  \|\big).
  \end{align*}
\end{lemma}
\begin{proof}
The commutator can be computed explicitly as follows
\begin{align*}
  [\hat{N}_{\hslash}(\delta),\hat{H}_{\hslash}]=\frac{1}{2}\sum_{j=1}^n \sum_{\mu=1}^d\big([\hat{N}_{\hslash}(\delta),{\rm D}_{\mu}^j]{\rm D}_{\mu}^j+{\rm D}_{\mu}^j [\hat{N}_{\hslash}(\delta),{\rm D}_{\mu}^j]\big),
\end{align*}
where ${\rm D}^j=\hat{{p}}_j-A_j$ is the covariant derivative. In addition,
\begin{align*}
  [\hat{N}_{\hslash}(\delta),{\rm D}_{\mu}^j]=-\ci(\hslash\, \Pi_{j,\mu}+\hat{N}_{\hslash}\Pi_{j,\mu}(\delta))e^{-\delta \hat{N}_{\hslash}},
\end{align*}
where $\Pi_{j}$ is the conjugate momentum of $A_j$, \emph{i.e.}
\begin{align*}
  \Pi_{j}=\ci\sum_{\lambda=1}^{d-1}\int_{{\mathbb R}^d}^{}
v_\lambda(k,\hat{q}_j)\hat{a}_\lambda^{*}(k)-
\bar v_\lambda(k,\hat{q}_j)\hat{a}_\lambda(k)\,  \mathrm{d}k,
\end{align*}
with
\begin{equation*}
  v_{\lambda}(k,\hat{q}_j)= \frac{\epsilon_{\lambda}(k)}{\sqrt{2\lvert k  \rvert_{}^{}}}\chi_j(k)e^{2\pi \ci k\cdot \hat{q}_j}
\end{equation*}
and
\begin{align*}
  \Pi_{j}(\delta)=\ci\sum_{\lambda=1}^{d-1}\int_{{\mathbb R}^d}^{}
v_\lambda(k,\hat{q}_j)(e^{-\hslash\delta}-1)\hat{a}_\lambda^{*}(k)-\bar v_\lambda(k,\hat{q}_j) (e^{\hslash\delta}-1)\hat{a}_\lambda(k)\,  \mathrm{d}k.
\end{align*}
The above formula is yielded by the commutation relations
\begin{align*}
&  [e^{-\delta \hat{N}_{\hslash}},\hat{a}^{*}_{\lambda}  (f)]=\hat{a}^{*}_{\lambda}  ((e^{-\hslash\delta}-1)f)e^{-\delta \hat{N}_{\hslash}},\\
&  [e^{-\delta \hat{N}_{\hslash}},\hat{a}_{\lambda}(f)]=\hat{a}_{\lambda}  ((e^{\hslash\delta}-1)f)e^{-\delta \hat{N}_{\hslash}}.
\end{align*}
The operators $\Pi_{j}(\delta)$ are not symmetric, but closed on a suitable domain.
Therefore,
\begin{align*}
      \frac{\ci}{\hslash}[\hat{N}_{\hslash}(\delta),\hat{H}_{\hslash}]&=
      \frac{1}{2}\sum_{j=1}^n \big((\Pi_{j}+\hat{N}_1\Pi_{j}(\delta))e^{-\delta \hat{N}_{\hslash}}\cdot {\rm D}^j+{\rm D}^j\cdot (\Pi_{j}+\hat{N}_1\Pi_{j}(\delta))e^{-\delta \hat{N}_{\hslash}}\big)\\
      &=\frac{1}{2}\sum_{j=1}^n
      (\mathrm{I}_1^j+\mathrm{I}_2^j+\mathrm{I}_3^j+\mathrm{I}_4^j ) ,
\end{align*}
with $\hat{N}_1=\hat{N}_{\hslash=1}$, and
\begin{eqnarray*}
  \mathrm{I}_1^j&=\Pi_{j}e^{-\delta \hat{N}_{\hslash}}\cdot {\rm D}^j,  \qquad
  \mathrm{I}_2^j&=\hat{N}_1\Pi_{j}(\delta)e^{-\delta \hat{N}_{\hslash}}\cdot {\rm D}^j,  \\
  \mathrm{I}_3^j&={\rm D}^j\cdot \Pi_{j}e^{-\delta \hat{N}_{\hslash}},   \qquad
 \mathrm{I}_4^j&={\rm D}^j\cdot \hat{N}_1\Pi_{j}(\delta)e^{-\delta \hat{N}_{\hslash}} .
\end{eqnarray*}
The terms $\mathrm{I}^j_1$--$\mathrm{I}^j_4$ are estimated separately.
\begin{align*}
  \mathrm{I}^j_1=\underbrace{e^{-\delta \hat{N}_{\hslash}}\Pi_{j}\phantom{(i)}\mspace{-20mu}\cdot {\rm D}^j}_{\mathrm{I}_1^{(1)}}\,+\,\underbrace{A_j(\delta)e^{-\delta \hat{N}_{\hslash}}\cdot {\rm D}^j}_{\mathrm{I}_1^{(2)}},
\end{align*}
where
\begin{align*}
  A_j(\delta)=\ci\sum_{\lambda=1}^{d-1}\int_{{\mathbb R}^d}^{}
  v_\lambda(k,\hat{q}_j)
  (1-e^{-\hslash\delta})\hat{a}_\lambda^{*}(k)-\bar v_\lambda(k,\hat{q}_j)(1-e^{\hslash\delta})\hat{a}_\lambda(k))  \mathrm{d}k.
\end{align*}
Again, $A_{j}(\delta)$ are not symmetric. Hence
\begin{align*}
  \| \mathrm{I}^{(1)}_1\psi  \|\leq \| \Pi_{j}\cdot {\rm D}^j\psi  \|.
\end{align*}
Now observe that
\begin{align*}
  \| \Pi_{j}\cdot \hat{{p}}\psi \|\leq C\| (\hat{H}_{\hslash}^0+1)\psi  \|,
\end{align*}
and
\begin{align*}
  \| \Pi_{j}\cdot A_j\psi \|\leq C\| (\hat{H}_{\hslash}^0+1)\psi  \|.
\end{align*}
Combining the above estimates, we get
\begin{align*}
  \| \mathrm{I}^{(1)}_1\psi  \|\leq C\| (\hat{H}_{\hslash}^0+1)\psi  \|,
\end{align*}
where $C$ {does not} depend on $\delta$ and $\hslash$. To complete the estimate for $\mathrm{I}^j_1$, we combine the
inequality
\begin{align*}
  \| A_{j,\mu}(\delta)e^{-\delta \hat{N}_{\hslash}}\psi  \|\leq 2\| \chi_j/\sqrt{|\cdot|}   \| _2^{}e^{\hslash\delta}\sqrt{\hslash}
  (|e^{\hslash\delta}-1|+|e^{-\hslash\delta}-1| )\| (\hat{N}_1+1)^{1/2} e^{-\delta\hslash (\hat{N}_1+1)}\psi  \|,
\end{align*}
with the fact that, denoting by ${(E_{\lambda})}_{\lambda\in {\mathbb R}^+}$ the spectral measure of an operator $T$, $T\geq 0$, evaluated on
$\psi$, for all $\alpha\in \frac{1}{2}{\mathbb N}$,
\begin{align*}
  \| T^{\alpha}e^{-\delta\hslash T}\psi  \| _{}^2=\int_0^{\infty}\lambda^{2\alpha} e^{-2\hslash\delta\lambda}  \mathrm{d}E_{\lambda},
\end{align*}
with
\begin{align*}
  \frac{{(2\hslash\delta\lambda)}^{2\alpha}}{(2\alpha)!} e^{-2\hslash\delta\lambda}\leq 1,
\end{align*}
and therefore
\begin{align}
  \label{eq:22}
&  \| \hat{N}_1^{\alpha}e^{-\hslash\delta \hat{N}_1}  \| ^{}\leq \frac{\sqrt{(2\alpha)!}}{{(2\hslash\delta)}^{\alpha}}, \\
  \label{eq:22-1}
 &  \| {(\hat{N}_1+1)}^{\alpha}e^{-\hslash\delta (\hat{N}_1+1)}  \| ^{}\leq \frac{\sqrt{(2\alpha)!}}{{(2\hslash\delta)}^{\alpha}},
\end{align}
to obtain
\begin{align*}
   &\| A_j(\delta)e^{-\delta \hat{N}_{\hslash}}\cdot {\rm D}^j \psi \|\\
   &\leq 2\| \chi_j/\sqrt{|\cdot|}   \| ^{}e^{\hslash\delta}\sqrt{\hslash}
   (|e^{\hslash\delta}-1|+|e^{-\hslash\delta}-1| )
   \| (\hat{N}_1+1)^{1/2} e^{-\hslash\delta (\hat{N}_1+1)}  \|^{}\sum_{\mu=1}^d\| {\rm D}_{\mu}^j\psi  \|\\
    &\leq \sqrt{2}\| \chi_j/\sqrt{|\cdot|}   \| ^{}e^{\hslash\delta}\sqrt{\hslash}
    \frac{|e^{\hslash\delta}-1|+|e^{-\hslash\delta}-1|}{\sqrt{\hslash\delta}}\sum_{\mu=1}^d\| {\rm D}_{\mu}^j\psi  \|\leq C(\hslash,\delta)\| (\hat{H}_{\hslash}^0+1)\psi  \|,
\end{align*}
where $C(\hslash,\delta)\to 0$ as $\hslash\to 0$ for any $\delta\geq 0$. Therefore,
\begin{align*}
  \| \mathrm{I}^j_1 \psi \|\leq C\| (\hat{H}^0_{\hslash}+1)\psi  \|.
\end{align*}
Then, we estimate $\mathrm{I}^j_2$; for this term we need the bounds~\eqref{eq:22} both with $\alpha=1$ and
$\alpha=\frac{3}{2}$. Again, let us split it in two parts:
\begin{align*}
  \mathrm{I}^j_2=\hat{N}_1\Pi_{j}(\delta)e^{-\delta \hat{N}_{\hslash}}\cdot {\rm D}^j=\underbrace{\hat{N}_1e^{-\delta \hat{N}_{\hslash}}\Pi_{j}(\delta)\cdot {\rm D}^j}_{\mathrm{I}_2^{(1)}}\, +\, \underbrace{\hat{N}_1A_j(\delta^2)
  e^{-\delta \hat{N}_{\hslash}}\cdot {\rm D}^j}_{\mathrm{I}_2^{(2)}},
\end{align*}
where
\begin{align*}
  \begin{split}
    A_j(\delta^2)=\ci\sum_{\lambda=1}^{d-1}\int_{{\mathbb R}^d}^{}
    v_\lambda(k,\hat{q}_j)
    {(1-e^{-\hslash\delta})}^2\hat{a}_\lambda^{*}(k)-\bar v_\lambda(k,\hat{q}_j){(1-e^{\hslash\delta})}^2\hat{a}_\lambda(k)  \mathrm{d}k.
  \end{split}
\end{align*}
Similarly to $\mathrm{I}_1^{(2)}$, we obtain
\begin{align*}
  \begin{split}
    \| \mathrm{I}_2^{(1)}\psi  \|\leq 3\sqrt{2} d(\|\chi_j/\sqrt{|\cdot|}   \|^{}+\| \chi_j/{|\cdot|}   \| ^{})
    \frac{|e^{\hslash\delta}-1|+|e^{-\hslash\delta}-1|}{\hslash\delta}\| (\hat{H}_{\hslash}^0+1)\psi  \|\leq C \| (\hat{H}_{\hslash}^0+1)\psi  \|.
  \end{split}
\end{align*}
For $\mathrm{I}_2^{(2)}$ one gets
\begin{align*}
  \begin{split}
    \| \mathrm{I}_2^{(2)}\psi  \|\leq \sqrt{3} d\| \chi_j/\sqrt{|\cdot|}   \| ^{}e^{\hslash\delta}\sqrt{\hslash}
    \frac{{(e^{\hslash\delta}-1)}^2+{(e^{-\hslash\delta}-1)}_{}^2}{\sqrt{\hslash^3\delta^3}}\| (\hat{H}_{\hslash}^0+1)\psi  \|\leq C \| (\hat{H}_{\hslash}^0+1)\psi  \|.
  \end{split}
\end{align*}
Summing up, we have
\begin{align*}
  \| \mathrm{I}^j_2\psi  \|\leq C\| (\hat{H}_{\hslash}^0+1)\psi  \|.
\end{align*}
The estimate of $\mathrm{I}^j_3$ is easy, yielding
\begin{align*}
  \| \mathrm{I}^j_3\psi  \|\leq C\| (\hat{H}_{\hslash}^0+1)e^{-\delta \hat{N}_{\hslash}}\psi  \|\leq C\| (\hat{H}_{\hslash}^0+1)\psi  \|.
\end{align*}
Finally, we estimate
\begin{align*}
  \mathrm{I}^j_4={\rm D}^j\cdot \hat{N}_1\Pi_{j}(\delta)e^{-\delta \hat{N}_{\hslash}}=\underbrace{{\rm D}^j\cdot \Pi_{j}(\delta)\hat{N}_1e^{-\delta \hat{N}_{\hslash}} }_{\mathrm{I}_4^{(1)}}\, +\, \underbrace{{\rm D}^j\cdot A_j(\delta)e^{-\delta \hat{N}_{\hslash}}}_{\mathrm{I}_4^{(2)}}.
\end{align*}
Hence
\begin{align*}
 \| \mathrm{I}^{(1)}_4\psi  \|
&\leq 6(\| \chi_j/\sqrt{|\cdot|}   \| ^{} + \| \chi_j/{|\cdot|}   \| ^{} ) (|e^{\hslash\delta}-1|+|e^{-\hslash\delta}-1| )\| (\hat{H}_{\hslash}^0+1)N_1e^{-\delta \hat{N}_{\hslash}}\psi  \|\\
&\leq 3\sqrt{2}(\| \chi_j/\sqrt{|\cdot|}   \| ^{} + \| \chi_j/{|\cdot|}   \| ^{} )\frac{|e^{\hslash\delta}-1|+|e^{-\hslash\delta}-1| }{\hslash\delta}\| (\hat{H}_{\hslash}^0+1)\psi \|\\
 &\leq C\| (\hat{H}_{\hslash}^0+1)\psi \|.
 \end{align*}
Estimating $\mathrm{I}_4^{(2)}$ is easier, and one gets
\begin{align*}
  \| \mathrm{I}^{(2)}_4\psi  \|\leq 6(\| \chi_j/\sqrt{|\cdot|}   \| ^{} + \| \chi_j/{|\cdot|}   \| ^{} )(|e^{\hslash\delta}-1|+|e^{-\hslash\delta}-1|)\| (\hat{H}_{\hslash}^0+1)\psi  \|  ;
\end{align*}
yielding
\begin{align*}
  \| \mathrm{I}^j_4\psi  \|\leq C\| (\hat{H}_{\hslash}^0+1)\psi  \|.
\end{align*}
Putting the results above together, we prove the  lemma.
\end{proof}

With the aid of Lemma \ref{lemma:33}, we prove the following bound on $\dot{M}(t)$.
\begin{lemma}
  \label{lemma:4}
  Assume \eqref{eq:hypo0}, \eqref{eq:hypo1} and
  \emph{(}\hyperref[eq:hypo2]{$A_2^{(1)}$}\emph{)}. In addition, let $M(t)=\| \hat{N}_{\hslash}(\delta)\psi(t) \| _{}^2$. Then for any $t\in {\mathbb R}$,
  \begin{align*}
    \dot{M}(t)\leq C (\| (\hat{H}_{\hslash}+1)\psi  \| _{}^2 + M(t)).
  \end{align*}
\end{lemma}
\begin{proof}
  By Lemma \ref{lemma:33}, the domain of ${\rm q}$ can be uniquely extended to $D(\hat H_{\hslash})\times D(\hat H_{\hslash})$. In
  particular, we have
  \begin{align*}
       |\dot{M}(t) | &=| {\rm q}(\psi(t),\psi(t))  | \leq
       C\| (\hat{H}_{\hslash}+1)\psi  \|\| \hat{N}_{\hslash}(\delta)\psi(t)  \| \leq
              C(\| (\hat{H}_{\hslash}+1)\psi  \|^2+\| \hat{N}_{\hslash}(\delta)\psi(t)  \|^2) .
  \end{align*}
Then the lemma is proven.
\end{proof}
Gr\"onwall's inequality then yields the sought result of propagation.
\begin{prop}
  \label{prop:3}
  Assume \eqref{eq:hypo0}, \eqref{eq:hypo1} and
  \emph{(}\hyperref[eq:hypo2]{$A_2^{(1)}$}\emph{)}. In addition, let $\psi \in D(\hat{{p}}^2)\cap D(\hat{H}_{\rm f} )\cap D(\hat{N}_{\hslash})$. Then $e^{-\ci\frac{t}{\hslash}\hat{H}_{\hslash}}\psi \in
  D(\hat{N}_{\hslash})$.
\end{prop}
\begin{proof}
  By Lemma \ref{lemma:4}, we have
  \begin{align*}
    M(t)=\int_0^t\dot{M}(s)  \mathrm{d}s + M(0)\leq C\int_0^t(M(s)+\theta)  \mathrm{d}s + M(0),
  \end{align*}
  where $\theta=\| (\hat{H}_{\hslash}+1)\psi \| _{}^2$. Let $M(t)+\theta=N(t)$. Then
  \begin{align*}
    M(t)\leq N(t)\leq C\int_0^t N(s)  \mathrm{d}s + N(0).
  \end{align*}
  Then Gr\"onwall's inequality yields
  \begin{align*}
    M(t)\leq N(t)\leq \big(\| \hat{N}_{\hslash}(\delta)\psi  \| _{}^2+\| (\hat{H}_{\hslash}+1)\psi  \| _{}^2\big)e^{C|t| },
  \end{align*}
  where $C$ does not depend on $\hslash$ and $\delta$. For all $\psi \in D(\hat{N}_{\hslash})$,
  \begin{align*}
    \lim_{\delta\downarrow 0}\| \hat{N}_{\hslash}(\delta)\psi(t)  \| _{}^2=\lim_{\delta\downarrow0}\int_0^{\infty}\lambda^2e^{-2\delta\lambda}  \mathrm{d}E_{\lambda}.
  \end{align*}
  Hence, by the monotone convergence theorem, it follows that $\psi(t)\in D(\hat{N}_{\hslash})$, and
  \begin{align*}
    \| \hat{N}_{\hslash}\psi(t)  \| _{}^2\leq \big(\| \hat{N}_{\hslash}\psi  \| _{}^2+\| (\hat{H}_{\hslash}+1)\psi  \| _{}^2\big)e^{C|t| }.
  \end{align*}
\end{proof}

\begin{cor}
  \label{cor:2}
  Assume \eqref{eq:hypo0}, \eqref{eq:hypo1} and
  \emph{(}\hyperref[eq:hypo2]{$A_2^{(1)}$}\emph{)}. Then there exist two constants $K,C>0$ such that: for all $\hslash\in (0,1)$, for all $\psi \in D(\hat{{p}}^2)\cap
  D(\hat{H}_{\rm f} )\cap D(\hat{N}_{\hslash})$, and for all $t\in {\mathbb R}$,
  \begin{align*}
    \| \hat{N}_{\hslash}\psi(t)  \|\leq K\big(\| \hat{N}_{\hslash}\psi   \|+\| (\hat{H}_{\hslash}^0+1)\psi   \|\big)e^{C|t| }.
  \end{align*}
\end{cor}

\subsubsection{Momentum estimates}
\label{sec:momentum-estimates}

The commutation relations of $\hat{N}^{(\sigma)}_{\hslash}={\rm d}\Gamma(|k|^{2\sigma})$ with $\hat{a}_\lambda(f)$ and $\hat{a}_\lambda^{*}(f)$ is of the form:
\begin{align*}
  &[\hat{N}^{(\sigma)}_{\hslash},\hat{a}_{\lambda}(f)]=-\hslash\, \hat{a}_{\lambda}(|k|^{2\sigma}f),    \\
  &[\hat{N}^{(\sigma)}_{\hslash},\hat{a}^{*}_{\lambda}(f)]=\hslash\, \hat{a}^{*}_{\lambda}  (|k|^{2\sigma}f).
\end{align*}
In a similar way to $\hat{N}_\hslash^{}(\delta)$ we introduce
\begin{align*}
  \hat{N}^{(\sigma)}_{\hslash}(\delta)=\hat{N}^{(\sigma)}_{\hslash}e^{-\delta \hat{N}_\hslash^{(\sigma)}} ,   \delta\geq 0, \sigma \geq 0.
\end{align*}
Let $\psi \in D(\hat H_\hslash^0)$, and define
\begin{align*}
  M^\sigma (t)= \langle\psi(t), \hat{N}^{(\sigma)}_{\hslash}(\delta) \psi(t)\rangle.
\end{align*}
The map $t\mapsto M^\sigma (t)$ is differentiable, with
\begin{align}
  \label{eq:18-s}
  \dot{M}^\sigma(t)=\frac{\ci}{\hslash}\langle\psi(t), [\hat{H}_{\hslash}, \hat{N}^{(\sigma)}_{\hslash}(\delta)] \psi(t)\rangle.
\end{align}
We set
\begin{align*}
  {\rm q}_{\delta,\hslash}^\sigma (\phi,\psi)=\frac{\ci}{\hslash}
  \langle\phi, [\hat{H}_{\hslash}, \hat{N}^{(\sigma)}_{\hslash}(\delta)] \psi\rangle,
\end{align*}
for any $\phi,\psi\in \mathscr{C}_0^{\infty}({\mathbb R}^{dn})\otimes \Gamma_{\mathrm{fin}}$.
Let
${\rm q}^\sigma\equiv {\rm q}_{\delta,\hslash}^\sigma$.
Note that
$3/2-\sigma\geq \sigma-1/2$ if and only if $\sigma\leq 1$.
In what follows we assume that
$\hat H_\hslash+a>0$ and $\hat H_\hslash^0+b>0$.
Here $a$ and $b$ are introduced in Corollary \ref{s4.lemE}.
\begin{lemma}
  \label{lemma:3-s}
  Assume \eqref{eq:hypo0}, \eqref{eq:hypo1} and
  \emph{(}\hyperref[eq:hypo2]{$A_2^{(1)}$}\emph{)}. For $\sigma\in[0,1]$ there exists $C>0$ such   that for all $\delta>0$, $\hslash\in (0,1)$ and $\phi,\psi\in \mathscr{C}_0^{\infty}({\mathbb R}^{dn})\otimes   \Gamma_{\mathrm{fin}}$:
\begin{align*}
|  {\rm q}^\sigma (\phi,\psi)|
\leq
C\||\cdot|^{\sigma-1/2}\chi\|
(\|(\hat{H}_{\hslash}+1)^{1/2} \phi\|\|(\hat{N}_\hslash^{(\sigma)})^{1/2}\psi\|
+\|(\hat{H}_{\hslash}+1)^{1/2}  \psi\|\|(\hat{N}_\hslash^{(\sigma)})^{1/2}\phi\|)
\end{align*}
\end{lemma}
\begin{proof}
The commutator can be computed explicitly as follows
\begin{align*}
  [\hat{N}^{(\sigma)}_{\hslash}(\delta),\hat{H}_{\hslash}]=\frac{1}{2}\sum_{j=1}^n \sum_{\mu=1}^d\big([\hat{N}^{(\sigma)}_{\hslash}(\delta),{\rm D}_{\mu}^j]{\rm D}_{\mu}^j+{\rm D}_{\mu}^j [\hat{N}^{(\sigma)}_{\hslash}(\delta),{\rm D}_{\mu}^j]\big).
\end{align*}
Note that
\begin{align*}
&[e^{-\delta \hat{N}_\hslash^{(\sigma)}}, \hat{a}_{\lambda}(f)]=e^{-\delta \hat{N}_\hslash^{(\sigma)}}a_{\lambda}(\xi f),\\
&[e^{-\delta \hat{N}_\hslash^{(\sigma)}}, \hat{a}^*_{\lambda}(f)]=\hat{a}^*_{\lambda}(-\xi f)e^{-\delta \hat{N}_\hslash^{(\sigma)}},
\end{align*}
where
$$\xi=1-e^{-\delta\hslash|k|^{2\sigma}}.$$
We have
\begin{align*}
  [\hat{N}^{(\sigma)}_{\hslash}(\delta),{\rm D}_{\mu}^j]&=
  -\ci\hslash(
  \Pi_{j,\mu}^{\sigma} e^{-\delta \hat{N}_\hslash^{(\sigma)}} +
\hat{N}^{(\sigma)}_{\hslash}e^{-\delta \hat{N}_\hslash^{(\sigma)}} \hat{a}_{j,\mu} (\xi) +
\hat{N}^{(\sigma)}_{\hslash}\hat{a}_{j,\mu}^\ast (-\xi)
e^{-\delta \hat{N}_\hslash^{(\sigma)}})\\
  &=-\ci\hslash(X+Y+Z),
  \end{align*}
where
\begin{align*}
 &X=\Pi_{j,\mu}^\sigma e^{-\delta \hat{N}_\hslash^{(\sigma)}},\\
 &Y= \hat{N}^{(\sigma)}_{\hslash}e^{-\delta \hat{N}_\hslash^{(\sigma)}} \hat{a}_{j,\mu} (\xi),\\
 &Z= \hat{N}^{(\sigma)}_{\hslash}\hat{a}_{j,\mu}^\ast (-\xi) e^{-\delta \hat{N}_\hslash^{(\sigma)}},
\end{align*}
$\Pi_{j}^\sigma$ is given by
\begin{align*}
  \Pi_{j}^\sigma=\ci\sum_{\lambda=1}^{d-1}\int_{{\mathbb R}^d}^{}
  |k|^{2\sigma -\frac{1}{2}}\epsilon_\lambda(k)(\chi_j (k)e^{-2\pi\ci k\cdot \hat{q}_j}a_\lambda^{*}(k)-\bar{\chi }_j(k)e^{2\pi\ci k\cdot \hat{q}_j}a_\lambda(k))  \mathrm{d}k,
\end{align*}
and
\begin{align*}
&\hat{a}_j(\xi)=
\ci\sum_{\lambda=1}^{d-1}\int_{{\mathbb R}^d}^{}
  |k|^{-\frac{1}{2}}\epsilon_\lambda(k)\chi_j (k)(1-e^{-\hslash\delta|k|^{2\sigma}})e^{2\pi\ci k\cdot \hat{q}_j}\hat{a}_\lambda(k)\,\mathrm{d}k,\\
  &\hat{a}_j^{*}(-\xi)=
\ci\sum_{\lambda=1}^{d-1}\int_{{\mathbb R}^d}^{}
  |k|^{-\frac{1}{2}}\epsilon_\lambda(k)\chi_j (k)(e^{-\hslash\delta|k|^{2\sigma}}-1)e^{-2\pi\ci k\cdot \hat{q}_j}\hat{a}^{*}_\lambda(k)\mathrm{d}k.
\end{align*}
We use the same trick for $Y$ and $Z$ as we used in the estimate on $\hat N_\hbar e^{-\delta\hat N_\hbar}$.
We have
\begin{align*}
\hat{N}^{(\sigma)}_{\hslash}e^{-\delta \hat{N}_\hslash^{(\sigma)}} \hat{a}_{j,\mu} (\xi )=
\delta \hat{N}^{(\sigma)}_{\hslash}e^{-\delta \hat{N}_\hslash^{(\sigma)}} \hat{a}_{j,\mu} (\xi/\delta),
\end{align*}
where
we note that
$\sup_{\delta\geq0} \|\delta \hat{N}^{(\sigma)}_{\hslash}e^{-\delta \hat{N}_\hslash^{(\sigma)}} \|\leq 1/e$
and
$(1-e^{-\hslash\delta |k|^{2\sigma}})/ \delta \leq \hslash |k|^{2\sigma}$.
Let
\begin{align*}
  {\rm q}^\sigma (\phi,\psi)=
\frac{1}{2}\sum_{j=1}^n \sum_{\mu=1}^d
\left\{  \langle (X+Y+Z)\phi, D_\mu^j \psi\rangle
  +
  \langle D_\mu^j\phi,  (X+Y+Z)\psi\rangle\right\}
\end{align*}
Note that
we have already seen  that
\begin{align*}
\sum_{j=1}^n \sum_{\mu=1}^d
\|D_\mu^j\psi\|
\leq C\|(\hat H_\hslash^0 +b)^{1/2}\psi\|
\leq C\|(\hat H_\hslash+a)^{1/2}\psi\|.
\end{align*}
It is straightforward to see that
\begin{align*}
&\frac{1}{2}\sum_{j=1}^n \sum_{\mu=1}^d
\{\langle Y\phi,D_\mu^j\psi\rangle+\langle D_\mu^j \phi, Y\psi\rangle\}\\
&\leq
C \||\cdot|^{2\sigma-1}\chi_j\|
(\|(\hat{N}_\hslash^{(\sigma)})^{1/2}\phi\|\|(\hat H_\hslash+a)\psi\|
+\|(\hat{N}_\hslash^{(\sigma)})^{1/2}\psi\|\|(\hat H_\hslash+a)\phi\|)
\end{align*}
We estimate $Z$:
\begin{align*}
&\frac{1}{2}\sum_{j=1}^n \sum_{\mu=1}^d
\{\langle Z\phi,D_\mu^j\psi\rangle+\langle D_\mu^j \phi, Z\psi\rangle\}\\
&=\langle\hat{a}_{j\mu}^{*}(-\xi)\hat{N}_\hslash^{(\sigma)} e^{-\delta \hat{N}_\hslash^{(\sigma)}}\phi, D_\mu^j\psi\rangle
+
\langle D_\mu^j\phi, \hat{a}_{j\mu}^{*}(-\xi) \hat{N}_\hslash^{(\sigma)} e^{-\delta \hat{N}_\hslash^{(\sigma)}}\psi\rangle\\
&+
\langle\hat{a}_{j\mu}^{*}(-|\cdot|^{2\sigma} \xi)e^{-\delta \hat{N}_\hslash^{(\sigma)}}\phi, D_\mu^j\psi\rangle
+
\langle D_\mu^j\phi, \hat{a}_{j\mu}^{*}(-|\cdot|^{2\sigma}\xi)e^{-\delta \hat{N}_\hslash^{(\sigma)}}\psi\rangle.
\end{align*}
Taking the Coulomb gauge condition we obtain that
$$
\sum_{\mu=1}^d
\hat{a}_{j\mu}(-\xi) D_\mu^j
=
\sum_{\mu=1}^d
 D_\mu^j\hat{a}_{j\mu}(-\xi)$$ and
\begin{align*}
&\left|
\frac{1}{2}\sum_{j=1}^n \sum_{\mu=1}^d
\langle \hat{N}_\hslash^{(\sigma)} e^{-\delta \hat{N}_\hslash^{(\sigma)}}\phi, \hat{a}_{j\mu}(-\xi)D_\mu^j\psi\rangle
+
\langle \hat{a}_{j\mu}(-\xi)D_\mu^j\phi, \hat{N}_\hslash^{(\sigma)} e^{-\delta \hat{N}_\hslash^{(\sigma)}}\psi\rangle\right|
\\
&=
\left|
\frac{1}{2}\sum_{j=1}^n \sum_{\mu=1}^d
\langle D_\mu^j\delta \hat{N}_\hslash^{(\sigma)} e^{-\delta \hat{N}_\hslash^{(\sigma)}}\phi, \hat{a}_{j\mu}(-\xi/\delta)\psi\rangle
+
\langle\hat{a}_{j\mu}(-\xi/\delta)\phi, D_\mu^j\delta \hat{N}_\hslash^{(\sigma)} e^{-\delta \hat{N}_\hslash^{(\sigma)}}\psi\rangle\right|\\
&\leq
C
\||\cdot|^{\sigma-1/2}\chi_j\|
(\|(\hat H_\hslash^0 +b)^{1/2} \delta \hat{N}_\hslash^{(\sigma)} e^{-\delta \hat{N}_\hslash^{(\sigma)}}\phi\|
\|(\hat{N}_\hslash^\sigma)^{1/2}\psi\|+
\|(\hat H_\hslash^0 +b)^{1/2} \delta \hat{N}_\hslash^{(\sigma)} e^{-\delta \hat{N}_\hslash^{(\sigma)}}\psi\|
\|(\hat{N}_\hslash^{(\sigma)})^{1/2}\phi\|)\\
&\leq
C\||\cdot|^{\sigma-1/2}\chi_j\|
(\|(\hat H_\hslash^0 +b)^{1/2} \phi\|\|(\hat{N}_\hslash^{(\sigma)})^{1/2}\psi\|
+\|(\hat H_\hslash^0 +b)^{1/2}  \phi\|\|(\hat{N}_\hslash^{(\sigma)})^{1/2}\psi\|).
\end{align*}
Similarly
\begin{align*}
&\frac{1}{2}\sum_{j=1}^n \sum_{\mu=1}^d
\langle e^{-\delta \hat{N}_\hslash^{(\sigma)}}\phi, \hat{a}_{j\mu}(-|\cdot|^{2\sigma}\xi)D_\mu^j\psi\rangle
+
\langle \hat{a}_{j\mu}(-|\cdot|^{2\sigma}\xi)D_\mu^j\phi, e^{-\delta \hat{N}_\hslash^{(\sigma)}}\psi\rangle\\
&=
\frac{1}{2}\sum_{j=1}^n \sum_{\mu=1}^d
\langle D_\mu^j e^{-\delta \hat{N}_\hslash^{(\sigma)}}\phi, \hat{a}_{j\mu}(-|\cdot|^{2\sigma}\xi)\psi\rangle
+
\langle\hat{a}_{j\mu}(-|\cdot|^{2\sigma}\xi)\phi, D_\mu^j e^{-\delta \hat{N}_\hslash^{(\sigma)}}\psi\rangle\\
&\leq
C
\||\cdot|^{\sigma-1/2}\chi_j\|
(\|(\hat H_\hslash^0 +b)^{1/2}  e^{-\delta \hat{N}_\hslash^{(\sigma)}}\phi\|
\|(\hat{N}_\hslash^{(\sigma)})^{1/2}\psi\|+
\|(\hat H_\hslash^0 +b)^{1/2}  e^{-\delta \hat{N}_\hslash^{(\sigma)}}\psi\|
\|(\hat{N}_\hslash^{(\sigma)})^{1/2}\phi\|)\\
&\leq
C\||\cdot|^{\sigma-1/2}\chi_j\|
(\|(\hat H_\hslash^0 +b)^{1/2} \phi\|\|(\hat{N}_\hslash^{(\sigma)})^{1/2}\psi\|
+\|(\hat H_\hslash^0 +b)^{1/2}  \phi\|\|(\hat{N}_\hslash^{(\sigma)})^{1/2}\psi\|).
\end{align*}
Finally we estimate $X$.
The result is
\begin{align*}
&\frac{1}{2}\sum_{j=1}^n \sum_{\mu=1}^d
\left\{  \langle X\phi, D_\mu^j \psi\rangle
  +
  \langle D_\mu^j\phi,  X\psi\rangle\right\}  \\
&\leq
C\||\cdot|^{\sigma-1/2}\chi_j\|
(\|(\hat H_\hslash^0 +b)^{1/2} \phi\|\|(\hat{N}_\hslash^{(\sigma)})^{1/2}\psi\|
+\|(\hat H_\hslash^0 +b)^{1/2}  \phi\|\|(\hat{N}_\hslash^{(\sigma)})^{1/2}\psi\|).
\end{align*}
Then
\begin{align*}
|  {\rm q}^\sigma (\phi,\psi)|
\leq
C\||\cdot|^{\sigma-1/2}\chi_j\|
(\|(\hat H_\hslash+a)^{1/2} \phi\|\|(\hat{N}_\hslash^{(\sigma)})^{1/2}\psi\|
+\|(\hat H_\hslash+a)^{1/2}  \psi\|\|(\hat{N}_\hslash^{(\sigma)})^{1/2}\phi\|).
\end{align*}
Here $C$ is independent of $\hslash$ and $\delta$.
\end{proof}
With the aid of Lemma \ref{lemma:3-s},
we prove the following bound on $\dot{M}^\sigma(t)$.
\begin{lemma}
  \label{lemma:4-s}
 Assume \eqref{eq:hypo0}, \eqref{eq:hypo1} and
  \emph{(}\hyperref[eq:hypo2]{$A_2^{(1)}$}\emph{)}. For $\delta>0$ and $\sigma\in[0,1]$ if we set $M^\sigma (t)=\langle\psi(t), \hat{N}^\sigma_{\hslash}(\delta)\psi(t)\rangle$, then there exists $C>0$ such that for any $t\in {\mathbb R}$,
  \begin{align*}
    \dot{M}^\sigma(t)\leq C \langle\psi, (\hat H_{\hslash}+a)\psi\rangle + M^\sigma (t).
  \end{align*}
\end{lemma}
\begin{proof}
  From Lemma \ref{lemma:3-s}, the lemma follows.
  \end{proof}
Let $Q(T)$ be the form domain of $T$.
Gr\"onwall's inequality then yields the sought result of propagation.

\begin{prop}
  \label{prop:3-s}
  Assume \eqref{eq:hypo0}, \eqref{eq:hypo1} and
  \emph{(}\hyperref[eq:hypo2]{$A_2^{(1)}$}\emph{)}. Let $\sigma\in[0,1]$ and $\psi \in
  Q(\hat{H}^0_\hslash)\cap Q(\hat{N}_\hslash^{(\sigma)})$,
then $e^{-\ci\frac{t}{\hslash}\hat{H}_{\hslash}}\psi \in
  Q(\hat{N}^{(\sigma)}_{\hslash})$.
\end{prop}
\begin{proof}
Let $\delta>0$ and $\psi\in D(\hat H^0_\hslash)$.
By Lemma \ref{lemma:4-s}, we have
  \begin{align*}
    M^\sigma (t)=\int_0^t\dot{M^\sigma}(s)  \mathrm{d}s + M^\sigma(0)\leq C\int_0^t(M^\sigma(s)+\theta)  \mathrm{d}s + M^\sigma(0),
  \end{align*}
  where $\theta=\|(\hat H_{\hslash}+a)^{1/2}\psi\|^2$. Then Gr\"onwall's inequality yields
  \begin{align*}
    M^\sigma (t)\leq
    ( \|(\hat{N}^{(\sigma)}_{\hslash}(\delta))^{1/2}\psi\|^2+
    \|(\hat H_{\hslash}+a)^{1/2}\psi\|^2)e^{C|t|},
  \end{align*}
  where $C$ does not depend on $\hslash$ and $\delta$.
 This inequality can be extended to $\psi\in Q(\hat{H}_\hslash)\cap Q(\hat{N}_\hslash^{(\sigma)})$.
Furthermore
  for all $\psi \in Q(\hat{H}_\hslash)\cap Q(\hat{N}_\hslash^{(\sigma)})$,
  \begin{align*}
    \lim_{\delta\downarrow 0}
    \|(\hat{N}^{(\sigma)}_{\hslash}(\delta))^{1/2}\psi(t)\|^2
    =\lim_{\delta\downarrow0}\int_0^{\infty}\lambda e^{-\delta\lambda}  \mathrm{d}E_{\lambda}.
  \end{align*}
  Hence, by the monotone convergence theorem, it follows that $\psi(t)\in Q(\hat{N}^\sigma_{\hslash})$, and
  \begin{align*}
  \|(\hat{N}^{(\sigma)}_{\hslash})^{1/2}\psi(t)\|^2\leq
   \big(
     \|(\hat{N}^{(\sigma)}_{\hslash})^{1/2}\psi\|^2
+
    \|(\hat H_{\hslash}+a)^{1/2}\psi\|^2)e^{C|t| }.
  \end{align*}
Then the proposition follows.
\end{proof}
From the proof of Proposition \ref{prop:3-s}
 we obtain the corollary below.

 \begin{cor}
  \label{cor:2-s}
  Assume \eqref{eq:hypo0}, \eqref{eq:hypo1} and
  \emph{(}\hyperref[eq:hypo2]{$A_2^{(1)}$}\emph{)}. For $\sigma\in[0,1]$ there exist  two constants $K,C>0$ such that for all $\hslash\in (0,1)$,  all $\psi \in Q(\hat{H}_\hslash^0)\cap Q(\hat{N}_\hslash^{(\sigma)})$,
  and all $t\in {\mathbb R}$,
  \begin{align*}
    \| (\hat{N}_{\hslash}^{(\sigma)})^{1/2}\psi(t)  \|\leq K\big(\| (\hat{N}_{\hslash}^{(\sigma)})^{1/2}\psi   \|+\| (\hat{H}_{\hslash}^0+1)^{1/2}\psi   \|\big)e^{C|t| }.
  \end{align*}
\end{cor}

\subsubsection{Particles position estimates}
\label{sec:position-estimates}
We end this section by proving  localization estimates for the particles along the Pauli--Fierz quantum dynamics. In fact, we show below that if an initial state
is localized in position then for all times it will stay localized uniformly with respect to the semiclassical parameter $\hslash\in(0,1)$.

\begin{lemma}
\label{lemma:q2}
Assume \eqref{eq:hypo0}, \eqref{eq:hypo1} and \emph{(}\hyperref[eq:hypo2]{$A_2^{(1)}$}\emph{)}.
Then there exist constants $K,C>0$ such that for all $\Psi\in  Q(\hat{H}^0_\hslash)\cap  Q(\hat q^2)$, all $t\in\R$ and all $\hslash\in(0,1)$:
\begin{eqnarray*}
&(i)& \qquad \langle e^{-\ci \frac{t}{\hslash} \hat{H}^0_\hslash}\Psi , ~ \hat q^2 ~ e^{-\ci \frac{t}{\hslash} \hat{H}^0_\hslash}\Psi \rangle \leq K \langle \Psi , ~ (\hat H_\hslash^0+\hat q^2+1) ~ \Psi \rangle \,e^{C|t|} \,,\\
&(ii)& \qquad  \langle e^{-\ci \frac{t}{\hslash} \hat{H}_\hslash}\Psi , ~ \hat q^2 ~ e^{-\ci \frac{t}{\hslash} \hat{H}_\hslash}\Psi \rangle \leq K \langle \Psi , ~ (\hat H_\hslash^0+\hat q^2+1) ~ \Psi \rangle \,e^{C|t|}\,.
\end{eqnarray*}
\end{lemma}
\begin{proof}
The argument is based on Gr\"onwall's inequality and is similar to the one employed in the previous proofs of  Subsection \ref{sec:regularity}. So, we just indicate here the key point. In particular, one can use the work of Faris and Lavine \cite[Theorem 2]{FL} (see also \cite[Theorem B.1]{AmBr}).
So, it is enough to estimate as quadratic forms on $Q(\hat{H}^0_\hslash)\cap  Q(\hat q^2)$ the following commutators,
\begin{equation}
\label{eq:pq2}
T_1:=\ci [\hat p^2_j,\hat q^2_j]= 2 \hslash \sum_{\nu=1}^{d}  \hat p_j^\nu \hat q_j^\nu = 2 \hslash \, \hat p_j \cdot \hat q_j\,,
\end{equation}
and
\begin{equation}
\label{eq:crosspq}
T_2:=\ci [\hat p_j \cdot A_j(\hat q_j,\hat a),\hat q^2_j]= 2 \hslash \, \hat q_j\cdot A_j(\hat q_j,\hat a)\,.
\end{equation}
Since the right hand sides in \eqref{eq:pq2} and \eqref{eq:crosspq} are bounded respectively by
\begin{equation}
\hslash^{-1}T_1 \leq \hat p_j^2+\hat q_j^2\ \,,\qquad  \hslash^{-1}T_2 \leq \hat q_j^2+A_j(\hat q_j,\hat a)^2 \,,
\end{equation}
then by  Corollaries \ref{s4.lemE} and \ref{cor:1},
\begin{align*}
\theta_1(t):=\langle e^{-\ci \frac{t}{\hslash} \hat{H}_\hslash}\Psi , ~ \hat q^2 ~ e^{-\ci \frac{t}{\hslash} \hat{H}_\hslash}\Psi \rangle &\leq \theta_1(0)+t K_0 \langle \Psi , ~  (\hat H_\hslash^0+1) ~ \Psi \rangle+ C_0 \int_0^t \theta_1(s) \,ds\\
&\leq K \langle \Psi , ~  (\hat H_\hslash^0+\hat q^2+1) ~ \Psi \rangle \, e^{C t}\,,
\end{align*}
for all $t\geq 0$  and some constants $K_0,C_0,K,C$ independent of time and $\hslash$ (the case $t\leq 0$ follows by a simple change of variable). This proves (ii). On the other hand a similar inequality
\begin{align*}
\theta_0(t):=\langle e^{-\ci \frac{t}{\hslash} \hat{H}^0_\hslash}\Psi , ~ \hat q^2 ~ e^{-\ci \frac{t}{\hslash} \hat{H}^0_\hslash}\Psi \rangle &\leq
\theta_0(0)+t K_0 \langle \Psi , ~  (\hat H_\hslash^0+1) ~ \Psi \rangle+ C_0 \int_0^t \theta_0(s) \,ds\\
&\leq K \langle \Psi , ~  (\hat H_\hslash^0+\hat q^2+1) ~ \Psi \rangle \, e^{C t}\,,
\end{align*}
 holds true for some possibly different constants. Thus proving  (i).
\end{proof}

\section{Derivation of the characteristic equation}
\label{sec.CH}

Our approach for the derivation of classical electrodynamics from quantum electrodynamics  is based on the analysis  of  time-dependent Wigner measures $(\mu_t)_{t\in\R}$ of  the quantum evolved states $(\varrho_\hslash(t))_{\hslash\in (0,1)}$ given in Theorem \ref{thm:main-2}. Due to the complexity generated by the particles-field interactions, it is not possible to compute directly or explicitly these Wigner measures. Therefore, it makes more sense to write a dynamical equation that characterizes the time course of such probabilities. In particular, one  obtains such relation if one differentiates  with respect to time the characteristic or generating  functions of these Wigner  measures $(\mu_t)_{t\in\R}$. Such equation is a crucial step in our analysis; it will be given in Proposition \ref{s4:propChar} and named \emph{characteristic equation}.

\bigskip
\paragraph{\it Weyl--Heisenberg operators:}
The Weyl--Heisenberg translation operators on the particle variables are defined as the unitary operators,
\begin{equation}
\label{eq:23}
\hat{T}(p,q)=e^{\ci (p\cdot \hat{q}-q\cdot \hat{p})}, \quad \forall (p,q)\in \mathds{R}^{dn}\times \mathds{R}^{dn}\,.
\end{equation}
The space $\mathds{R}^{dn}\times \mathds{R}^{dn}\equiv \mathds{C}^{dn}$ is endowed with a  complex structure
$$
z:= z(p,q)=q+\ci p, \quad  \forall (p,q)\in \mathds{R}^{dn}\times \mathds{R}^{dn}\,,
$$
and a  canonical symplectic  form
$$
\Imm\langle z, z'\rangle= q\cdot p'-p\cdot q'\,, \quad \forall (p,q),(p',q')\in \mathds{R}^{2dn}\,.
$$
Hence, Weyl--Heisenberg operators can  also be labelled  by the phase-space points $z\in\mathds{C}^{dn}$,
\begin{equation}
\label{eq:weyl-Heis}
\hat{T}(z) =\hat{T}(p,q)=e^{2\ci\Imm\langle\hat{q}+\ci\hat{p},z\rangle}\,,
\end{equation}
satisfying  the particles  \emph{Weyl commutation relations}:
\begin{eqnarray*}
&& \quad \hat{T}(z)  \; \hat{T}(z') =e^{-\ci \frac{\hslash }{2}\, \Imm\langle z, z'\rangle} \; \hat{T}(z+z')\,,\\
&& \quad \hat{T}(z)^*=   \hat{T}(z)^{-1}=  \hat{T}(-z)\,.
\end{eqnarray*}
Moreover, the commutations rules blow  hold true for any $z=q+\ci p\in\mathds{C}^{dn}$,
\begin{equation}
\label{eq:Tcc}
\begin{aligned}
\hat{T}(z)\; \hat{q} \;\hat{T}(z)^*&=\hat{q}-\hslash q\,,\\
\hat{T}(z) \;\hat{p} \;\hat{T}(z)^*&=\hat{p} -\hslash p\,.
\end{aligned}
\end{equation}
Similarly,  one  defines also the Weyl operators on the Fock space $\Gamma_{\mathrm{s}}\bigl(L^2 (\mathds{R}^d, \mathds{C}^{d-1}  )\bigr) $ for any $f\in L^2(\mathds{R}^d, \mathds{C}^{d-1})$ according to the formula,
\begin{equation}
\label{eq:weyl}
W(f)= e^{\frac{\ci}{\sqrt{2}} (\hat{a}(f)+\hat{a}^*(f))}\,,
\end{equation}
where $\hat{a}^*$ and $\hat{a}$ are the creation-annihilation  operators in \eqref{eq:an-cr1}   satisfying  \eqref{eq:an-cr2}. With these notations the  following \emph{Weyl commutations relations} for the field are satisfied,
\begin{equation*}
W(f) W(g)=e^{-\ci \frac{\hslash }{2}\, \Imm\langle f, g\rangle_{\mathfrak{H}^0}} \, W(f+g)\,, \quad \forall f,g \in  \mathfrak{H}^0=L^2(\mathds{R}^d, \mathds{C}^{d-1})\,,
\end{equation*}
where $\Imm\langle \cdot, \cdot\rangle_{\mathfrak{H}^0}$ denotes the imaginary part of the scalar product on $\mathfrak{H}^0=L^2(\mathds{R}^d, \mathds{C}^{d-1})$.
Since the total phase-space of the particles-field system is
\begin{equation}
\label{eq:phase-space}
X^0= \mathds{R}^{2dn}\oplus \mathfrak{H}^0\equiv
\mathds{C}^{dn} \oplus \mathfrak{H}^0\,,
\end{equation}
it follows that the mapping
\begin{equation}
\label{eq:repWeyl}
(z, \alpha)\in X^0 \longmapsto  \mathcal{W}(z,\alpha):=\hat{T}(z)\otimes W(\alpha)\,,
\end{equation}
defines a strongly continuous  irreducible representation of the Weyl's commutation relations over the Hilbert space  $\mathscr H
=L^2 (\mathds{R}^{dn}, \mathds{C} )\otimes\Gamma_{\mathrm{s}}\bigl(\mathfrak{H}^0\bigr) $. Note that we will use indifferently the notations  $\mathcal{W}(z,\alpha)$, $\mathcal{W}(p,q,\alpha)$ or $\mathcal{W}(\xi)$  when $z\equiv(p,q)$ and $\xi=(z,\alpha)$.

\bigskip
\paragraph{\it Coherent states:}
Consider  the normalized Gaussian function
$$
\varphi_0(x) = (\pi\hslash)^{-\frac{dn}{4}} e^{- x^2/2\hslash}
\in L^2(\R^{dn},\mathrm{d}x)\,,
$$
then the  coherent vector, centered  on $z_0\in \mathds{C}^{dn}$,  is defined as
\begin{equation}
\label{eq:cohphi}
\varphi_{z_0}:= \hat{T}\bigg(\frac{\sqrt{2}}{\ci\hslash} z_0\bigg) \varphi_0\,.
\end{equation}
Similarly, for any $\alpha\in \mathfrak{H}^0$ one defines coherent vectors in the Fock spaces as
\begin{equation}
\label{eq:Wo}
 W\bigg(\frac{\sqrt{2}}{\ci\hslash} \alpha\bigg)\,\Omega\,,
\end{equation}
with $\Omega$ is the vacuum vector in $\Gamma_s(\mathfrak{H}^0)$. Both vectors
\eqref{eq:cohphi} and \eqref{eq:Wo} are normalized. For any $u_0=(p_0,q_0,\alpha_0)\in X^0$ such that  $z_0\equiv(p_0,q_0)$,  the one rank orthogonal projection
\begin{equation}
\label{eq:cohdef}
\mathcal C_\hslash(u_0)=\bigg| \varphi_{z_0}\otimes W\bigg(\frac{\sqrt{2}}{\ci\hslash} \alpha_0\bigg)\,\Omega \bigg\rangle
\bigg\langle \varphi_{z_0}\otimes W\bigg(\frac{\sqrt{2}}{\ci\hslash} \alpha_0\bigg)\,\Omega \bigg|\,,
\end{equation}
 defines a $\hslash$-scaled family of density matrices over $\mathscr H$ called the coherent states.

\bigskip
\paragraph{\it Examples of Wigner measures:}
The following  result proved in  \cite[Theorem 6.2]{ammari2008ahp}, shows that under mild assumptions any family of quantum states admits at least one Wigner measure.

\begin{prop}
\label{wig}
Any family of density matrices  $(\varrho_\hslash)_{\hslash\in (0,1)}$   on  $\mathscr{H}$  satisfying the condition:
\begin{eqnarray}
\label{eq:wigcd}
\exists \delta>0,  C_\delta>0: \; \forall \hslash\in (0,1) ,
\quad \Tr\big[ \varrho_\hslash \;(\hat{p}^2+\hat{q}^2+\hat{N}_\hslash)^\delta\big]<C_\delta\,,
\end{eqnarray}
has a non trivial set of Wigner (probability) measures, \emph{i.e.}:  $\mathcal{M}(\varrho_\hslash, \hslash\in (0,1))\neq\emptyset$. Moreover, any
$\mu\in \mathcal{M}(\varrho_\hslash, \hslash\in (0,1))$ satisfies
$$
\int_{X^0} \|u\|_{X^0}^{2\delta} \; d\mu(u)<\infty\,.
$$
\end{prop}

On the other hand, the set of Wigner measures for  family of coherent states is remarkably simple and it is given below.
\begin{lemma}
\label{lemma:coh}
Consider for any  $u_0=(p_0,q_0,\alpha_0)\in X^0$  the family of coherent states $(\mathcal C_\hslash(u_0))_{\hslash \in (0,1)}$ defined in \eqref{eq:cohdef}.  Then
\begin{equation}
\mathcal{M}\big(\mathcal C_\hslash(u_0), \hslash\in(0,1)\big)=\left\{\delta_{u_0}\right\}\,,
\end{equation}
 where $\delta_{u_0}$ is the Dirac measure centered on $u_0$.
Furthermore, if $u_0=(p_0,q_0,\alpha_0)\in X^\sigma$ for some $\sigma\in[\frac 1 2, 1]$, then the family $(\mathcal C_\hslash(u_0))_{\hslash\in(0,1)}$ satisfies  the assumptions \emph{(}\hyperref[eq:Ass0]{$S_0^{(1)}$}\emph{)}, \emph{(}\hyperref[eq:Ass1]{$S_1^{(1)}$}\emph{)} and \eqref{eq:Ass2}.
\end{lemma}
\begin{proof}
According to Definition \ref{def:wigner}, in order to determine the set of Wigner measures of the family of states $(\mathcal C_\hslash(u_0))_{\hslash \in (0,1)}$ it is enough to explicitly compute the limit
$$
\lim_{\hslash \to 0} \Tr\big[ \mathcal C_{\hslash}(u_0) \;\mathcal{W}\big(\pi q,-\pi p,\sqrt{2}\pi \alpha\big)\big]\,,
$$
for any $\zeta=(\pi q,-\pi p,\sqrt{2}\pi \alpha)\in X^0$. Indeed, using Weyl's commutation relations one shows
\begin{equation}
\label{eq:24}
\begin{aligned}
\lim_{\hslash \to 0} \langle \mathcal{C}_\hslash(u_0); \mathcal{W}(\zeta)\;  \mathcal{C}_\hslash(u_0)\rangle_{\mathscr H}&= e^{2\ci\pi \Ree\langle \zeta, u_0\rangle_{{X}^{0}}} \\
&= \mathscr{F}^{-1}[\delta_{u_0}] (\zeta)\,.
\end{aligned}
\end{equation}
Furthermore, if $u_0=(p_0,q_0,\alpha_0)\in X^\sigma$ then one can check that
\begin{equation}
\label{eq:cohbwik}
\begin{split}
\Tr\big[ \mathcal C_\hslash(u_0) \; {\rm d}\Gamma(|k|)\big] &=\langle \alpha_0, |k| \alpha_0\rangle_{L^2}=
\|\alpha_0\|^2_{\dot{\mathfrak{H}}^{1/2}}\,,\\
 \Tr\big[ \mathcal C_\hslash(u_0) \; {\rm d}\Gamma(|k|^{2\sigma})\big] &= \langle \alpha_0, |k|^{2\sigma} \alpha_0\rangle_{L^2}=\|\alpha_0\|^2_{\dot{\mathfrak{H}}^{\sigma}}\,,\\
 \Tr\big[ \mathcal C_\hslash(u_0) \; \hat{N}_\hslash\big] &= \| \alpha_0\|_{L^2}^2\,.
 \end{split}
\end{equation}
Such identities can be proved using  commutation relations in Fock spaces (See for instance \cite[Prop.~2.7 and 2.10]{ammari2008ahp}). Similarly, using \eqref{eq:Tcc} one also proves that
\begin{equation}
\label{eq:gaus}
\begin{split}
 \Tr\big[ \mathcal C_\hslash(u_0) \; \hat  q^2\big] &=
 \langle \varphi_0, (\hat q-\ci \sqrt{2} q_0)^2\varphi_0\rangle\\
 &= \langle \varphi_0, \hat q^2\varphi_0\rangle-2 q_0^2\,,\\
\end{split}
\end{equation}
and
\begin{equation}
\label{eq:gausp}
\begin{split}
  \Tr\big[ \mathcal C_\hslash(u_0) \; \hat  p^2\big] &=
 \langle \varphi_0, (\hat p-\ci \sqrt{2} p_0)^2\varphi_0\rangle\\
 &= \langle \varphi_0, \hat p^2\varphi_0\rangle-2 p_0^2\,,
 \end{split}
\end{equation}
with $\varphi_0$ is the Gaussian function in \eqref{eq:cohphi} and $z_0=q_0+\ci p_0\in\C^{dn}$.
 Since $u_0=(p_0,q_0,\alpha_0)\in X^\sigma$ for some $\sigma\in [\frac 1 2,1]$,
 one notices that all the quantities in  \eqref{eq:cohbwik}, \eqref{eq:gaus} and \eqref{eq:gausp} are bounded uniformly with respect to $\hslash\in(0,1)$.
\end{proof}

   A simple observation shows that  any Borel probability measure $\mu_0$ over $X^0$ is  a Wigner measure of at least one given family of quantum states.

\begin{lemma}
\label{lemma:examwig}
Consider for any $\mu_0\in\mathfrak{P}(X^0)$ the family of density matrices
\begin{equation}
\label{eq:cohmu}
\varrho_\hslash= \int_{X^0} \mathcal C_\hslash(u) \,\mathrm{d}\mu_0(u)\,,
\end{equation}
where $\mathcal C_\hslash(u)$ is the coherent state defined  in \eqref{eq:cohdef} and centered on $u=(p,q,\alpha)\in X^0$. Then
\begin{equation}
\label{eq:stwigcohmu}
\mathcal{M}\big(\varrho_\hslash, \hslash\in(0,1)\big)=\{\mu_0\}\,.
\end{equation}
Furthermore, suppose  that
\begin{align}
\label{eq:cdmes}
\int_{X^0} \|u\|^2_{X^\sigma}\;\mathrm{d}\mu_0 <\infty\,,
\end{align}
then the family of density matrices $(\varrho_\hslash)_{\hslash\in(01)}$ satisfies the assumptions  \emph{(}\hyperref[eq:Ass0]{$S_0^{(1)}$}\emph{)}, \emph{(}\hyperref[eq:Ass1]{$S_1^{(1)}$}\emph{)} and \eqref{eq:Ass2}.
\end{lemma}
\begin{proof}
Using \eqref{eq:24} and  \eqref{eq:cohmu},
\begin{equation}
\label{eq:varrho}
\begin{split}
\lim_{\hslash \to 0} \Tr\big[ \mathcal \varrho_{\hslash} \;\mathcal{W}\big(\pi q,-\pi p,\sqrt{2}\pi \alpha\big)\big]&=
\int_{X^0}  e^{2\ci\pi \Ree\langle \xi, u\rangle_{{X}^{0}}} \, \mathrm{d}\mu_0(u)\\
&= \mathscr{F}^{-1}[\mu_0] (\xi)\,,
\end{split}
\end{equation}
for all $\xi=(p,q,\alpha)\in X^0$. Hence,  according to Definition \ref{def:wigner} the first statement \eqref{eq:stwigcohmu} is proved since the right hand side  of \eqref{eq:varrho}  uniquely determines  the probability measure $\mu_0$.
The second statement is a consequence of \eqref{eq:cohbwik}--\eqref{eq:gausp} and the assumption \eqref{eq:cdmes}.
\end{proof}

Finally, we note that the set of Wigner measures resulting from a family of quantum states is not in general a singleton.
However, for any  family $(\varrho_\hslash)_{\hslash\in \mathscr{I}}$ it is  always possible to find a subset $\mathscr{I}_0\subset \mathscr{I}$,
$0\in \overline{\mathscr{I}_0}$, such that $\mathcal{M} (\varrho_\hslash, \hslash\in \mathscr{I}_0)$ is a singleton.

\subsection{Duhamel's formula}
\label{sec:Duhamel}

\subsubsection{Weyl--Heisenberg operator estimates}
\label{sec:Weyl-estimates}
Here we  highlight  the fact that  the Weyl--Heisenberg operators  keep invariant the form domain of the Pauli--Fierz Hamiltonian $\hat{H}_\hslash$.  These properties will be useful to establish a Duhamel formula for quantum evolved states.

\begin{lemma}
\label{lem:estN}
For any $\alpha\in \mathfrak{H}^{0}$ there exists  a constant $C>0$ such that for any $\hslash\in(0,1)$ and any $\Psi\in {D}(\hat{N}_\hslash)$,
\begin{eqnarray}
\| (\hat{N}_\hslash)^{1/2} \, W(\alpha) \Psi\|_{\Gamma_s(L^2(\mathds{R}^d,\mathds{C}^{d-1}))} &\leq& C \,
\|(\hat{N}_\hslash+1)^{1/2}\, \Psi\|_{\Gamma_s(L^2(\mathds{R}^d,\mathds{C}^{d-1}))}\,,\\
\| \hat{N}_\hslash \, W(\alpha) \Psi\|_{\Gamma_s(L^2(\mathds{R}^d,\mathds{C}^{d-1}))} &\leq& C \,
\|\big(\hat{N}_\hslash+1\big)\, \Psi\|_{\Gamma_s(L^2(\mathds{R}^d,\mathds{C}^{d-1}))}\,.
\end{eqnarray}
\end{lemma}
\begin{proof}
The above estimate  is a consequence of the commutator expansion formula,
\begin{equation*}
W(\alpha)^* \, \hat{N}_\hslash \, W(\alpha)= \hat{N}_\hslash+\frac{\ci\hslash}{\sqrt{2}} \hat a^*(\alpha)-\frac{\ci\hslash}{\sqrt{2}}  \hat a(\alpha)+\frac{\hslash^2}{2} \|\alpha\|_2\,,
\end{equation*}
and  $\hslash$-scaled standard number estimates (see ,\emph{e.g.},~\cite[Lemma 2.5]{ ammari2008ahp}).
\end{proof}

\begin{lemma}
\label{s3.lemW}
For any $\alpha\in \mathfrak{H}^{1/2}$ there exists  a constant $C>0$ such that for any $\hslash\in(0,1)$ and any $\Psi\in
{D}(\hat{H}^0_\hslash)$,
\begin{equation}
\| (\hat{H}^0_\hslash)^{1/2}  \; W(\alpha) \,\Psi\|_{\mathscr H} \leq C \,
\|\big(\hat{H}^0_\hslash+1\big)^{1/2}\, \Psi\|_{\mathscr H}\,.
\end{equation}
\end{lemma}
\begin{proof}
It is a simple  consequence of the fact that
\begin{equation*}
(\hat{H}^0_\hslash+1)^{-1/2} \,\big[\hat a^\sharp(\alpha), \hat{H}^0_\hslash\big]\, (\hat{H}^0_\hslash+1)^{-1/2}\,,
\end{equation*}
is bounded uniformly with respect to $\hslash\in(0,1)$ and where $a^\sharp$ stands for both the creation and the annihilation operators.
\end{proof}

Similar estimates also hold for the Weyl--Heisenberg operators  $\hat T(\cdot)$.
\begin{lemma}
\label{lem:Tsmo}
For any $z\in\mathds{C}^{dn}$ there exists $C>0$  such that for any $\hslash\in(0,1)$ and any $\Psi\in
{D}((\hat{p}^2+\hat{q}^2)^{1/2})$,
\begin{equation}
\label{eq:26}
\| (\hat{p}^2+\hat{q}^2)^{1/2}  \; \hat{T}(z) \,\Psi\|_{L^2(\mathds R^{dn})} \leq C \,
\|(\hat{p}^2+\hat{q}^2+1)^{1/2} \,\Psi\|_{L^2(\mathds R^{dn})}\,.
\end{equation}
\end{lemma}
\begin{proof}
In order to prove the above estimate, it is enough to evaluate the commutator between the generator of the  Weyl--Heisenberg operator $\hat T(\cdot)$ and the harmonic oscillator $\hat{p}^2+\hat{q}^2$. An exact calculus gives,
\begin{eqnarray*}
\ci\big[ \hat{p}^2+\hat{q}^2; p\cdot \hat{q}-q\cdot \hat{p}\big]
&=&2 \hslash \;(q\cdot \hat{q}+p\cdot \hat{p})\,.
\end{eqnarray*}
Moreover, the right hand side is bounded uniformly with respect  to $\hslash\in(0,1)$ by $\hat{p}^2+\hat{q}^2$ in the quadratic form sense.
\end{proof}

\subsubsection{Duhamel formula}
\label{sec:int-formula}
The argument behind the Duhamel's formula in Proposition \ref{lemma:Duhamel}  is quite general. To highlight the main ingredients in this step, we list the assumptions at our disposal. For $\hslash\in(0,1)$:
\begin{itemize}
\item[(i)]  $\xi\in X^0\to \mathcal{W}(\xi)$ is a strongly continuous representation of the Weyl commutation relations on the Hilbert space $\mathscr{H}$.
\item[(ii)] $(\hat{H}_\hslash, \hat{H}_\hslash^0)$ is a couple of self-adjoint operators on $\mathscr{H}$  such that
 $\hat{H}_\hslash^0$ is non-negative and the form domain $Q(\hat{H}_\hslash)={D}(|\hat{H}_\hslash|^{\frac{1}{2}})$ coincides with the one of
 $\hat{H}^0_\hslash$.
\item[(iii)] The form domain $Q(\hat{H}^0_\hslash)$  is invariant with respect to the Weyl--Heisenberg operators $\mathcal{W}(\xi)$
 for all $\xi$ in  a dense subset $X^{1/2}$  of the phase-space $X^0$.
\end{itemize}
\medskip
Let  $(\varrho_\hslash)_{\hslash\in (0,1)}$ be a family of density matrices on $\mathcal{H}$ and define:
\begin{equation}
\label{eq:25}
\varrho_\hslash(t)= e^{-\ci \frac{t}{\hslash} \hat{H}_\hslash}  \, \varrho_\hslash \,  e^{\ci \frac{t}{\hslash} \hat{H}_\hslash}
\qquad \text{ and } \qquad  \tilde{\varrho}_\hslash (t)=
e^{\ci \frac{t}{\hslash} \hat{H}^0_\hslash}  \, \varrho_\hslash(t) \,  e^{-\ci \frac{t}{\hslash} \hat{H}^0_\hslash}\,.
\end{equation}
Recall the hypothesis (\hyperref[eq:Ass0]{$S_0^{(1)}$}) which specifies the regularity condition that we shall require for the initial quantum states $(\varrho_\hslash)_{\hslash\in (0,1)}$:
\begin{equation*}
\exists C_0>0, \forall \hslash\in(0,1), \qquad \Tr\big[ \varrho_\hslash \; (\hat{H}_\hslash^0+1)\big] \leq C_0\,.
\end{equation*}
The uniform bound in the trace means precisely that the operators
\begin{equation}
(\hat{H}_\hslash^0+1)^{1/2}\; \varrho_\hslash\; (\hat{H}_\hslash^0+1)^{1/2} \in \mathscr{L}^1(\mathscr H)\,,
\end{equation}
has a finite uniform trace norm with respect to $\hslash\in(0,1)$.
\begin{prop}
\label{lemma:Duhamel}
Assume  \eqref{eq:hypo0}, \eqref{eq:hypo1} and \emph{(}\hyperref[eq:hypo2]{$A_2^{(1)}$}\emph{)}. Let $(\varrho_\hslash)_{\hslash\in(0,1)}$ be a family of density matrices satisfying  \emph{(}\hyperref[eq:Ass0]{$S_0^{(1)}$}\emph{)}. Then for all   $\xi\in X^{1/2}$, all $\hslash\in(0,1)$ and all $t,t_0\in\R$,
\begin{eqnarray}
\label{eq:Duhamel}
\Tr \bigg[ \mathcal{W}(\xi) \; \tilde{\varrho}_\hslash (t)\bigg]= \Tr \bigg[ \mathcal{W}(\xi) \; \tilde{\varrho}_\hslash(t_0) \bigg]
-\frac{\ci}{\hslash} \int_{t_0}^t \Tr\bigg( \big [\mathcal{W}(\xi), \hat{H}_I(s) \big] \;
\tilde{\varrho}_\hslash (s) \bigg) \; \mathrm{d}s\,,
\end{eqnarray}
with
\begin{equation}
\label{eq:27}
\hat{H}_I(s)=  e^{\ci \frac{s}{\hslash} \hat{H}^0_\hslash} \,\big(\hat{H}_\hslash-\hat{H}^0_\hslash\big) \,
e^{-\ci \frac{s}{\hslash} \hat{H}^0_\hslash}\,,
\end{equation}
and where the commutator in the right hand side is interpreted  as a bounded quadratic form  on $Q(\hat{H}^0_\hslash)$.
\end{prop}
\begin{proof}
It follows using the Weyl--Heisenberg estimates of Lemma \ref{s3.lemW}, Corollary \ref{s4.lemE} and the  fundamental theorem of calculus for the function $t\in\mathds{R}\to  \Tr [ \mathcal{W}(\xi) \; \tilde{\varrho}_\hslash (t)]$ which is of class $\mathscr C^1$ in our case.
\end{proof}

\subsection{Commutator expansion}
\label{sec:commutators}

The main part to be analyzed in the Duhamel formula \eqref{eq:Duhamel}  is the commutator  $[\mathcal{W}(\xi), \hat{H}_I(s)]$. In particular, we wish to expand it with respect to the parameter $\hslash$. A simple computation gives
\begin{eqnarray*}
\hat{H}_\hslash-\hat{H}^0_\hslash &=& \sum_{i=1}^n \frac{1}{2m_i} \big(\hat{p}_i-{A}_i(\hat q_i,\hat a)\big)^2 -\frac{1}{2m_i} \hat{p}_i^2+V(\hat q)\\
&=&   \sum_{i=1}^n \frac{1}{2m_i} \bigg( \;{A}_i(\hat q_i,\hat a)^2- \;\hat{p}_i \cdot {A}_i(\hat q_i,\hat a)-\; {A}_i(\hat q_i,\hat a)  \cdot  \hat{p}_i\bigg)+V(\hat q)\\
&=&   \sum_{i=1}^n \frac{1}{2m_i} \bigg(  \;{A}_i(\hat q_i,\hat a)^2- 2\, \hat{p}_i  \cdot  {A}_i(\hat q_i,\hat a)  \bigg)+V(\hat q)\,.
\end{eqnarray*}
The last equality is a consequence of  the Coulomb gauge.
Moreover, one has
\begin{eqnarray*}
{A}_i^\nu(\hat q_i,\hat a)^2 &=&  \hat a^*(r_i^\nu)^2+\hat a(r_i^\nu)^2+2 \hat a^*(r_i^\nu) \hat a(r_i^\nu)+\hslash \,\|r_i^\nu\|^2_{L^2(\mathds{R}^d,\mathds{C}^{d-1})}\,,
\end{eqnarray*}
with  the function  $r^\nu_i$ given by
\begin{equation}
r_i^\nu:=  \left(\,\frac{\chi_i(k)}{\sqrt{2 |k|}}  \,e^{-2\pi\ci k\cdot \hat q_i}\, \epsilon_\lambda^{\nu}(k)\,\right)_{\lambda=1,\dots,d-1}\,\in L^\infty\big(\mathds{R}^d_{x_i}, L^2(\mathds{R}^d_k,\mathds{C}^{d-1})\big)\,,
\end{equation}
and
\begin{eqnarray*}
\hat a^*(r_i^\nu)&=& \sum_{\lambda=1}^{d-1} \int_{\mathds{R}^d} \,\frac{\chi_i(k)}{\sqrt{2 |k|}}  \,e^{-2\pi\ci k\cdot \hat q_i}\, \epsilon_\lambda^{\nu}(k)
\;a^*_\lambda(k) \,  \mathrm{d}k\,,\\
\hat a(r_i^\nu)&=& \sum_{\lambda=1}^{d-1} \int_{\mathds{R}^d} \,   \frac{\overline{\chi_i(k)}}{\sqrt{2 |k|}}  \,e^{2\pi\ci k\cdot \hat q_i}\, \epsilon_\lambda^{\nu}(k)
\;a_\lambda(k) \,  \mathrm{d}k\,.
\end{eqnarray*}
Using the above notations, one obtains
\begin{equation}
\label{eq:Hi0}
\begin{aligned}
\hat{H}_I(0)&= \sum_{i=1}^n \frac{1}{2m_i} \;\bigg( \sum_{\nu=1}^d \;\hat a^*(r_i^\nu)^2+ \hat a(r_i^\nu)^2+ 2 \hat a^*(r_i^\nu) \hat a(r_i^\nu)+ \hslash \, \|r_i^\nu\|_{L^2(\R^d,\C^{d-1})}^2 \bigg)\\
& -\sum_{i=1}^n \frac{1}{m_i}\bigg(\sum_{\nu=1}^d \,\hat{p}_i^\nu \cdot \big( \hat a^*(r_i^\nu)+\hat a(r_i^\nu)\big)\bigg)+V(\hat q)\,.
\end{aligned}
\end{equation}
In order to use the estimates of Subsection \ref{sec:regularity}, it is more convenient to put the quantity $\hat{H}_I(s)$ defined in \eqref{eq:27}  in the following form.
\begin{lemma}
For any $s\in\mathds{R}$ the time-evolved interaction  term $\hat{H}_I(s)$ takes the from,
\begin{eqnarray*}
\hat{H}_I(s)&=& \sum_{i=1}^n \frac{1}{2m_i} \; \sum_{\nu=1}^d \;a^*(r_i^\nu(s))^2+ a(r_i^\nu(s))^2+ 2 a^*(r_i^\nu(s)) a(r_i^\nu(s))+ \hslash \, \|r_i^\nu\|_{L^2(\R^d,\C^{d-1})}^2 \\
&-& \sum_{i=1}^n \frac{1}{m_i} \;\sum_{\nu=1}^d \, a^*(r_i^\nu(s)) \,\cdot\, \hat{p}_i^\nu  + \hat{p}_i^\nu\,\cdot\, a(r_i^\nu(s)) +
V(\hat q_1+s\frac{\hat p_1}{m_1},\dots,\hat q_n+s\frac{\hat p_n}{m_n})\,,
\end{eqnarray*}
where
\begin{equation}
r_i^\nu(s)=\left(\,\frac{\chi_i(k)}{\sqrt{2 |k|}}  \, \epsilon_\lambda^{\nu}(k) \,e^{-2\pi\ci k\cdot (\hat q_i+s \frac{\hat{p}_i}{m_i})+\ci s|k|}\,\right)_{\lambda=1,\dots,d-1}\,.
\end{equation}
Furthermore, $\hat{H}_I(s)$ is well defined as a quadratic form on $Q(\hat{H}_\hslash^0)$ which is relatively  $\hat{H}_\hslash^0$-form bounded.
\end{lemma}
\begin{proof}
Recall that $\hat{H}_I(s)= e^{\ci\frac{s}{\hslash} \hat{H}_\hslash^0} \hat{H}_I(0)  e^{-\ci\frac{s}{\hslash} \hat{H}_\hslash^0} $. So, the expression of  $\hat{H}_I(s)$ is a consequence of \eqref{eq:Hi0} and the commutation relations,
\begin{eqnarray*}
e^{\ci\frac{s}{\hslash} \hat{H}_f} \, a^\sharp(f) \, e^{-\ci\frac{s}{\hslash} \hat{H}_f}&=&  a^\sharp(e^{\ci s |k|} f)\,, \\
e^{\ci\frac{s}{\hslash} \frac{\hat{p}_i^2}{2 m_i} } \, \hat q_i^\nu\,e^{-\ci\frac{s}{\hslash} \frac{\hat{p}_i^2}{2 m_i}}&=& \hat q_i^\nu +s\frac{\hat{p}_i^\nu}{m_i}  \,, \quad
\nu=1,\dotsc,d\,.
\end{eqnarray*}
Additionally, the Coulomb gauge \eqref{eq:Coulgauge} condition is used  to show that
$$
\sum_{\nu=1}^d \,\hat{p}_i^\nu \cdot \hat a^*(r_i^\nu(s))=\sum_{\nu=1}^d \, a^*(r_i^\nu(s)) \,\cdot\, \hat{p}_i^\nu\,.
$$
Finally, thanks to Corollary \ref{cor:1} one checks that $\hat{H}_I(s)$ is a $\hat{H}_\hslash^0$-form bounded quadratic form.
\end{proof}

As mentioned before, our main purpose here is to establish a semiclassical expansion of the commutator
\begin{equation}
\label{eq:28}
\big[ \mathcal{W} (\xi); \hat{H}_I(s)\big]=
\bigg(\mathcal{W} (\xi) \, \hat{H}_I(s)\, \mathcal{W} (\xi)^*-\hat{H}_I(s) \bigg)  \,\mathcal{W} (\xi).
\end{equation}

\begin{lemma}
\label{lemma:5}
For any $s\in\mathds{R}$ and any $\xi=(p,q,\alpha)\in X^{1/2}$ the following identity holds true in the sense of quadratic forms in $Q(\hat{H}_\hslash^0)$,
\begin{eqnarray*}
\mathcal{W} (\xi)  \hat{H}_I(s) \mathcal{W} (\xi)^*&=& \sum_{i=1}^n \frac{1}{2m_i} \, \sum_{\nu=1}^d  \bigg[  \bigg( a^*(\tilde r_i^\nu(s))
+\frac{i\hslash}{\sqrt{2}}  \langle\alpha, \tilde{r}_i^\nu(s)\rangle \bigg)^2  +  \bigg( a(\tilde r_i^\nu(s))-
\frac{i\hslash}{\sqrt{2}}  \langle \tilde{r}_i^\nu(s), \alpha\rangle \bigg)^2 \\
&+& 2
\bigg(a^*(\tilde r_i^\nu(s))
+\frac{i\hslash}{\sqrt{2}}  \langle\alpha, \tilde{r}_i^\nu(s)\rangle \bigg) \, \bigg(a(\tilde r_i^\nu(s))
-\frac{i\hslash}{\sqrt{2}}  \langle \tilde{r}_i^\nu(s), \alpha\rangle \bigg)+\hslash \|r_i^\nu\|^2_{L^2}\bigg]\\
& -& \sum_{i=1}^n \frac{1}{m_i} \sum_{\nu=1}^d  \bigg(a^*(\tilde r_i^\nu(s))
+\frac{i\hslash}{\sqrt{2}}  \langle\alpha, \tilde{r}_i^\nu(s)\rangle \bigg) \bigg(  \hat{p}_i^\nu-\hslash  p_i^\nu\bigg) \\
& -& \sum_{i=1}^n \frac{1}{m_i} \sum_{\nu=1}^d  \bigg(  \hat{p}_i^\nu-\hslash  p_i^\nu\bigg)  \bigg(a(\tilde r_i^\nu(s))
-\frac{i\hslash}{\sqrt{2}}  \langle \tilde{r}_i^\nu(s), \alpha\rangle \bigg) \\
& +& V\big(\hat q_1+s\frac{\hat p_1}{m_1}-\hslash(q_1+s\frac{p_1}{m_1}),\dots,\hat q_n+s\frac{\hat p_n}{m_n}-\hslash(q_n+s\frac{p_n}{m_n})\big)\,.
\end{eqnarray*}
where
\begin{equation}
\label{eq:30}
\tilde{r}_i^\nu(s) = r_i^\nu(s) \, e^{2\pi\ci \hslash k \cdot (q_i+s\frac{p_i}{m_i})}\,.
\end{equation}
\end{lemma}
\begin{proof}
The stated identity is a consequence of the commutation relations
\begin{eqnarray}
W(\alpha) a^*(f) W(\alpha)^*&=& a^*(f) +\ci \frac{\hslash}{\sqrt{2}} \; \langle \alpha, f\rangle_{L^2(\mathds{R}^d, \mathds{C}^{d-1})}\,,\\
W(\alpha) a(f) W(\alpha)^*&=& a(f) -\ci \frac{\hslash}{\sqrt{2}} \; \langle  f, \alpha\rangle_{L^2(\mathds{R}^d, \mathds{C}^{d-1})}\,,
\end{eqnarray}
and
\begin{eqnarray}
\hat{T}(z) \;\hat{p}_i^\nu \;\hat{T}(z) ^*= \hat{p}_i^\nu-\hslash p_i^\nu\,, \qquad
\hat{T}(z) \;\hat{q}_i^\nu\; \hat{T}(z) ^*= \hat{q}_i^\nu-\hslash q_i^\nu\,,
\end{eqnarray}
where $z\equiv (p,q)\in \R^{2dn}$. In particular, we have
\begin{equation}
\hat{T}(z) \,r_i^\nu(s) \,\hat{T}(z) ^*= \tilde{r}_i^\nu(s) \in  \mathscr{L}\big(L^2(\mathds{R}^d_{x_i}); L^2(\mathds{R}^d_{x_i})\otimes
L^2(\mathds{R}^d_k, \mathds{C}^{d-1})\big)\,.
\end{equation}
\end{proof}

\begin{lemma}
\label{prop:comexp}
For any $s\in\mathds{R}$ and any $\xi=(p,q,\alpha)\in X^{1/2}$  the following semiclassical commutator expansion holds true,
\begin{equation}
\label{eq:29}
\frac{1}{\hslash} \; \big [ \mathcal{W}(\xi); \hat{H}_I(s) \big] =\bigg({M}(s,\hslash,\xi)  + \hslash \,{R}(s,\hslash, \xi) \bigg)  \; \mathcal{W}(\xi)
\end{equation}
as quadratic form in $Q(\hat{H}_\hslash^0)$ with the principal part ${M}(\cdot)$  given by
\begin{equation}
\label{eq:Mhxi}
\begin{aligned}
\hspace{-.8in}{M}(s,\hslash,\xi)&= \sum_{i=1}^n \frac{1}{2m_i}  \,\left[ \frac{\ci}{\sqrt{2}} \sum_{\nu=1}^d \langle \alpha, {r}_i^\nu(s)\rangle a^*(r_i^\nu(s)) - \langle
 r_i^\nu(s), \alpha\rangle a( r_i^\nu(s))\right] \;+\;\mathrm{h.c.} \\
&- \sum_{i=1}^n \frac{1}{m_i}  \,\left[ \frac{\ci}{\sqrt{2}}\sum_{\nu=1}^d   a^*(r_i^\nu(s)) \langle  r_i^\nu(s), \alpha\rangle \right]\;+\;hc \\
&- \sum_{i=1}^n \frac{1}{m_i} \left[\frac{i}{\sqrt{2}}\sum_{\nu=1}^d  \langle \alpha, r_i^\nu(s)\rangle \, \hat{p}_i^\nu- \sum_{\nu=1}^d a^*(r_i^\nu(s)) p_i^\nu\right] \;+\;hc \\
&- \nabla V\big(\hat q_1+s\frac{\hat p_1}{m_1},\dots,\hat q_n+s\frac{\hat p_n}{m_n}\big)\left[q_1+s \frac{p_1}{m_1},\dots,q_n+s \frac{p_n}{m_n}\right]\,,
\end{aligned}
\end{equation}
and such that there exists a constant $c>0$ so that for all $\hslash\in(0,1)$, $s\in\R$ and $\xi\in X^{1/2}$,
\begin{equation}
\label{eq:Ms}
\|(\hat{H}_\hslash^0+1)^{-1/2} \, M(s,\hslash,\xi) \,(\hat{H}_\hslash^0+1)^{-1/2}\|_{\mathscr{L}(\mathscr H)}\leq c (1+|s|) \|\xi\|_{X^0}\,,
\end{equation}
and
\begin{equation}
\label{eq:32}
\left\| (\hat{H}_\hslash^0+1)^{-1/2} \,  {R}(s,\hslash,\xi) \, (\hat{H}_\hslash^0+1)^{-1/2} \right\|_{\mathscr{L}(\mathscr{H})}
 \leq c (1+|s|)^2 (1+\|\xi\|_{X^0})^2 \, .
\end{equation}

\end{lemma}
\begin{proof}
The commutator expansion \eqref{eq:29} is proved  by using Lemma \ref{lemma:5}, the identity \eqref{eq:28} and
the Taylor expansion for $\tilde{r}_i^\nu$ and for the potential $V$ with respect to the $\hslash$ parameter,  \emph{i.e.}
\begin{equation}
\label{eq:31}
\tilde{r}_i^\nu(s) = r_i^\nu(s) \big(1+ \hslash  \;\frac{e^{2\pi\ci \hslash k \cdot (q_i+s\frac{p_i}{m_i})}-1}{\hslash}\big)= r_i^\nu(s) + \hslash \, g_i^\nu(\hslash,s)\,,
\end{equation}
and
\begin{equation}
\label{eq:TalV}
V\big(\hat q+s\hat{\mathbf{p}}-\hslash(q+s\mathbf{p}))\big)=V\big(\hat q+s\hat{\mathbf{p}}\big)-\hslash \nabla V\big(\hat q+s\hat{\mathbf{p}}\big)(q+s\mathbf{p})+ \hslash^2 \,\theta(\hslash,s)\,,
\end{equation}
where
$$
\hat{\mathbf{p}}=(\frac{\hat p_i}{m_i})_{i=1,\dots,n},  \quad \text{ and } \quad\mathbf{p}=(\frac{p_i}{m_i})_{i=1,\dots,n}.
 $$
In particular, using the elementary estimate $|e^{\ci y}-1|\leq \sqrt{2} |y|$ for all $y\in\R$, one shows  that $g_i^\nu$ satisfy the following uniform bounds with respect to $\hslash\in(0,1)$ and $s\in\R$
\begin{equation}
\begin{split}
\|g_i^\nu(\hslash,s)\|_{\mathscr{L}\bigl(L^2,L^2\otimes L^2 (\mathds{R}^d, \mathds{C}^{d-1})\bigr)} &\leq M_1 (1+|s|) \, \|\xi\|_{X^0} \,\|\chi_i\|_{\mathfrak{H}^{1/2}}\,, \\
\|\omega^{-1/2} g_i^\nu(\hslash,s)\|_{\mathscr{L}\bigl(L^2,L^2\otimes L^2 (\mathds{R}^d, \mathds{C}^{d-1})\bigr)} &\leq M_1 (1+|s|) \, \|\xi\|_{X^0} \,\|\chi_i\|_{\mathfrak{H}^{0}}\,,
\end{split}
\end{equation}
 for some constant $M_1>0$. On the other hand,  $\theta$ satisfies for some constant $M_2$,
\begin{equation}
 \lVert\theta(\hslash,s)\rVert_{\mathscr{L}(L^2(\R^{dn},dx))} \leq M_2 (1+|s|)^2 \,\|\xi\|_{X^0}^2\,,
\end{equation}
 uniformly  with respect to $\hslash\in(0,1)$ and $s\in\R$. Combining the latter estimates on $g^\nu_i$ and $\theta$ with the aid of
 Corollary \ref{cor:1} and the assumptions \eqref{eq:hypo0} on the potential $V$ and \eqref{eq:hypo1}--(\hyperref[eq:hypo2]{$A_2^{(1)}$}) on $\chi_i$, one shows the estimates \eqref{eq:Ms} and  \eqref{eq:32}.
\end{proof}

\subsection{Characteristic equation}
\label{sec:classical-limit}
In this subsection, we take the classical limit $\hslash\to 0$ in the Duhamel formula previously derived in Proposition \ref{lemma:Duhamel}.  Recall  that   $(\tilde\varrho_\hslash(t))_{\hslash\in (0,1)}$  and $(\varrho_\hslash(t))_{\hslash\in (0,1)}$ are given in \eqref{eq:25}.

\begin{lemma}
\label{lemma:rhoSt}
Let $\sigma\in[\frac 1 2,1]$ and assume \eqref{eq:hypo0}, \eqref{eq:hypo1} and \eqref{eq:hypo2}. Let $(\varrho_\hslash)_{\hslash\in (0,1)}$  be a family of density matrices on $\mathscr{H}$ satisfying the assumptions \emph{(}\hyperref[eq:Ass0]{$S_0^{(1)}$}\emph{)}, \emph{(}\hyperref[eq:Ass1]{$S_1^{(1)}$}\emph{)} and \eqref{eq:Ass2}. Then  the families of states
  $(\varrho_\hslash(t))_{\hslash\in (0,1)}$ and $(\tilde\varrho_\hslash(t))_{\hslash\in (0,1)}$ defined according to \eqref{eq:25} satisfy the same assumptions   \emph{(}\hyperref[eq:Ass0]{$S_0^{(1)}$}\emph{)}, \emph{(}\hyperref[eq:Ass1]{$S_1^{(1)}$}\emph{)} and \eqref{eq:Ass2}  uniformly for any $t\in\R$ in an arbitrary compact time interval.
\end{lemma}
\begin{proof}
First, note that
\begin{eqnarray*}
\Tr\big[ \tilde\varrho_\hslash(t) \; \hat{H}_\hslash^0\big] &=&\Tr\big[ \varrho_\hslash(t) \; \hat{H}_\hslash^0\big]\,, \\
 \Tr\big[ \tilde\varrho_\hslash(t) \; {\rm d}\Gamma(|k|^{2\sigma})\big] &=&\Tr\big[ \varrho_\hslash(t) \; {\rm d}\Gamma(|k|^{2\sigma})\big] \,,\\
  \Tr\big[ \tilde\varrho_\hslash(t) \; \; \hat N_\hslash\big] &=& \Tr\big[ \varrho_\hslash(t) \; \hat N_\hslash\big]\,.
\end{eqnarray*}
Thus, it is enough to bound only the right hand sides. By spectral  decomposition,
$$
\varrho_\hslash(t)=\sum_{m\in\N} \lambda_\hslash(m) \; | e^{-\ci \frac{t}{\hslash} \hat{H}_\hslash}\psi_\hslash(m)\rangle \langle  e^{-\ci \frac{t}{\hslash} \hat{H}_\hslash}\psi_\hslash(m)|\,,
$$
and
$$
\Tr\big[ \varrho_\hslash(t) \; \; \hat N_\hslash\big]=\sum_{m\in\N} \lambda_\hslash(m) \; \|
(\hat N_\hslash)^{1/2} \,e^{-\ci \frac{t}{\hslash} \hat{H}_\hslash}\psi_\hslash(m)\|^2\,.
$$
Hence, using Corollary \ref{cor:2-s} one shows for some constants $K,C>0$,
\begin{align*}
    \Tr\big[ \varrho_\hslash(t) \;  \hat N_\hslash\big] \leq K \big( \Tr\big[ \varrho_\hslash \, \hat N_{\hslash}]+ \Tr\big[ \varrho_\hslash \;(\hat H_{\hslash}^0+1)]\big)e^{C|t| }\,,
  \end{align*}
  uniformly in $\hslash\in(0,1)$. Similarly, $\varrho_\hslash(t)$ satisfies the assumption \eqref{eq:Ass2}  by Corollary \ref{cor:2-s}   and the spectral decomposition. Now, using Corollary \ref{s4.lemE} one deduces
\begin{equation}
\label{eq:rhtH}
\begin{aligned}
   \Tr\big[ \varrho_\hslash(t) \; \; \hat{H}^0_{\hslash}\big] &=
   \sum_{m\in\N} \lambda_\hslash(m) \; \|
(\hat H^0_\hslash)^{1/2} \,e^{-\ci \frac{t}{\hslash} \hat{H}_\hslash}\psi_\hslash(m)\|^2 \\ & \lesssim
\sum_{m\in\N} \lambda_\hslash(m) \; \|
(\hat H_\hslash+a)^{1/2} \,\psi_\hslash(m)\|^2 \\ & \lesssim
\sum_{m\in\N} \lambda_\hslash(m) \; \|
(\hat H^0_\hslash+b)^{1/2} \,\psi_\hslash(m)\|^2 \\ & \lesssim
\Tr\big[ \varrho_\hslash \,  (\hat{H}^0_{\hslash}+1)\big]\,.
\end{aligned}
\end{equation}
On the other hand, thanks to Lemma \ref{lemma:q2}-(ii) the estimate
\begin{equation}
\label{eq:vrq2}
\Tr\big[ \varrho_\hslash(t) \, \hat q^2\big] \leq K_1 \, \Tr\big[ \varrho_\hslash \, (\hat H^0_\hslash+\hat q^2+1)\big]\,e^{c_1 |t|}
\end{equation}
holds true  for some constant $K_1,c_1>0$ and uniformly in $\hslash\in(0,1)$. Finally, applying  Lemma \ref{lemma:q2}-(i) for the state
$\varrho_\hslash(t)$ with the help of its spectral decomposition, one proves the estimate
\begin{eqnarray*}
\Tr\big[ \tilde\varrho_\hslash(t) \, \hat q^2\big] &\leq& K_1 \, \Tr\big[ \varrho_\hslash(t) \, (\hat H^0_\hslash+\hat q^2+1)\big]\,e^{c_1 |t|}\\
&\leq& K_2 \, \Tr\big[ \varrho_\hslash \, (\hat H^0_\hslash+\hat q^2+1)\big]\,e^{c_2 |t|}\,.
\end{eqnarray*}
For the last inequality we have used \eqref{eq:rhtH} and \eqref{eq:vrq2}.
\end{proof}

\begin{lemma}
\label{lem:apest}
 Let $(\varrho_\hslash)_{\hslash\in(0,1)}$ be a family of density matrices on $\mathscr H$ satisfying  \emph{(}\hyperref[eq:Ass0]{$S_0^{(1)}$}\emph{)},   \emph{(}\hyperref[eq:Ass1]{$S_1^{(1)}$}\emph{)} and \eqref{eq:Ass2} for some $\sigma\in[\frac 1 2,1]$. Then for any compact time interval $J$ there exists $C>0$ such that for any Wigner measure $\mu_t\in \mathcal{M}({\varrho}_{\hslash}(t), \hslash\in(0,1))$,  $\tilde\mu_t\in \mathcal{M}(\tilde{\varrho}_{\hslash}(t),\hslash\in(0,1))$ and for all  $t\in J$,
\begin{eqnarray}
\label{eq:apest}
(i) \qquad \int_{X^0}  \|u\|^2_{X^\sigma}\; \mathrm{d}{\mu}_t<C\,,\\
(ii) \qquad \int_{X^0}  \| u\|_{X^{\sigma}}^2\; \mathrm{d}\tilde{\mu}_t<C.
\end{eqnarray}
\end{lemma}
\begin{proof}
The result is a consequence of the previous Lemma \ref{lemma:rhoSt} and a general argument related to Wick calculus and semiclassical analysis elaborated in \cite[Lemma 3.12]{MR3379490}. In fact, according to the latter reference if one checks that a family of density matrices $(\rho_\hslash)_{\hslash\in(0,1)}$  on $\mathscr H$ satisfies
$$
\Tr\left[\rho_\hslash \, \hat N_\hslash\right]\leq C , \qquad \Tr\left[\rho_\hslash \, {\rm d}\Gamma(|k|^{2\sigma})\right]\leq C, \quad \text{ and } \Tr\left[\rho_\hslash \, (\hat p^2+\hat q^2)\right]\leq C\,,
$$
for some constant $C>0$ uniformly in $\hslash\in(0,1)$ then it holds true that for any Wigner measure $\mu\in \mathcal{M}({\rho}_{\hslash},\hslash\in(0,1))$,
\begin{eqnarray*}
\int_{X^0} \|\alpha\|_{\mathfrak{H}^0}^2 \, \mathrm{d}\mu\leq C\,, \quad \int_{X^0} \|\alpha\|_{\mathfrak{H}^\sigma}^2 \, \mathrm{d}\mu\leq C\,,\qquad \text{ and }  \int_{X^0} (p^2+q^2) \, \mathrm{d}\mu\leq C\,.
\end{eqnarray*}
Thus, applying such result for $\varrho_\hslash(t)$ and $\tilde\varrho_\hslash(t)$ with the help of the uniform estimates in Lemma \ref{lemma:rhoSt}, one proves the claimed bounds.
\end{proof}

\begin{lemma}
\label{lem:Ghst}
 Let $(\varrho_\hslash)_{\hslash\in(0,1)}$ be a family of density matrices satisfying  \eqref{eq:Ass0}. Then for any compact time interval $J$ there exists $C>0$ such that for all
$\xi\in X^{1/2}$ and all $\hslash\in(0,1)$,
\begin{eqnarray}
\label{eq:Ghst}
\left|\Tr \big[ \mathcal{W}(\xi) \; \tilde{\varrho}_\hslash (t)\big]-\Tr \big[ \mathcal{W}(\xi) \; \tilde{\varrho}_\hslash(s) \big]\right|\leq C |t-s| \,(1+\|\xi\|_{X^0})^2\,,
\end{eqnarray}
for any $t,s\in J$.
\end{lemma}
\begin{proof}
Applying Proposition \ref{lemma:Duhamel} and Lemma \ref{prop:comexp}, one concludes
\begin{eqnarray*}
&&\hspace{-.1in}\left|\Tr \big[ \mathcal{W}(\xi) \; \tilde{\varrho}_\hslash (t)\big]-\Tr \big[ \mathcal{W}(\xi) \; \tilde{\varrho}_\hslash(t_0) \big]\right|  \\
 && \hspace{.8in}\lesssim \left|\int_{t_0}^t \left\| (\hat{H}_\hslash^0+1)^{-1/2} \,\bigg( {M}(s,\hslash,\xi) + \hslash \, {R}(s,\hslash, \xi) \bigg) \,(\hat{H}_\hslash^0+1)^{-1/2}
\right\|_{\mathscr{L}(\mathscr H)} \,\mathrm{d}s  \right|\,.
\end{eqnarray*}
Note that in the last inequality we have also used Lemma \ref{s3.lemW} . Now, using the bounds  \eqref{eq:Ms} and \eqref{eq:32} one obtains for any $t,t_0\in J$,
\begin{eqnarray*}
 \left|\Tr \big[ \mathcal{W}(\xi) \; \tilde{\varrho}_\hslash (t)\big]-\Tr \big[ \mathcal{W}(\xi) \; \tilde{\varrho}_\hslash(t_0) \big]\right| &\leq&  c \left|
\int_{t_0}^t \; (1+|s|)^2 (1+\|\xi\|_{X^0})^2\, \mathrm{d}s \right| \\  &\leq & C
|t-t_0| \; (1+\|\xi\|_{X^0})^2\,.
\end{eqnarray*}
\end{proof}

\begin{lemma}
\label{lem:Ghxi}
 Let $(\varrho_\hslash)_{\hslash\in(0,1)}$ be a family of density matrices satisfying  \emph{(}\hyperref[eq:Ass0]{$S_0^{(1)}$}\emph{)} and \emph{(}\hyperref[eq:Ass1]{$S_1^{(1)}$}\emph{)}. Then for any compact time interval $J$ there exists $C>0$ such that for all   $\xi_1,\xi_2\in X^{1/2}$, all $\hslash\in(0,1)$ and all $t\in J$,
\begin{eqnarray}
\label{eq:Ghxi}
\left|\Tr \big[ \mathcal{W}(\xi_1) \; \tilde{\varrho}_\hslash (t)\big]-\Tr \big[ \mathcal{W}(\xi_2) \; \tilde{\varrho}_\hslash(t) \big]\right|\leq C
 \|\xi_1-\xi_2\|_{X^0} (\|\xi_1\|_{X^0}+\|\xi_2\|_{X^0}+1) \,.
\end{eqnarray}
\end{lemma}
\begin{proof}
We shall use  the following estimate:
\begin{equation*}
\|\big(\mathcal{W}(\xi_1)-\mathcal{W}(\xi_2) \big) \,(\hat N_\hslash+1)^{-1/2}\|_{\mathscr{L}(\mathscr H)}\leq c \|\xi_1-\xi_2\|_{X^0} (\|\xi_1\|_{X^0}+\|\xi_2\|_{X^0}+1)\,.
\end{equation*}
See for instance \cite[Lemma 3.1]{ammari2008ahp} for a proof of such inequality.  This yields
\begin{equation}
\label{eq:33}
\begin{split}
&\hspace{-.7in}\left|\Tr \big[ \big(\mathcal{W}(\xi_1)- \mathcal{W}(\xi_1)\big) \; \tilde{\varrho}_\hslash (t)\big] \right| \\
&\leq  \|\big(\mathcal{W}(\xi_1)-\mathcal{W}(\xi_2) \big) \,(\hat N_\hslash+1)^{-1/2}\|_{\mathscr{L}(\mathscr H)} \times \|(\hat N_\hslash+1)^{1/2} \; \tilde{\varrho}_\hslash (t) \|_{\mathscr L^1(\mathscr H)}\,.
\end{split}
\end{equation}
Now, assumptions  (\hyperref[eq:Ass0]{$S_0^{(1)}$})--(\hyperref[eq:Ass1]{$S_1^{(1)}$}) and Lemma \ref{lemma:rhoSt} imply that there exists a constant $c>0$ such that for all times $t\in J$,
\begin{equation}
\label{eq:trN}
\Tr\left[ \; \tilde{\varrho}_\hslash(t) (\hat N_\hslash+1)\right]\leq c\,,
\end{equation}
uniformly with respect to $\hslash\in(0,1)$. Such bound \eqref{eq:trN} ensures  that the norm
$$\|(\hat N_\hslash+1)^{1/2} \; \tilde{\varrho}_\hslash (t)\|_{\mathscr L^1(\mathscr H)}$$
 is uniformly bounded for $t\in J$, $\hslash\in(0,1)$, and completes the proof.
\end{proof}

\subsubsection{Extraction}
For each time $t_0\in\R$, it is possible thanks to Definition \ref{def:wigner} and  Proposition \ref{wig} to extract a sequence $(\hslash_k)_{k\in\N}$ -- which  may depend in the time $t_0$ -- such that $\mathscr{M}(\tilde\varrho_{\hslash_k},k\in\N)$ is a singleton. However, for our analysis we need to found a sequence $(\hslash_k)_{k\in\N}$ such that for all times $\mathscr{M}(\tilde\varrho_{\hslash_k}(t),k\in\N)$ is a singleton.  Such property is proved below and it is based in a diagonal extraction procedure and Prokhorov's theorem. In particular, we shall use the following  "weak" narrow convergence topology on $\mathfrak{P}(X^\sigma)$. Let $(e_{n})_{n \in \N}$  be a O.N.B of the Hilbert space $X^\sigma$ and define  the norm,
\begin{equation}
\label{normw}
||u||_{X^\sigma_{w}}^2:=\sum_{n \in \N}\frac{1}{n^{2}} |\langle
      u, e_{n}\rangle_{X^\sigma}|^{2}, \,
\end{equation}
on $X^\sigma$.  A curve $t\in \R\to \mu_t\in\mathfrak{P}(X^\sigma)$ is said to be \emph{weakly narrowly continuous} if and only if the real-valued maps,
$$
t\in \R\to \int_{X^\sigma} \psi \, d\mu_t\in\R\,,
$$
are continuous for every  bounded continuous function on the space $(X^\sigma,\|\cdot\|_{X^\sigma_w})$.

\begin{prop}
\label{prop:4}
Let $\sigma\in[\frac 1 2, 1]$ and assume that \eqref{eq:hypo0}, \eqref{eq:hypo1} and \eqref{eq:hypo2} are satisfied.
Let $(\varrho_\hslash)_{\hslash\in (0,1)}$  be a family of density matrices in $\mathscr{H}$ satisfying the assumptions  \emph{(}\hyperref[eq:Ass0]{$S_0^{(1)}$}\emph{)}, \emph{(}\hyperref[eq:Ass1]{$S_1^{(1)}$}\emph{)} and \eqref{eq:Ass2}. For any  sequence $(\hslash_n)_{n\in\mathds{N}}$ in $(0,1)$ such that  $\hslash_n\to 0$,  there exists a subsequence  $(\hslash_{n_k})_{k\in\mathds{N}}$ and a family of probability measures $(\tilde{\mu}_t)_{t\in\mathds R}$ such that for
all $t\in \mathds{R}$,
\begin{eqnarray*}
\mathcal{M}(\tilde{\varrho}_{\hslash_{n_k}}(t), k\in \mathds{N})=\{ \tilde{\mu}_t\}\,.
\end{eqnarray*}
Moreover, for every compact time interval $J$ there exits a constant $C>0$ such that for all times $t\in J$,
\begin{equation}
\label{eq:estmut}
\int_{X^0}  \|u\|^2_{X^\sigma}\; \mathrm{d}\tilde{\mu}_t<C\,.
\end{equation}
\end{prop}
\begin{proof}
Denote for $\xi=(p,q,\alpha)\in X^0$,
$$
G_\hslash(t,\xi):=\Tr\left[ \tilde{\varrho}_{\hslash}(t) \, \mathcal{W}\big(\pi q,-\pi p,\sqrt{2}\pi \alpha\big)\right]\,.
$$
For every compact interval $J$ there exists $C>0$ such that for every $t,s\in J$ and every $\xi_1,\xi_2\in X^0$,
\begin{equation}
\label{eq:36}
\begin{split}
\left|G_\hslash(t,\xi_1)-G_\hslash(s,\xi_2)\right| &\leq  \left|G_\hslash(t,\xi_1)-G_\hslash(s,\xi_1)\right| +\left|G_\hslash(s,\xi_1)-G_\hslash(s,\xi_2)\right|
\\ &\leq  C |t-s| \,(1+\|\xi_1\|_{X^0})^2+ C
 \|\xi_1-\xi_2\|_{X^0} (\|\xi_1\|_{X^0}+\|\xi_2\|_{X^0}+1)\,.
\end{split}
\end{equation}
The last inequality is a consequence of Lemmas \ref{lem:Ghst} and \ref{lem:Ghxi}. Now take a countable dense set $(t_j)_{j\in\N}$ in $\R$. Then
by Proposition \ref{wig}  for $t_1$ there exists a sequence $(\hslash_\ell)_{\ell\in\N}$ and a probability measure $\tilde\mu_{t_1}$ such that
$$
\mathcal{M}(\tilde{\varrho}_{\hslash_{\ell}}(t_1), \ell\in \mathds{N})=\{ \tilde{\mu}_{t_1}\}\,.
$$
Repeating the same argument for each $t_2,\dots,t_j,\dots$ and extracting at each time a subsequence from the previous one, one obtains the existence of
the sequences $(\hslash_{\psi_2(\ell)})_{\ell\in\N},\dots, (\hslash_{\psi_j(\ell)})_{\ell\in\N},\dots$ such that
$$
\mathcal{M}(\tilde{\varrho}_{\hslash_{\psi_j(\ell)}}(t_j), \ell\in \mathds{N})=\{ \tilde{\mu}_{t_j}\}\,,
$$
and  such that $\psi_k:\N\to\N$ are successive extractions. Hence, by the diagonal argument there exists a sequence -- given precisely by $(\hslash_{\psi_\ell(\ell)})_{\ell\in\N}$ --  still denoted by $(\hslash_\ell)_{\ell\in\N}$ for simplicity such that for all $j\in\N$,
$$
\mathcal{M}(\tilde{\varrho}_{\hslash_{\ell}}(t_j), \ell\in \mathds{N})=\{ \tilde{\mu}_{t_j}\}\,.
$$
So, according to Definition \ref{def:wigner} this means that for any $\xi\in X^0$ and $j\in\N$,
$$
G(t_j,\xi):=\lim_{\ell} \Tr\big[ \tilde\varrho_{\hslash_\ell}(t_j) \;\mathcal{W}\big(\pi q,-\pi p,\sqrt{2}\pi \alpha\big)\big] =
\int_{u\in X^0} e^{2\pi\ci  \Ree\langle \xi, u\rangle_{X^0}}\; \mathrm{d}\tilde\mu_{t_j}(u)\,.
$$
Consider the family $(\tilde\varrho_{\hslash_{\ell}}(t_j))_{\ell\in\N}$ with the diagonal extracted sequence $(\hslash_\ell)_{\ell\in\N}$ constructed above; then the estimate \eqref{eq:Ghst} yields,
\begin{eqnarray}
\left|G(t_j,\xi)-G(t_{j'},\xi)\right|\leq C |t_j-t_{j'}| \,(1+\|\xi\|_{X^0})^2\,.
\end{eqnarray}
So, by a density argument and completeness one extends the function $G(\cdot,\xi)$ to all $\R$ such that
$$
G(t,\xi):=\lim_{t_j\to t} G(t_j,\xi)\,.
$$
The limit is independent of the choice of the sequence $t_j\to t$ and satisfies as in \eqref{eq:36} the estimate,
$$
 \left|G(t,\xi_1)-G(s,\xi_2)\right|\leq C |t-s| \,(1+\|\xi_1\|_{X^0})^2+ C
 \|\xi_1-\xi_2\|_{X^0} (\|\xi_1\|_{X^0}+\|\xi_2\|_{X^0}+1)\,.
$$
Moreover, by the triangle inequality and \eqref{eq:36} one deduces  that for all $t\in\R$ and $\xi\in X^0$,
\begin{equation}
\label{eq:39}
\lim_\ell G_{\hslash_\ell}(t,\xi)=\lim_\ell\Tr\left[ \tilde{\varrho}_{\hslash_{\ell}}(t) \, \mathcal{W}\big(\pi q,-\pi p,\sqrt{2}\pi \alpha\big)\right]=G(t,\xi)\,.
\end{equation}
Remark that by Lemma \ref{lem:apest}--(ii) one checks that for all $t_j$ in an arbitrary compact interval,
\begin{equation}
\label{eq:37}
 \int_{X^0}  \| u\|_{X^{\sigma}}^2\; \mathrm{d}\tilde{\mu}_{t_j}<\infty\,.
\end{equation}
Now, we claim that for all $t\in\R$ there exists $\tilde\mu_{t}\in\mathfrak P(X^0)$ such that for all $\xi\in X^0$,
\begin{equation}
\label{eq:38}
G(t,\xi)=\int_{X^0}  e^{2\pi\ci  \Ree\langle \xi, u\rangle_{X^0}}\; \mathrm{d}\tilde\mu_{t}(u)\,.
\end{equation}
Of course, the above statement holds true for all $t_j$ and we need to extended to every time $t$.
On the other hand, we know that the limit \eqref{eq:39} exists for all times  and so it is enough to identify $G(t,\xi)$ as the characteristic function of a given probability measure over $X^0$.
In fact, thanks to the bound \eqref{eq:37}
 one notices that the family
$(\tilde\mu_{t_j})_{j\in\N}$ is tight in $\mathfrak{P}(X^0)$, where $X^0$ is endowed with the "weak" norm $\|\cdot\|_{X^0_w}$. Hence, by Prokhorov's theorem
such family is sequentially compact with respect to the weak narrow convergence topology on $\mathfrak P(X^0)$. Therefore, for any $t\in\R$ and $(t_{\psi(j)})_{j\in\N}$ a subsequence such that $\lim t_{\psi(j)}=t$ there exists a subsequence -- still denote the same -- so that $\tilde\mu_{\psi(j)}$ converges towards a Borel probability measure $\tilde\mu_t$. Consequentially, for every time $t$ and for every $\xi\in X^0$,
$$
\lim_j G\big(t_{\psi(j)},\xi\big)=G(t,\xi)= \int_{X^0}  e^{2\pi\ci  \Ree\langle \xi, u\rangle_{X^0}}\; \mathrm{d}\tilde\mu_{t}\,,
$$
since the complex-valued function $u\mapsto e^{2\pi\ci  \Ree\langle \xi, u\rangle_{X^0}}$ is bounded continuous relatively to the "weak" norm
 $\|\cdot\|_{X^0_w}$. Hence, from \eqref{eq:39} one deduces  for all $t\in\R$ and $\xi=(p,q,\alpha)\in X^0$,
$$
\lim_\ell\Tr\left[ \tilde{\varrho}_{\hslash_{\ell}}(t) \, \mathcal{W}\big(\pi q,-\pi p,\sqrt{2}\pi \alpha\big)\right]=\int_{X^0}  e^{2\pi\ci  \Ree\langle \xi, u\rangle_{X^0}}\; \mathrm{d}\tilde\mu_{t}\,.
$$
\end{proof}

\subsubsection{Convergence }
Now, we are in position to prove convergence of the Duhamel formula \eqref{eq:Duhamel} when $\hslash\to 0$ and to derive the characteristic equation satisfied by the Wigner measures $(\tilde\mu_t)_{t\in\R}$.
\begin{lemma}
\label{eq:ch2}
Suppose the  same hypothesis as in Proposition \ref{prop:4}. Then for all $t,t_0\in\R$ and all $\xi=(z,\alpha)\in X^0$,
\begin{equation*}
\tilde\mu_{t}(e^{2\ci\Imm\langle \cdot,z\rangle} e^{\sqrt{2}\ci\Ree\langle \alpha,\cdot\rangle_{\mathfrak{H}^0}})=\tilde\mu_{t_0}(
e^{2\ci\Imm\langle \cdot,z\rangle} e^{\sqrt{2}\ci\Ree\langle \alpha,\cdot\rangle_{\mathfrak{H}^0}}) -\ci \lim_{k\to\infty} \int_{t_0}^t \,
\Tr\bigg [ {M}(s,\hslash,\xi)\, \mathcal{W}(\xi) \, \tilde{\varrho}_{\hslash_{n_k}}(s)\,\bigg] \mathrm{d}s,
\end{equation*}
where the term ${M}(s,\hslash,\xi)$ is given explicitly in Lemma \ref{prop:comexp}.
\end{lemma}
\begin{proof}
Apply the Duhamel's formula of Proposition  \ref{lemma:Duhamel} with aid of Lemma \ref{prop:comexp}.
\end{proof}

The main step next is to use dominated convergence and to calculate the limit
 $$\lim_{k\to\infty}
\Tr[ {M}(s,\hslash,\xi) \, \mathcal{W}(\xi) \; \tilde{\varrho}_{\hslash_{n_k}}(s)]\,,
$$
explicitly as integrals with respect to the measures $\tilde\mu_s$.
This leads  to an equation describing the time-evolution of  characteristic
functions of $\tilde\mu_s$.
\begin{lemma}
\label{eq:ch1}
Suppose the same hypothesis as in Proposition \ref{prop:4}. Then for all $\xi=(p_0,q_0,\alpha_0)\equiv(z_0,\alpha_0)\in X^0$,
\begin{equation}
\label{eq:Gsxi}
\begin{split}
\mathscr{G}(s,\xi)&:=\lim_{k\to\infty} \Tr\big [{M}(s,\hslash,\xi)\; \mathcal{W}( \xi) \; \tilde{\varrho}_{\hslash_{n_k}}(s)\;\big] \\
&=\int_{X^0} m(s,\xi) \; e^{2\ci\Imm\langle q+{\ci}p,z_0\rangle} e^{\sqrt{2}\ci\Ree\langle \alpha_0,\alpha\rangle_{\mathfrak{H}^0}} \, d\tilde\mu_{s}(u)\,,
\end{split}
\end{equation}
with the variable $u=(p,q,\alpha)\in X^0$ and
\begin{equation}
\label{eq:msxi}
\begin{aligned}
\hspace{-.8in}{m}(s,\xi)&= \sum_{i=1}^n \frac{1}{2m_i}  \,\left[ \frac{\ci}{\sqrt{2}} \sum_{\nu=1}^d \bigl(\langle \alpha_0, {F}_i^\nu(s)\rangle
\langle \alpha, F_i^\nu(s)\rangle - \langle
 F_i^\nu(s), \alpha_0\rangle \langle F_i^\nu(s),\alpha\rangle \bigr) \right] \;+\;\mathrm{h.c.} \\
&- \sum_{i=1}^n \frac{1}{m_i}  \,\left[ \frac{\ci}{\sqrt{2}}\sum_{\nu=1}^d   \langle\alpha, F_i^\nu(s)\rangle  \langle F_i^\nu(s), \alpha_0\rangle \right]
\;+\;\mathrm{h.c.} \\
&- \sum_{i=1}^n \frac{1}{m_i} \left[\frac{i}{\sqrt{2}}\sum_{\nu=1}^d  \langle \alpha_0, F_i^\nu(s)\rangle \, {p}_i^\nu-
\sum_{\nu=1}^d \langle \alpha, F_i^\nu(s)\rangle p_{0,i}^\nu\right] \;+\;\mathrm{h.c.} \\
&- \sum_{j=1}^{n}\nabla_{x_j} V\big(q_1+s\frac{p_1}{m_1},\dots,q_n+s\frac{p_n}{m_n}\big) \,\cdot \big(q_{0,j}+s \frac{p_{0,j}}{m_j}\big)\,,
\end{aligned}
\end{equation}
where
\begin{eqnarray*}
F_i^\nu(s)= F_i^\nu(s)(p,q,k) = \left(\frac{\chi_i(k)}{\sqrt{2|k|}}   \,\epsilon^{\nu}_\lambda(k) \,e^{-2\pi\ci k \cdot (q_i+s\frac{p_i}{m_i})+\ci s |k|}\right)_{\lambda=1,\dots,d-1}\,.
\end{eqnarray*}
\end{lemma}
\begin{proof}
One needs to compute the limit for each term in the right hand side of \eqref{eq:Mhxi}. However, there are essentially three types:
\begin{eqnarray}
(i) & = & \nabla_{x_j}V\big(\hat q_1+s\frac{\hat p_1}{m_1},\dots,\hat q_n+s\frac{\hat p_n}{m_n}\big)\,\cdot\big(q_{0,j}+s \frac{p_{0,j}}{m_j}\big)\\
(ii) & = & \langle \alpha_0, r_i^\nu(s)\rangle \, \hat{p}_i^\nu\\
(iii) & = & \langle \alpha_0, {r}_i^\nu(s)\rangle \, a^*(r_i^\nu(s))\,,
\end{eqnarray}
and the others are similar. So, we only provide here the limits for (\textit{i})--(\textit{iii}). Note that
\begin{equation}
\label{eq:41}
 (i)=e^{\ci \frac{s}{\hslash} \hat{H}^0_\hslash} \,f(\hat q)\,  e^{-\ci \frac{s}{\hslash} \hat{H}^0_\hslash} \qquad \text{ with } \qquad f(x)
 =\nabla_{x_j}V(x)\cdot\big(q_{0,j}+s \frac{p_{0,j}}{m_j}\big)\,.
\end{equation}
Hence, one obtains
\begin{eqnarray*}
\Tr\big [(i)\; \mathcal{W}( \xi) \; \tilde{\varrho}_{\hslash_{n_k}}(s)\;\big]&=&
\Tr\big [f(\hat q)\; e^{-\ci \frac{s}{\hslash} \hat{H}^0_\hslash}\,\mathcal{W}( \xi) \; \tilde{\varrho}_{\hslash_{n_k}}(s) e^{\ci \frac{s}{\hslash} \hat{H}^0_\hslash}
\;\big]\\
&=& \Tr\big [f(\hat q)\; \mathcal{W}(\Phi_s^0(\xi)) \; {\varrho}_{\hslash_{n_k}}(s) \;\big]\,.
\end{eqnarray*}
On the other hand, one knows according to Lemma \ref{lemma:rhoti} that $({\varrho}_{\hslash_{n_k}}(s))_{k\in\N}$ admits a unique Wigner measure $\mu_s$ and
$(\Phi_s^0)_\sharp\tilde\mu_s=\mu_s$ ($\Phi_s^0$ is the linear free flow defined on $X^0$ by \eqref{eq:freeflowB}).
So, since $V\in \mathscr{C}^2(\R^{dn})$ and $\nabla V$ is bounded, one can use
Lemma \ref{lem:B1}--\eqref{eq:B2} and deduce
\begin{eqnarray*}
\lim_{k\to\infty} \Tr\big [f(\hat q)\; \mathcal{W}(\Phi_s^0(\xi)) \; {\varrho}_{\hslash_{n_k}}(s) \;\big] &=&
\int_{X^0}  e^{2\ci\Imm\langle z,\Phi_s^0(z_0)\rangle} e^{\sqrt{2}\ci\Ree\langle \Phi_s^0(\alpha_0),\alpha\rangle_{\mathfrak{H}^0}}  \,
f(q) \;\mathrm{d}\mu_s(u)\\
&=& \int_{X^0}   e^{2\ci\Imm\langle z,z_0\rangle} e^{\sqrt{2}\ci\Ree\langle \alpha_0,\alpha\rangle_{\mathfrak{H}^0}} \,
f(\Phi_s^0(q)) \;\mathrm{d}\tilde\mu_s(u) \,.
\end{eqnarray*}
Note that here $\Phi_s^0(z_0)$, $\Phi_s^0(\alpha)$ and $\Phi_s^0(q)$ designate respectively the action of $\Phi_s^0$ on those components.
Hence, we recover the claimed limit related to (\textit{i}) by replacing $f$ and $\Phi_s^0$ with their values given respectively in
\eqref{eq:41} and \eqref{eq:freeflowB}.
\\
For the second term (\textit{ii}), one writes
\begin{eqnarray*}
\langle\alpha_0,r_i^\nu(s)\rangle &=& \sum_{\lambda=1}^{d-1} \int_{\R^{d}} \overline{\alpha_0(k)}\,
\frac{\chi_i(k)}{\sqrt{2 |k|}}  \, \epsilon_\lambda^{\nu}(k) \,e^{-2\pi\ci k\cdot (\hat q_i+s \frac{\hat{p}_i}{m_i})+\ci s|k|} \, \mathrm{d}k \\
&=& \sum_{\lambda=1}^{d-1} \int_{\R^{d}} g_\lambda^i(k) \, \hat T(-2\pi k,\frac{2\pi s}{m_i} k) \,\mathrm{d}k\,,
\end{eqnarray*}
where  $g_\lambda^i(k)= \overline{\alpha_0(k)}\,
\frac{\chi_i(k)}{\sqrt{2 |k|}}  \, \epsilon_\lambda^{\nu}(k) \,e^{\ci s|k|}$ belongs to $L^1(\R^{d})$.
In particular, Fubini's theorem yields
\begin{eqnarray*}
\Tr\big [(ii)\; \mathcal{W}( \xi) \; \tilde{\varrho}_{\hslash_{n_k}}(s)\;\big] &=& \sum_{\lambda=1}^{d-1}  \int_{\R^{d}}
g_\lambda^i(k) \,  \Tr\big [\hat T(-2\pi k,\frac{2\pi s}{m_i} k)\, \hat{p}_i^\nu\; \mathcal{W}( \xi) \; \tilde{\varrho}_{\hslash_{n_k}}(s)\;\big] \,\mathrm{d}k
\\&=& \sum_{\lambda=1}^{d-1}  \int_{\R^{d}}
g_\lambda^i(k) \,  \Tr\big [\hat{p}_i^\nu \,\hat T(-2\pi k,\frac{2\pi s}{m_i} k) \; \mathcal{W}( \xi) \; \tilde{\varrho}_{\hslash_{n_k}}(s)\;\big] \,\mathrm{d}k
+\mathcal{O}(\hslash)\,.
\end{eqnarray*}
The last identity is a consequence of the commutation relations \eqref{eq:Tcc}. So, using dominated convergence with the aid of the assumption  (\hyperref[eq:Ass0]{$S_0^{(1)}$})
and Lemma \ref{lem:B1}--\eqref{eq:B1} with Fubini's theorem one deduces
\begin{eqnarray*}
\lim_k\Tr\big [(ii)\; \mathcal{W}( \xi) \; \tilde{\varrho}_{\hslash_{n_k}}(s)\;\big]&=&
\sum_{\lambda=1}^{d-1}  \int_{\R^{d}}
g_\lambda^i(k) \,  \lim_k\Tr\big [\hat{p}_i^\nu \; \mathcal{W}( p_0-2\pi k, q_0+\frac{2\pi s}{m_i} k,\alpha_0) \; \tilde{\varrho}_{\hslash_{n_k}}(s)\;\big] \,\mathrm{d}k
\\
&&\hspace{-1in}=\sum_{\lambda=1}^{d-1}  \int_{\R^{d}}
g_\lambda^i(k) \, \left( \int_{X^0} p_i^\nu \; e^{2\ci Q(k,p,q)} \,e^{2\ci\Imm\langle q+ip,z_0\rangle}\,
e^{\sqrt{2}\ci\Ree\langle \alpha_0,\alpha\rangle_{\mathfrak{H}^0}} \;
\mathrm{d}\tilde\mu_s(u)\right) \,\mathrm{d}k
\\
&&\hspace{-1in}= \int_{X^0} \langle\alpha_0,F_i^\nu(s)\rangle \,p_i^\nu \; e^{2\ci\Imm\langle q+ip,z_0\rangle} e^{\sqrt{2}\ci\Ree\langle \alpha_0,\alpha\rangle_{\mathfrak{H}^0}} \;
\mathrm{d}\tilde\mu_s(u)\,,
\end{eqnarray*}
with
$$
Q(k,p,q)=-\pi k\cdot ( q_i+s \frac{{p}_i}{m_i}) \,.
$$
For the last term (\textit{iii}), one writes
\begin{eqnarray*}
\Tr\big [(iii)\; \mathcal{W}( \xi) \; \tilde{\varrho}_{\hslash_{n_k}}(s)\;\big]&=& \sum_{\lambda=1}^{d-1}  \int_{\R^{d}}
g_\lambda^i(k) \, \underset{=\mathscr T(s,\xi)}{\underbrace{\Tr\big [\hat T(-2\pi k,\frac{2\pi s}{m_i} k)\,  a^*(r_i^\nu(s))\; \mathcal{W}( \xi) \;
\tilde{\varrho}_{\hslash_{n_k}}(s)\;\big]}} \,\mathrm{d}k
\end{eqnarray*}
Moreover, using the definition of creation-annihilation operators in \eqref{eq:aa},

\begin{eqnarray*}
\mathscr T(s,\xi)&=&\sum_{\lambda=1}^{d-1}  \int_{\R^{d}} \Tr\big [\hat T(-2\pi k,\frac{2\pi s}{m_i} k)\, \left(r_i^\nu(s)\right)_\lambda(k') \, a^*(k')\; \mathcal{W}( \xi) \;
\tilde{\varrho}_{\hslash_{n_k}}(s)\;\big] \,\mathrm{d}k' \\
&=&\sum_{\lambda=1}^{d-1}  \int_{\R^{d}}  f_\lambda^i(k')
\Tr\big [\hat T(-2\pi k,\frac{2\pi s}{m_i} k)\, \hat T(-2\pi k',\frac{2\pi s}{m_i} k') \, a^*(k')\; \mathcal{W}( \xi) \;
\tilde{\varrho}_{\hslash_{n_k}}(s)\;\big] \,\mathrm{d}k' \,.
\end{eqnarray*}
Hence, dominated convergence and Lemma \ref{lem:B2}--\eqref{eq:B3} give
\begin{align*}
\lim_k \mathscr T(s,\xi)&=\sum_{\lambda=1}^{d-1}  \int_{\R^{d}}  f_\lambda^i(k')  \int_{X^0}
e^{2\ci\Imm\langle q+ip,z_0\rangle} e^{\sqrt{2}\ci\Ree\langle \alpha_0,\alpha\rangle_{\mathfrak{H}^0}}
\, e^{-2\pi\ci  (k+k')\cdot (q_i+s\frac{p_i}{m_i})} \,\overline{\alpha_\lambda(k')} \,\mathrm{d}\tilde\mu_s(u) \;\mathrm{d}k' \\
&=   \int_{X^0} e^{2\ci\Imm\langle q+ip,z_0\rangle} e^{\sqrt{2}\ci\Ree\langle \alpha_0,\alpha\rangle_{\mathfrak{H}^0}}\,
\, e^{-2\pi\ci  k\cdot (q_i+s\frac{p_i}{m_i})} \,\langle \alpha, F_i(s)^\nu\rangle_{\mathfrak{H}^0} \,\mathrm{d}\tilde\mu_s(u)\,.
\end{align*}
Hence, by Fubini's theorem one obtains
\begin{align*}
\lim_k\Tr\big [(iii)\; \mathcal{W}( \xi) \; \tilde{\varrho}_{\hslash_{n_k}}(s)\;\big]&\\
&  \hspace{-1.1in}=\sum_{\lambda=1}^{d-1}  \int_{\R^{d}} \bigg(
g_\lambda^i(k) \,\int_{X^0} e^{2\ci\Imm\langle q+ip,z_0\rangle+\sqrt{2}\ci\Ree\langle \alpha_0,\alpha\rangle_{\mathfrak{H}^0}-2\pi\ci  k\cdot (q_i+s\frac{p_i}{m_i})} \,
\langle \alpha, F^\nu_i(s)\rangle_{\mathfrak{H}^0} \,\mathrm{d}\tilde\mu_s(u)\,\mathrm{d}k\bigg)\\
&\hspace{-1.1in}=  \int_{X^0}e^{2\ci\Imm\langle q+ip,z_0\rangle} e^{\sqrt{2}\ci\Ree\langle \alpha_0,\alpha\rangle_{\mathfrak{H}^0}} \, \langle \alpha_0, {r}_i^\nu(s)\rangle \,\langle \alpha, F^\nu_i(s)\rangle_{\mathfrak{H}^0} \,\mathrm{d}\tilde\mu_s(u)\,,
\end{align*}
with the variable $u=(p,q,\alpha)\in X^0$ and $\xi=(p_0,q_0,\alpha_0)\in X^0$.
\end{proof}

So, we obtain below the intended \emph{characteristic equation} as a consequence of Lemmas \ref{eq:ch2} and \ref{eq:ch1}.

\begin{prop}
\label{s4:propChar}
Let $\sigma\in[\frac 1 2,1]$ and assume that \eqref{eq:hypo0}, \eqref{eq:hypo1} and \eqref{eq:hypo2} are satisfied.
Let $(\varrho_\hslash)_{\hslash\in (0,1)}$  be a family of density matrices in $\mathscr{H}$ satisfying the
assumptions  \emph{(}\hyperref[eq:Ass0]{$S_0^{(1)}$}\emph{)}, \emph{(}\hyperref[eq:Ass1]{$S_1^{(1)}$}\emph{)} and \eqref{eq:Ass2}. Then the Wigner measures $ \lbrace \tilde \mu_{t} \rbrace_{t \in \R}$
provided by Proposition \ref{prop:4} satisfy the following characteristic equation for all $t,t_0\in\R$ and all $\xi=(z,\alpha)\in X^0$,
\begin{eqnarray}
\label{eq:char-v1}
&&\tilde\mu_{t}(e^{2\ci\Imm\langle \cdot,z\rangle} e^{\sqrt{2}\ci\Ree\langle \alpha,\cdot\rangle_{\mathfrak{H}^0}}) =
\tilde\mu_{t_0}(e^{2\ci\Imm\langle \cdot,z\rangle} e^{\sqrt{2}\ci\Ree\langle \alpha,\cdot\rangle_{\mathfrak{H}^0}}) -
     \ci \int_{t_0}^{t} \; \mathscr G(s,\xi) \;\mathrm{d}s ~,~
\end{eqnarray}
 where the term $\mathscr G(s,\xi)$ is given explicitly in Lemma   \ref{eq:ch1} by \eqref{eq:Gsxi}--\eqref{eq:msxi}.
\end{prop}

\section{Proofs of the main results}
\label{sec.proof}

We reconsider now the Newton--Maxwell equation \eqref{eq:Newton-Maxwell} or equivalently \eqref{eq:NMc} as a non-autonomous initial
value problem over the Hilbert spaces  $X^\sigma$ with $\sigma\in[\frac 1 2,1]$,
\begin{equation}
\label{eq:ivp}
  \left\{
    \begin{aligned}
    &\frac{d}{dt} {u}(t) \ = \vartheta(t,u(t))\,,& \\
    &u(0) \ = u_0 \in X^\sigma\,.& \\
    \end{aligned}
  \right.
\end{equation}
Here the non-autonomous vector field $\vartheta$  is defined in terms of the nonlinearity $G:X^\sigma\to X^\sigma$ of the Newton--Maxwell equation
\begin{subequations}
\label{eq:Gnl}
\begin{align}
       \big(G(u)\big)_{p_i}&=\frac{1}{m_i} \sum_{\nu=1}^d \big(
       p_i^\nu-A^\nu_i(q_i,\alpha)\big) \, \nabla_{q_i}A^\nu_i(q_i,\alpha)- \nabla_{q_i} V(q)\,,         \label{eq:Gp} \\
        \big(G(u)\big)_{q_i}&=-\frac{1}{m_i} \,A_i(q_i,\alpha)\,, \label{eq:Gq}\\
      \big(G(u)\big)_{\alpha_\lambda}(k)&=\ci\sum_{i=1}^n \frac{1}{m_i} \, \frac{\chi_i(k)}{\sqrt{2|k|}}
        \, \big(p_i-A_i(q_i,\alpha)\big) \cdot \epsilon_\lambda(k) \,e^{-2\pi\ci k\cdot q_i}\,.
        \label{eq:Gfa}
\end{align}
\end{subequations}
and the free flow $\Phi_t^0:X^\sigma\to X^\sigma$ given by
\begin{equation}
\label{eq:freeflowB-hiro}
\Phi_{t}^0(p,q,\alpha)=\,\left((p_i)_{i=1,\cdots,n}\,;(q_i+t \frac{p_i}{m_i})_{i=1,\cdots,n}\,; (e^{-\ci t|\cdot|}\alpha_\lambda)_{\lambda=1,\cdots,d-1}\right)\,
\end{equation}
as follows,
\begin{align}
\label{eq:Vf}
\vartheta(t,u)&=\Phi_{-t}^0\circ G\circ\Phi_t^0(u).
\end{align}
In other words, we have written here the Newton--Maxwell dynamical system in its interaction representation.
Note that the nonlinearity $G$ is very similar to $F$ given in \eqref{eq:F} and differs from the latter only
by the term $\frac{p_i}{m_i}$ in \eqref{eq:Fq}.

 \medskip
As a consequence of Proposition \ref{s1:prop1}, $\vartheta$  inherits the following properties from $F$.
\begin{lemma}
\label{s5.lem1}
Let $\sigma\in[\frac 1 2,1]$ and assume that
\eqref{eq:hypo0}, \eqref{eq:hypo1} and \eqref{eq:hypo2} are satisfied.
Then the vector field $\vartheta :\R\times X^\sigma\to X^\sigma$ is continuous and  bounded on bounded  subsets of $\R\times X^\sigma$.
\end{lemma}
Moreover, the strong  solutions to the Newton--Maxwell equation \eqref{eq:Newton-Maxwell} are related to the above initial value problem \eqref{eq:ivp} according to the following relation.
In fact, assume the hypothesis of Lemma \ref{s5.lem1}  is satisfied.
Then if
$$
\tilde u(\cdot)=\big(\tilde p(\cdot),\tilde q(\cdot),\tilde \alpha(\cdot))\in \mathscr C(I,X^\sigma)\cap \mathscr C^1(I,X^{\sigma-1}\big)
$$
is a strong solution to the Newton--Maxwell equation \eqref{eq:Newton-Maxwell} or equivalently \eqref{eq:NMc},
 then
$$
u(\cdot)=
 \big(\tilde p(\cdot),\tilde q(\cdot)-t \bigl(\frac{\tilde p_i}{m_i}\bigr)(\cdot), e^{\ci t|\cdot|} \tilde\alpha(\cdot)\big)\in \mathscr C^1(I,X^\sigma)
 $$
  is a solution of the non-autonomous initial value problem  \eqref{eq:ivp}.
  Reciprocally, if
  $$
  u(\cdot)=
 \big(p(\cdot),q(\cdot),\alpha(\cdot)\big)\in \mathscr C^1(I,X^\sigma)
 $$
 is a solution of the non-autonomous initial value problem  \eqref{eq:ivp},
 then
 \begin{equation}
 \label{eq:tiu}
 \tilde u(\cdot)=\big(p(\cdot),q(\cdot)+t\bigl(\frac{p_i}{m_i}\bigr)(\cdot),e^{-\ci t|\cdot|}\alpha(\cdot)\big)\in \mathscr C(I,X^\sigma)\cap \mathscr C^1(I,X^{\sigma-1})
 \end{equation}
 is a strong solution to the Newton--Maxwell equation \eqref{eq:Newton-Maxwell} or \eqref{eq:NMc}. Furthermore,
 recalling   the free flow $\Phi_t^0$ on $X^\sigma$,
the strong solution  $\tilde u(\cdot)$ in \eqref{eq:tiu} takes the expression below for all $t\in I$,
 $$
 \tilde u(t)= \Phi^0_t(u(t))\,.
 $$
 Additionally, thanks to the statement in Lemma \ref{s5.lem1} the initial value problem \eqref{eq:ivp} is equivalent to the Duhamel formula,
\begin{equation}
\label{int-form}
u(t)=u_0+\int_{0}^t  \;\vartheta(s,u(s)){\rm d}s\,,  \quad \forall  t\in I\,.
\end{equation}
The above discussion is summarized in the  proposition below.
\begin{prop}
\label{prop:eqNM}
Let $\sigma\in[\frac 1 2,1]$ and assume that
\eqref{eq:hypo0}, \eqref{eq:hypo1} and \eqref{eq:hypo2} are satisfied.
Let $I$ be a bounded open interval containing the origin.
Then the statements below are equivalent:
\begin{itemize}
\item[\rm (1)] $u(\cdot)\in  \mathscr C^1(I,X^\sigma)$  is a strong solution of the initial value problem  \eqref{eq:ivp},
\item[\rm (2)]
$u(\cdot)\in  \mathscr C(I,X^\sigma)$ satisfies the Duhamel formula \eqref{int-form}  for all times $t\in I$,
\item[\rm (3)]
the curve $t\in I\mapsto\Phi_t^0(u(t))\in \mathscr C(I,X^\sigma)\cap \mathscr C^1(I,X^{\sigma-1})$ is a strong solution to the Newton--Maxwell equation \eqref{eq:Newton-Maxwell} or equivalently \eqref{eq:NMc}.
\end{itemize}
\end{prop}

\begin{lemma}
\label{s5.lem2}
Let $\sigma\in[\frac 1 2,1]$ and assume that
\eqref{eq:hypo0}, \eqref{eq:hypo1} and \eqref{eq:hypo2} are satisfied.
Let $(\varrho_\hslash)_{\hslash\in (0,1)}$ be  a family of density matrices in $\mathscr{H}$ satisfying the assumptions \emph{(}\hyperref[eq:Ass0]{$S_0^{(1)}$}\emph{)}, \emph{(}\hyperref[eq:Ass1]{$S_1^{(1)}$}\emph{)}  and \eqref{eq:Ass2}.
Then for every compact time interval $J$ there exists
 a constant $C>0$ such that for all $u=(p,q,\alpha)\in X^\sigma$ and $t\in J$,
\begin{subequations}
\label{eq:vbd0}
\begin{align}
 \| \vartheta(t,u) \|_{{X}^\sigma} \, &\leq  C_T\, (|p|^2+\|\alpha\|_{L^2}^2+1)  ~,\\
 & \leq  C_T\, (\|u\|_{X^0}^2+1)~.
\end{align}
\end{subequations}
Moreover, if $(\tilde\mu_t)_{t\in\R}$ is the family of Wigner measures provided by Proposition  \ref{prop:4} then for any bounded open interval $I$,
\begin{equation}
\label{eq:vbound}
\int_I\int_{ {X}^\sigma} \| \vartheta(t,u) \|_{{X}^\sigma} \,{\rm d}\tilde\mu_{t}(u){\rm d}t < + \infty ~.
\end{equation}
\end{lemma}
\begin{proof}
The estimate \eqref{eq:vbd0} is a consequence  of the inequalities  \eqref{eq:Fbd} while \eqref{eq:vbound}
follows from Proposition \ref{prop:4} and \eqref{eq:vbd0}.
\end{proof}

We recast below the characteristic equation of Proposition \ref{s4:propChar} in a form  more  suitable for applying the general arguments in Appendix \ref{appx}.
\begin{cor}
\label{s5.cor1}
Let $\sigma\in[\frac 1 2,1]$ and assume that
\eqref{eq:hypo0}, \eqref{eq:hypo1} and \eqref{eq:hypo2} are satisfied.
Let $(\varrho_\hslash)_{\hslash\in (0,1)}$  be a family of density matrices in $\mathscr{H}$ satisfying the assumptions  \emph{(}\hyperref[eq:Ass0]{$S_0^{(1)}$}\emph{)}, \emph{(}\hyperref[eq:Ass1]{$S_1^{(1)}$}\emph{)} and \eqref{eq:Ass2}. Then the characteristic equation given in Proposition \ref{s4:propChar} takes the simple form:
\begin{equation}
\label{eq:char}
    \tilde\mu_{t}(e^{2\ci\pi \Ree\langle y,\cdot\rangle_{{X}^\sigma}}) = \tilde\mu_{t_0}(e^{2\ci\pi \Ree\langle y,\cdot\rangle_{{X}^\sigma}}) +
     2\ci\pi \int_{t_0}^{t}  \;\tilde\mu_{s} \big(e^{2\ci\pi \Ree \langle y,u \rangle_{{X}^\sigma}} \;\Ree\langle \vartheta(s,u),y \rangle_{{X}^\sigma}\big)  \,{\rm d}s~,~
\end{equation}
for all $t,t_0\in\R$ and $y\in {X}^\sigma$.
\end{cor}
\begin{proof}
The scalar product over $X^\sigma$ can be written as a duality bracket such that for all $u_1,u_2\in X^\sigma$,
$$
\langle u_1,  u_2\rangle_{X^\sigma}=\langle u_1, {\Lambda}^{\sigma} u_2\rangle_{X^0} \,,
$$
where $\Lambda$ is the unbounded operator over $X^0$ defined by
$$
{\Lambda}=\begin{bmatrix}
1 & 0 & 0\\
0& 1 & 0\\
0&0& 1+|k|^2
\end{bmatrix}
\,.
$$
In particular,  one has for every $y\in X^{2\sigma}$ and every $\tilde\xi=\Lambda^\sigma y\in X^0$,
\begin{align*}
\Ree\langle y, u\rangle_{{X}^\sigma} &= \Ree\langle \tilde\xi, u\rangle_{{X}^0}\,,\\
\Ree\langle \vartheta(s,u),y \rangle_{{X}^\sigma}&= \Ree\langle \vartheta(s,u),\tilde\xi \rangle_{{X}^0}\,,
\end{align*}
Moreover, one can check that
\begin{align}
\label{eq:40}
\mathscr G(s,\xi) &=-2\pi\,\tilde\mu_{s} \big(e^{2\ci\pi \Ree \langle \tilde \xi,u \rangle_{{X}^0}} \;\Ree\langle \vartheta(s,u),\tilde\xi \rangle_{{X}^0}\big) \\
&=-
2\pi\,\tilde\mu_{s} \big(e^{2\ci\pi \Ree \langle y,u \rangle_{{X}^\sigma}} \;\Ree\langle \vartheta(s,u),y \rangle_{{X}^\sigma}\big)\,,\nonumber
\end{align}
if one chooses
$$
\tilde\xi=(\frac{z_0}{\ci \pi},\frac{\alpha_0}{\sqrt{2}\pi} )\,\qquad \text{ while } \qquad \xi=(z_0,\alpha_0)\in X^0\,,
$$
where the quantity $\mathscr G(s,\xi)$ is given in \eqref{eq:Gsxi}.
Therefore, the characteristic equation \eqref{eq:char-v1} yields \eqref{eq:char} for all $y\in  X^{2\sigma}$.  The latter equation \eqref{eq:char} extends
then to all $y\in X^\sigma$ by dominated convergence and the bound \eqref{eq:vbound}.
So, to finish the proof it is enough to justify  \eqref{eq:40}. In fact, this is equivalent to proving
\begin{equation}
\label{eq:mvfi}
m(s,\xi)=-2\pi\;\Ree\langle \vartheta(s,u),\tilde \xi \rangle_{{X}^0} \,,
\end{equation}
where $m(s,\xi)$ is given in \eqref{eq:msxi}. Indeed, the latter equality follows from the identity
\begin{equation}
\label{eq:sympeuc}
m(s,\xi)= -2\Imm\bigl\langle \big(G\circ \Phi_t^0(u)\big)_{z}, \big(\Phi_t^0(\xi)\big)_z \bigr\rangle-\sqrt{2}\Ree\bigl\langle
\big(G\circ \Phi_t^0(u)\big)_{\alpha}, \big(\Phi_t^0(\xi)\big)_{\alpha} \bigr\rangle\,.
\end{equation}
Hence, using the symplectic and euclidian invariance of the free particle and field  dynamics respectively, one finds
that
\begin{eqnarray*}
m(s,\xi)&=&
-2\Imm\big\langle \big(\Phi_{-t}^0\circ G\,(\Phi_t^0(u))\big)_{z}\,, z_0 \big\rangle
-\sqrt{2} \Ree\big\langle \big(\Phi_{-t}^0\circ G\,(\Phi_t^0(u))\big)_{\alpha} \,,\alpha_0 \big\rangle \\
&=&
-2\Imm\big\langle \big(\vartheta(s,u)\big)_{z}\,, z_0 \big\rangle
- \sqrt{2}\Ree\big\langle \big(\vartheta(s,u)\big)_{\alpha} \,,\alpha_0 \big\rangle \\
&=&
-2\pi \Ree\big\langle \big(\vartheta(s,u)\big)_{z}\,,  \frac{z_0}{\ci\pi} \big\rangle
- 2\pi \Ree\big\langle \big(\vartheta(s,u)\big)_{\alpha} \,,\frac{\alpha_0}{\sqrt{2}\pi} \big\rangle\\
&=&
-2\pi \Ree\big\langle \vartheta(s,u)\,, \tilde\xi\big\rangle \,.
\end{eqnarray*}
\end{proof}

In the sequel,  we establish some regularity of the family  $(\tilde\mu_t)_{t\in\R}$ with respect to time.

\begin{lemma}
\label{lemma:3}
The Wigner measures $\{ \tilde{\mu}_t\}_{t\in\R}$ provided by Proposition \ref{prop:4} satisfy:
\begin{itemize}
\item[(1)] $\tilde{\mu}_t$ concentrates on $X^\sigma$, \emph{i.e.} $\tilde{\mu}_t(X^\sigma)=1$.
\item[(2)] $\R\ni t\to \tilde{\mu}_t\in\mathfrak{P}(X^\sigma)$ is weakly narrowly continuous.
\end{itemize}
\end{lemma}
\begin{proof}
On one hand, the bound \eqref{eq:estmut} gives
$$
\int_{X^0} \| \alpha\|_{\dot{\mathfrak{H}}^{\sigma}}^2\; {\rm d}\tilde{\mu}_t<\infty\,.
$$
On the other, Markov's inequality implies
$$
\tilde\mu_t(\{(p,q,\alpha)\in X^0: \lVert\alpha\rVert_{\dot{\mathfrak{H}}^{\sigma}}\geq \varepsilon\})\leq \frac{1}{\varepsilon^2} \,\tilde\mu_t(\lVert\alpha\rVert^2_{\dot{\mathfrak{H}}^{\sigma}})\,.
$$
So, taking $\varepsilon \to \infty$ yields
$$
\tilde\mu_t(\{(p,q,\alpha)\in X^0: \alpha\notin \dot{\mathfrak{H}}^{\sigma}\})=0\,.
$$
Therefore, the probability measures $\tilde\mu_t$ are concentrated on $X^{\sigma}$.  To prove the weak narrow continuity of
$(\tilde \mu_t)_{t\in\R}$, we claim that it is enough to show
that the function
\begin{equation}
\label{eq:34}
\R\ni t\to \int_{X^\sigma} e^{2\pi\ci \Ree\langle y,u\rangle_{X^\sigma}} \,{\rm d}\tilde\mu_t(u)\,,
\end{equation}
is  continuous for every $y\in X^\sigma$.
In fact, suppose that such statement is true. One remarks that for every arbitrary compact interval $J$ the family
$(\tilde\mu_t)_{t\in J}$ satisfies the bound
$$
\sup_{t\in J} \int_{X^\sigma} \|u\|^2_{X_w^\sigma} \,{\rm d}\tilde\mu_t(u) <\infty\,,
$$
since  the estimate \eqref{eq:estmut} holds true and  $\|u\|_{X_w^\sigma}\leq  \|u\|^2_{X^\sigma}$. So, thanks to the above bound one deduces that the family
$(\tilde\mu_t)_{t\in J}$ is  tight in $\mathfrak{P}(X^\sigma)$, where $X^\sigma$ is endowed here with the norm $\|\cdot\|_{X^\sigma_w}$. Therefore, applying the
Prokhorov's theorem one deduces that for any arbitrary compact interval $J$ the family $(\tilde\mu_t)_{t\in J}$  is sequentially compact for the weak narrow convergence topology. In particular, take $t_0\in \R$ and $\psi$ a bounded continuous  function on $(X^\sigma,\|\cdot\|_{X^\sigma_w})$.
Then the steps below follow:
\begin{itemize}
 \item[(1)] For any sequence $t_j\to t_0$ one can extract a subsequence denoted by $(t_{j_k})_{k\in\N}$ such that $\tilde\mu_{t_{j_k}}$ weakly narrowly converges to a provability measure $\mu$.
 \item[(2)] Thanks to the above convergence and the continuity of
 the function \eqref{eq:34}, one deduces that the characteristic function of
 $\mu_{t_0}$ and $\mu$ coincide and hence $\mu=\mu_{t_0}$.
 Note that we have used that the
 map $u\mapsto e^{2\pi\ci \Ree\langle y,u\rangle_{X^\sigma}}$ is bounded and continuous over  $(X^\sigma,\|\cdot\|_{X^\sigma_w})$.
 \item[(3)] Hence, one concludes that for every sequence $(t_j)$ there exists a subsequence $(t_{j_k})_k$ such that
              $$
     \lim_{k} \int_{X^\sigma} \psi(u) \,{\rm d}\tilde\mu_{t_{j_k}}=\int_{X^\sigma} \psi(u) \,{\rm d}\tilde\mu_{t_0}\,.
     $$
 \item[(4)] So there exists a unique limit point and therefore  the sequence $(\tilde\mu_{t_j})_j$ converges weakly narrowly to $\tilde\mu_{t_0}$.
\end{itemize}
Now, the continuity of the functions \eqref{eq:34} is a consequence of the characteristic equation \eqref{eq:char} and the estimate \eqref{eq:vbd0} which yield
\begin{align}
 \left|\tilde\mu_{t}(e^{2\ci\pi \Ree\langle y,\cdot\rangle_{{X}^\sigma}}) -\tilde\mu_{t_0}(e^{2\ci\pi \Ree\langle y,\cdot\rangle_{{X}^\sigma}})\right| &\lesssim
      \left|       \int_{t_0}^{t} \|y\|_{X^\sigma} \;\int_{X^\sigma} \| \vartheta(t,u) \|_{{X}^\sigma} \,{\rm d}\tilde\mu_s(u)\; {\rm d}s\right| \nonumber\\
      &\leq C \left|  \int_{t_0}^{t} \int_{X^\sigma} (\|u\|_{X^0}^2+1) \,{\rm d}\tilde\mu_s(u)\; {\rm d}s\right| \nonumber \\
      &\leq C |t-t_0|\,,  \label{eq:35}
\end{align}
for any $t,t_0$ in an arbitrary compact interval. In the last inequality \eqref{eq:35}, we have used the bound \eqref{eq:estmut}
in Proposition \ref{prop:4}.
\end{proof}

\bigskip
Next we apply the general results of the Appendix  \ref{appx}.   At this point, we suggest the reading of Appendix \ref{appx}.  In particular, the spaces
of smooth cylindrical functions $\mathscr{C}_{0,cyl}^{\infty}$ and the gradient $\nabla$ are defined there.
The proposition below relates the characteristic equation in \eqref{eq:char} to the Liouville equation introduced in Appendix \ref{appx}.

\begin{prop}
\label{s5.prop1}
 The family of Wigner measures  $(\tilde\mu_t)_{t\in \R}$ defined in  Proposition \ref{prop:4} satisfies  the Liouville equation
 \begin{equation}
\label{eq:Liou-hiro}
     \int_{I}\int_{X^\sigma}  \{\partial_{t} \phi(t,u) + \Ree \langle \vartheta(t,u), \nabla_{X^\sigma} ~\phi(t,u) \rangle_{ X^\sigma} \}~
 {\rm d}\mu_{t}(u){\rm d}t =0 \,,
\end{equation}
for any open bounded time interval $I$ containing the origin with
$\phi\in\mathscr{C}_{0,cyl}^{\infty}(I \times X^\sigma)$.
\end{prop}
\begin{proof}
 It is enough to apply Proposition \ref{prop:char-Liouv-CR} which establishes an equivalence between  Liouville equation  and characteristic equation. Indeed, the assumptions required to apply such result are true:
\begin{itemize}
\item[(1)]  $\R\ni t \mapsto \tilde\mu_{t}$ is weakly narrowly continuous in $\mathfrak{P}({X}^\sigma)$ according to Lemma \ref{lemma:3}.
\item [(2)] The bound
\begin{equation*}
\int_{I} \int_{{X}^\sigma} \| \vartheta(t,u) \|_{{X}^\sigma} \,{\rm d}\tilde\mu_{t}(u){\rm d}t < + \infty ~,
\end{equation*}
holds true for any bounded open interval $I$ thanks to Lemma \ref{s5.lem2} with the bound \eqref{eq:vbound}.
\end{itemize}
\end{proof}

\paragraph{\bf Proof of Theorem \ref{thm:main-1} :}
Let $\sigma\in[\frac 1 2,1]$ and assume that \eqref{eq:hypo0}, \eqref{eq:hypo1} and \eqref{eq:hypo2}
are satisfied.   Consider for  $u_*\in X^\sigma$ the family of  coherent states defined in \eqref{eq:cohdef}:
\begin{equation}
\label{eq:cohst}
\varrho_\hslash=\mathcal C_\hslash(u_*)\,.
\end{equation}
 Then according to Lemma \ref{lemma:coh} one checks that such family  $\{\varrho_\hslash\}_{\hslash\in(0,1)}$  satisfies the assumptions  (\hyperref[eq:Ass0]{$S_0^{(1)}$}), {(}\hyperref[eq:Ass1]{$S_1^{(1)}$}{)} and \eqref{eq:Ass2}.
 Hence, one can apply Propositions \ref{prop:4} and \ref{s5.prop1} and concludes that there exists $\{\tilde\mu_t\}_{t\in\R}$ a weakly narrowly  continuous
curve in $\mathfrak{P}(X^\sigma)$ satisfying the bound \eqref{eq:vbound} and the Liouville
 equation \eqref{eq:Liou-hiro}.

 So, we have at our disposal all the ingredients to apply  Theorem \ref{thm:gflow} on the existence of global generalized flow for abstract   initial value problems. Indeed, the condition  \ref{cd:a} in  Theorem \ref{thm:gflow} is satisfied according to the above discussion
  while the uniqueness condition in Theorem \ref{thm:gflow}-\ref{cd:b}  holds true by Proposition \ref{s1:prop2}. Hence, one concludes that the ensemble of
  initial data
  $$
\mathcal G=\left\{u_0\in X^\sigma: \exists u(\cdot) \text{ a global strong solution of Newton--Maxwell equation satisfying } \eqref{eq:ivp}  \right\}
  $$
is Borel and total with respect to the probability measure $\tilde \mu_0$ which is the unique Wigner
measure of the coherent state family in \eqref{eq:cohst}. Therefore, by Lemma \ref{lemma:coh} one deduces that
$$
\tilde\mu_0(\mathcal G)=\delta_{u_*}(\mathcal G)=1\,,
$$
and consequently any  $u_*\in X^\sigma$ belongs to $\mathcal G$. This implies that
 for any initial condition $u_0\in X^\sigma$ there exists a unique global solution
$u(\cdot)\in \mathscr{C}(\R,X^\sigma)$ of the Newton--Maxwell equation \eqref{eq:ivp}. Moreover, the generalized flow map
\begin{equation}
\label{eq:phitilda}
\begin{aligned}
\widetilde\Phi_t:  X^\sigma & \rightarrow & X^\sigma\\
u_0 &\mapsto & u(t)
\end{aligned}
\end{equation}
is well defined and Borel measurable. Finally, using Proposition \ref{prop:eqNM} which gives  the equivalence between the two formulations \eqref{eq:ivp}   and \eqref{eq:Newton-Maxwell} of the Newton--Maxwell equation, one shows the existence and uniqueness
of global solutions  for the latter with a generalized global flow defined by
\begin{equation}
\label{eq:phi}
\Phi_t(u_0)=\Phi_{t}^0\circ  \widetilde\Phi_t(u_0)\,,
\end{equation}
where $\Phi_{t}^0$ is the free flow given in  \eqref{eq:freeflowB-hiro}
and $\tilde \Phi_t$ is the global flow of the Newton--Maxwell initial value problem \eqref{eq:ivp} in the interaction representation given by  \eqref{eq:phitilda}.
\hfill$\square$

\bigskip

\paragraph{\bf Proof of Theorem \ref{thm:main-2} :}[$\delta=1$ and singleton Wigner set]
Let $\sigma\in[\frac 1 2, 1]$ and assume that \eqref{eq:hypo0}, \eqref{eq:hypo1} and \eqref{eq:hypo2}
 are satisfied. Consider a family of density matrices $(\varrho_\hslash)_{\hslash\in (0,1)}$  on $\mathscr{H}$ satisfying
 the assumptions (\hyperref[eq:Ass0]{$S_0^{(0)}$}), (\hyperref[eq:Ass1]{$S_1^{(1)}$}) and \eqref{eq:Ass2}. Define
 \begin{equation*}
\varrho_\hslash(t)= e^{-\ci \frac{t}{\hslash} \hat{H}_\hslash}  \, \varrho_\hslash \,  e^{\ci \frac{t}{\hslash} \hat{H}_\hslash}
\qquad \text{ and } \qquad  \tilde{\varrho}_\hslash (t)=
e^{\ci \frac{t}{\hslash} \hat{H}^0_\hslash}  \, \varrho_\hslash(t) \,  e^{-\ci \frac{t}{\hslash} \hat{H}^0_\hslash}\,.
 \end{equation*}
For  $t_0\in \R$ pick  $\tilde\mu_{t_0}$ a Wigner measure of the family $( \tilde{\varrho}_\hslash (t_0))_{\hslash\in(0,1)}$. The existence of such Wigner measure $\tilde\mu_{t_0}$ is a consequence of Proposition \ref{wig} and Lemma \ref{lemma:rhoSt}. Moreover, according to Definition \ref{def:wigner} this implies  that there exists a sequence $(\hslash_k)_{k\in\N}$ such that $\lim_k \hslash_k=0$ and
$$
\mathcal{M}( \tilde{\varrho}_{\hslash_k} (t_0), k\in \N)=\{\tilde\mu_{t_0}\}\,.
$$
So, by Proposition \ref{prop:4} there exists a subsequence (still denoted by $(\hslash_k)_{k\in\N}$) such that for each
$t\in\R$,
$$
\mathcal{M}( \tilde{\varrho}_{\hslash_k} (t), k\in\N)=\{\tilde\mu_{t}\}\,.
$$
Furthermore,  by  Proposition \ref{s4:propChar} the curve $t\in\R\to \tilde\mu_t\in \mathfrak{P}(X^\sigma)$ satisfies the characteristic equation \eqref{eq:char-v1} or equivalently \eqref{eq:char}. Therefore, using
Proposition \ref{s5.prop1} one deduces that  $\R\ni t\mapsto \tilde\mu_t$ is a measure-valued solution of the Liouville equation \eqref{eq:Liou-hiro}.
So applying the probabilistic representation in Theorem \ref{thm.eta} one obtains the existence of a measure
$$\eta\in\mathfrak P(\mathcal X),$$ with $\mathcal X\equiv X^\sigma\times\mathscr C(\bar I,X^\sigma)$, concentrated on the ensemble of solutions of the Newton--Maxwell equation \eqref{eq:ivp}, \emph{i.e.},
$$\eta(\mathcal F_I)=1,$$ where
$$
\mathcal F_I=\left\{(u_0,u(\cdot))\in \mathcal X : ~u(\cdot) \text{ satisfies } \eqref{eq:ivp} \text{ on } I \text{ with the initial condition } u_0\right\}\,.
$$
Moreover, Theorem \ref{thm.eta}-(ii) implies that for any
bounded Borel functions $\psi:X^\sigma\to \R$
\begin{equation}
\label{eq:pr1}
\int_{X^\sigma} \psi(u) \,{\rm d}\tilde\mu_t=\int_{\mathcal F_I} \psi\big(\Xi_t(u_0,u(\cdot))\big) \,{\rm d}\eta(u_0,u(\cdot))\;.
\end{equation}
Thanks to the proof of Theorem \ref{thm:main-1}, the generalized global flow $\widetilde\Phi_t$ in \eqref{eq:phitilda}  of the Newton--Maxwell equation \eqref{eq:ivp} is already constructed. Hence,  one proves that the following identity holds true on the set $\mathcal F_I$
$$
\Xi_t(u_0,u(\cdot))= \widetilde\Phi_t(\Xi_0(u_0,u(\cdot))=\widetilde \Phi_t(u_0)\,.
$$
Recall that $\Xi_t$ is the evaluation map defined in \eqref{app.evmap}.
Thus, the equality \eqref{eq:pr1} yields
\begin{eqnarray*}
\int_{X^\sigma} \psi(u) \,{\rm d}\tilde\mu_t(u) &=& \int_{\mathcal X} \psi\circ \widetilde\Phi_t(\Xi_0(u_0,u(\cdot))) \;{\rm d}\eta(u_0,u(\cdot))\\
&=& \int_{X^\sigma} \psi\circ \widetilde\Phi_t(u) \; {\rm d}\tilde\mu_0(u)\,,
\end{eqnarray*}
since $(\Xi_0)_\sharp\eta=\tilde\mu_0$ and $\widetilde\Phi_t$ is Borel measurable. So, one concludes that for all $t\in I$,
\begin{equation}
\label{eq:pr2}
\tilde\mu_t=(\widetilde\Phi_t)_{\sharp}\tilde\mu_0\,.
\end{equation}
In particular, since any Wigner measure $\mu_{t_0}$ of the family $( \tilde{\varrho}_\hslash (t_0))_{\hslash\in(0,1)}$ is equal to $(\widetilde\Phi_{t_0})_{\sharp}\tilde\mu_0$,
one concludes that such family of density matrices admits a unique Wigner measure, \emph{i.e.}
$$
\mathcal{M}(\tilde\varrho_\hslash(t_0), \hslash\in(0,1))=\{(\widetilde\Phi_{t_0})_{\sharp}\tilde\mu_0\}\,.
$$
Finally, using Lemmas \ref{lemma:rhoSt} and \ref{lemma:rhoti} and the fact that the time $t_0$ is arbitrary one concludes for all times $t\in\R$,
\begin{eqnarray*}
\mathcal{M}\big(\varrho_\hslash(t), \hslash\in(0,1)\big)&=& \left\{(\Phi_t^0)_{\sharp}\tilde\mu_t, \;\tilde \mu_t\in\mathcal{M}(\tilde{\varrho}_\hslash(t), \hslash\in(0,1))\right\}\\
&=&\left\{(\Phi_t^0\circ \widetilde\Phi_t)_{\sharp}\tilde\mu_0\right\}\\
&=&\left\{(\Phi_t)_{\sharp}\mu_0\right\}\,,
\end{eqnarray*}
since  $\Phi_t=\Phi_t^0\circ \widetilde\Phi_t$ is the Newton--Maxwell global flow and $\tilde\mu_0=\mu_0$ as a consequence of
$\tilde\varrho_\hslash(0)=\varrho_\hslash(0)=\varrho_\hslash$.
\hfill$\square$

\bigskip
\paragraph{\bf Proof of Theorem \ref{thm:main-2} (bis):}[$\delta>0$ and general Wigner set]
\label{proof:impro}
In this paragraph, we indicate how to consider a general family of density matrices $(\varrho_\hslash)_{\hslash\in(0,1)}$ without assuming  \eqref{eq:wigsingle}
 and on the other hand how to weaken the assumptions  (\hyperref[eq:Ass0]{$S_0^{(1)}$}) and (\hyperref[eq:Ass1]{$S_1^{(1)}$}) in Theorem \ref{thm:main-2} and replace them by the ones given in
   \eqref{eq:Ass0} and \eqref{eq:Ass1} for some
$\delta\in (0,1]$.  The two arguments are independent:

\medskip
\noindent
1) Suppose that $(\varrho_\hslash)_{\hslash\in(0,1)}$ is a family of density matrices on $\mathscr{H}$ satisfying (\hyperref[eq:Ass0]{$S_0^{(1)}$}), (\hyperref[eq:Ass1]{$S_1^{(1)}$})
 and
\eqref{eq:Ass2}. According to Definition \ref{def:wigner}, if $\mu_0\in \mathcal{M}\big(\varrho_\hslash, \hslash\in(0,1)\big)$ then there exists a
sequence $(\hslash_k)_{k\in\N}$, $\hslash_k\to 0$,  such that
$$
\mathcal{M}\big(\varrho_{\hslash_k}, k\in\N\big)=\{\mu_0\}\,.
$$
Hence, thanks to the previous results so far proved one concludes that for all $t\in\R$,
$$
\mathcal{M}\big(\varrho_{\hslash_k}(t), k\in\N\big)=\{(\Phi_t)_\sharp\mu_0\}\,,
$$
and consequently, one proves
$$
\{(\Phi_t)_\sharp\mu_0, \; \mu_0\in\mathcal{M}\big(\varrho_\hslash, \hslash\in(0,1)\big)\}\subset  \mathcal{M}\big(\varrho_\hslash(t), \hslash\in(0,1)\big)\,.
$$
Reciprocally, for any $t\in\R$ take a $\mu_t\in \mathcal{M}\big(\varrho_\hslash(t), \hslash\in(0,1)\big)$ then there exists a sequence
$(\hslash_k)_{k\in\N}$, $\hslash_k\to 0$, such that
$$
\mathcal{M}\big(\varrho_{\hslash_k}(t), k\in\N\big)=\{\mu_t\}\,.
$$
So, applying the previous results so far proved  one obtains
$$
\mathcal{M}\big(e^{\ci \frac{t}{\hslash}\hat H_\hslash} \,\varrho_{\hslash_k}(t) e^{-\ci \frac{t}{\hslash}\hat H_\hslash}, k\in\N\big)=
\mathcal{M}\big(\varrho_{\hslash_k}, k\in\N\big)=
\{(\Phi_{-t})_\sharp\mu_t\}\,.
$$
Thus, we conclude that $\mu_t=(\Phi_t)_\sharp\mu_0$ for some $\mu_0\in \mathcal{M}\big(\varrho_\hslash, \hslash\in(0,1)\big)$ and prove the
opposite inclusion. Note that we have used implicitly Lemma \ref{lemma:rhoSt}.

\medskip
\noindent
2) Suppose that $(\varrho_\hslash)_{\hslash\in(0,1)}$ is a family of density matrices on $\mathscr{H}$ satisfying \eqref{eq:Ass0}, \eqref{eq:Ass1} and
\eqref{eq:Ass2}. Then we will use an approximation argument: Let $\chi\in\mathscr{C}^{\infty}_{0}(\mathds{R})$ be a smooth cutoff function  such that $0\leq
\chi\leq 1$ and  $\chi\equiv 1$ in a neighborhood of the origin. Denote
$$
\chi_R(\cdot)=\chi\big(\frac \cdot R\big)\,\,
$$
and let $A=\hat N_\hslash+\hat q^2+\hat H_\hslash^0$. Then the family of density matrices
$$
\varrho_{\hslash,R}=\frac{\chi_R(A)\,\varrho_{\hslash}\,\chi_R(A)}{\Tr\left[\chi_R(A)\varrho_{\hslash}\chi_R(A)\right]}
$$
approximates $\varrho_{\hslash}$ in the trace norm  as $R\to \infty$ for fixed $\hslash$. Moreover, let
$$m(R)=\sup_{\hslash\in(0,1),t\in\R}\|\varrho_{\hslash}(t)-\varrho_{\hslash,R}(t)\|_{\mathscr{L}^1(\mathscr{H})},$$
where $\varrho_{\hslash,R}(t)=e^{-\ci\frac{t}{\hslash}\hat H_{\hslash}}\varrho_{\hslash,R}\
e^{\ci\frac{t}{\hslash}\hat H_{\hslash}}$. Thanks to the assumptions
\eqref{eq:Ass0} and \eqref{eq:Ass1} one deduces
$$
\lim_{R\to0} m(R)=0.
$$
 Furthermore,  (\hyperref[eq:Ass0]{$S_0^{(1)}$}), (\hyperref[eq:Ass1]{$S_1^{(1)}$}) and \eqref{eq:Ass2} are satisfied by
the family of density matrices $(\varrho_{\hslash,R})_{\hslash\in(0,1)}$. In fact, ${\rm d}\Gamma(|k|)$, $\hat N_\hslash$
and  ${\rm d}\Gamma(|k|^{2\sigma})$ commute with $A$ which leads to (\hyperref[eq:Ass1]{$S_1^{(1)}$}) and
\eqref{eq:Ass2} while for  (\hyperref[eq:Ass0]{$S_0^{(1)}$}) one uses the commutation relations.
Hence, applying the so far proved results above one deduces that  for all $t\in\R$,
$$
\mathcal{M}(\varrho_{\hslash,R}(t), \hslash\in(0,1))=\left\{(\Phi_t)_{\sharp}\mu_{0,R}, \;\mu_{0,R}\in \mathcal{M}(\varrho_{\hslash,R}, \hslash\in(0,1))\right\}\,.
$$
On the other hand for each $t\in\mathds{R}$ and $R$ sufficiently large, according to Definition \ref{def:wigner} from any sequence $(\hslash_k)_{k\in\N}$, $\hslash_k\to 0$, one can  extract a
subsequence -- still denoted the same -- such that
\begin{equation}
\label{eq:42}
\mathcal{M}(\varrho_{\hslash_{k}}(t), k\in\N)=\left\{\mu_t\right\}\,,\qquad  \qquad
\mathcal{M}(\varrho_{\hslash_{k},R}(t), k\in\N)=\left\{(\Phi_t)_{\sharp}\mu_{0,R}\right\}\,,
\end{equation}
and
\begin{equation}
\label{eq:43}
\mathcal{M}(\varrho_{\hslash_{k}}, k\in\N)=\left\{\mu_0\right\}\,,\qquad  \qquad
\mathcal{M}(\varrho_{\hslash_{k},R}, k\in\N)=\left\{\mu_{0,R}\right\}\,.
\end{equation}
Let $|\nu|$ be the total variation of measure $\nu$.
So, applying the general result on the comparison of Wigner measures in  \cite[Proposition~2.10 ]{ammari2011jmpa} yields
\begin{eqnarray*}
  &&|\mu_{t}-(\Phi_t)_{\sharp}\mu_{0,R}|\leq \liminf_{k\to\infty}
  \|\varrho_{\hslash_k}(t)-\varrho_{\hslash_k,R}(t)
  \|_{\mathscr{L}^{1}(\mathscr{H})}\leq m(R)\,,
  \end{eqnarray*}
and
\begin{eqnarray*}
  &&|\mu_{0}-\mu_{0,R}|\leq \liminf_{k\to\infty}
  \|\varrho_{\hslash_k}-\varrho_{\hslash_k,R}
  \|_{\mathscr{L}^{1}(\mathscr{H})}\leq m(R)\,,
\end{eqnarray*}
where the left hand side denotes the total variation of the signed measures
$\mu_t-(\Phi_t)_{\sharp}\mu_{0,R}$ and $\mu_0-\mu_{0,R}$. Therefore, the triangle inequity implies
\begin{equation}
\label{eq:44}
|\mu_t-(\Phi_t)_{\sharp}\mu_{0}|\leq
|\mu_t-(\Phi_t)_{\sharp}\mu_{0,R}|+
|\mu_{0,R}-\mu_{0}| \leq 2 m(R)\,,
\end{equation}
since in particular the total variation of $(\Phi_t)_{\sharp}\mu_{0,R}-(\Phi_t)_{\sharp}\mu_{0}$ and those of
$\mu_{0,R}-\mu_{0}$ are equal.  So, taking $R\to \infty $ in \eqref{eq:44} one deduces  that
\begin{equation}
\label{eq:45}
\mu_t=(\Phi_t)_{\sharp}\mu_{0}\,.
\end{equation}
Therefore, picking  any $\mu_t\in\mathcal{M}(\varrho_{\hslash}(t), \hslash\in (0,1))$ and using the above argument one shows that
$$
\mathcal{M}(\varrho_{\hslash}(t), \hslash\in (0,1))\subset\left\{(\Phi_t)_{\sharp}\mu_{0}, \mu_0\in\mathcal{M}(\varrho_{\hslash}, \hslash\in (0,1))\right\}\,.
$$
While the opposite inclusion is justified by taking any $\mu_0\in\mathcal{M}(\varrho_{\hslash}, \hslash\in (0,1))$ and then choosing and  extracting
sequences $(\hslash_k)_{k\in\N}$ such that \eqref{eq:42}--\eqref{eq:43} are satisfied; thus arriving to the conclusion  \eqref{eq:45} for some
$\mu_t\in\mathcal{M}(\varrho_{\hslash}(t), \hslash\in (0,1))$.  This ends the proof of Theorem~\ref{thm:main-2}.

\hfill$\square$

As an illustration of our method, we give further quantitative propagation estimates satisfied by the Newton--Maxwell flow  derived from the quantum dynamics of Pauli--Fierz.

\begin{prop}
\label{prop.propag.NM}
Assume \eqref{eq:hypo0}, \eqref{eq:hypo1} and   \emph{(}\hyperref[eq:hypo2]{$A_2^{(1)}$}\emph{)}. Then for any
probability measure $\mu_0\in\mathfrak{P}(X^1)$ satisfying
\begin{equation}
\label{eq.propag.1}
\int_{X^1} \|u\|_{X^1}^2 \;\mathrm{d}\mu_0(u)<+\infty\, \quad \text{ and } \quad \int_{X^1} \big( p^4+ \|\alpha \|^4_{\dot{\mathfrak{H}}^{1/2}}+\|\alpha\|_{L^2}^4 \big)\;\mathrm{d}\mu_0(u)<+\infty\,,
\end{equation}
there exists a constant $c>0$ such that for all times $t\in\mathbb{R}$,
\begin{equation}
\label{eq.propag.2}
\int_{X^1} p(t)^4+\|\alpha(t)\|_{\dot{\mathfrak{H}}^{1/2}}^4 \;\mathrm{d}\mu_0(u) \leq c \quad \text{ and } \quad \int_{X^1} \|\alpha(t)\|_{L^2}^4 \;\mathrm{d}\mu_0(u)\leq c \,e^{c |t|}\,,
\end{equation}
where $(p(t),q(t),\alpha(t))$ is the solution of the Newton--Maxwell equation \eqref{eq:Newton-Maxwell} at time $t$ starting from the initial condition $u=(p,q,\alpha)\in X^1$.
\end{prop}
\begin{proof}
Consider the following family of superposed coherent states defined in  \eqref{eq:cohmu},
\begin{equation*}
\varrho_\hslash= \int_{X^1} \mathcal C_\hslash(u) \,d\mu_0(u)\,.
\end{equation*}
According to Lemma \ref{lemma:examwig}, $(\varrho_\hslash)_{\hslash\in(0,1)}$ satisfies the assumptions  (\hyperref[eq:Ass0]{$S_0^{(1)}$}), (\hyperref[eq:Ass1]{$S_1^{(1)}$}) and (\hyperref[eq:Ass2]{$S_2^{(1)}$}).
So thanks to Theorems \ref{thm:main-1} and \ref{thm:main-2}, one concludes
\begin{equation}
\label{eq.proag.5}
\mathcal{M}\big(\varrho_\hslash(t),
\; \hslash\in (0,1)\big)=\{\mu_t\}=\{(\Phi_t)\,_\sharp\,\mu_0\}\,.
\end{equation}
On the other  hand, there exists a constant $c>0$ such that
for all $\hslash\in(0,1)$ and $t\in\mathbb{R}$,
\begin{equation}
\label{eq.propag.3}
\Tr\big[ \varrho_{\hslash}(t) \, (\hat H_\hslash^0)^2\big] \leq c\,,
\end{equation}
and
\begin{equation}
\label{eq.propag.4}
\Tr\big[ \varrho_{\hslash}(t) \, (\hat N_\hslash)^2\big] \leq c\,e^{c|t|}\,.
\end{equation}
The above inequalities are consequences of Lemma \ref{s3.lemE}, Corollary \ref{cor:2} and the fact that
\eqref{eq.propag.3}--\eqref{eq.propag.4} hold true at time $t=0$. Indeed, a direct computation yields
\begin{eqnarray*}
\Tr\big[ \varrho_{\hslash}\, (\hat H_\hslash^0)^2\big] &\leq& 2 \int_{X^1} \big(p^4+\hslash \|\alpha\|^2_{\dot{\mathfrak{H}}^1}+ \|\alpha\|^4_{\dot{\mathfrak{H}}^{1/2}} \big)\, \mathrm{d}\mu_0<+\infty \,,\\
\Tr\big[ \varrho_{\hslash} \, (\hat N_\hslash)^2\big]  &=&  \int_{X^1} \big(\hslash \|\alpha\|^2_{L^2}+ \|\alpha\|^4_{L^2} \big)\, \mathrm{d}\mu_0 <+\infty\,.
\end{eqnarray*}
So, applying the argument in \cite[Lemma 3.12]{ammari2011jmpa} one deduces
\begin{eqnarray}
\label{eq.propag.6}
\int_{X^1} \big(p^4+\|\alpha\|^4_{\dot{\mathfrak{H}}^{1/2}} \big)\, \mathrm{d}\mu_t  \leq & \liminf_{\hslash\to 0} \Tr\big[ \varrho_{\hslash}(t)\, (\hat H_\hslash^0)^2\big] & \leq c \,, \\
\label{eq.propag.7}
\int_{X^1} \|\alpha\|^4_{\dot{\mathfrak{H}}^{1/2}} \, \mathrm{d}\mu_t  \leq  & \liminf_{\hslash\to 0}
\Tr\big[ \varrho_{\hslash}(t) \, (\hat N_\hslash)^2\big]   &  \leq  c\,e^{c|t|}\,.
\end{eqnarray}
Hence, the claimed estimates \eqref{eq.propag.2} are consequences of  \eqref{eq.proag.5} and \eqref{eq.propag.6}--\eqref{eq.propag.7} .
\end{proof}

\bigskip
\appendix
\section{Measure theoretical techniques}
\label{appx}

In this appendix, we introduce a general abstract framework where a natural
relationship between characteristic equations, Liouville equations and
initial value problems is established and a generalized global flow is
constructed for the latter. Our discussion sums up and
  simplifies several technical results proved in
  \cite{MR4134153,MR3721874,rouffort2018general}, which we then use in a
  subtle way to construct the global dynamics of the Newton--Maxwell
  equation. In particular, this part is a crucial step for proving our main
  Theorems \ref{thm:main-1} and \ref{thm:main-2}. It is also worth noting
  that the techniques involved here are related to the close subject of
  {G}ross-{P}itaevskii and NLS hierarchies (see
  \cite{MR2988730,MR3500833,MR4065266,MR3360742}, and references therein).
Firstly, we recall the notion of Liouville equations over Hilbert spaces and
state their equivalence to characteristic equations. Secondly, we provide a
probabilistic representation to measure-valued solutions of Liouville
equations and finally we deduce the existence of a Borel global flow for such
initial value problem.

\bigskip
Consider a \emph{separable} Hilbert space $\mathcal H$ and a \emph{continuous} vector field  $ v : \R \times \mathcal{H} \to \mathcal{H}$ such that it is bounded on bounded sets. In all the sequel $I$ denotes a bounded open interval containing the origin.
\bigskip
\paragraph{\textit{Initial value problem:}}
 The following equation defines a non-autonomous initial value problem:
\begin{equation}
\label{eq:ivp:apx}\tag{{\it ivp}}
  \left\{
    \begin{aligned}
    &\frac{d}{dt} {u}(t) \ = v(t,u(t))\,,& \\
    &u(0) \ = u_0 \in \mathcal H\,.& \\
    \end{aligned}
  \right.
\end{equation}
A curve $t\in I\to u(t)\in\mathcal H$ is a strong solution to the above initial value problem
if and only if $u\in\mathscr C^1(I,\mathcal H)$ and  satisfies \eqref{eq:ivp:apx} for some $u_0\in\mathcal H$ and for all $t\in I$. In such case the Duhamel formula,
\begin{equation}
\label{eq:Duhap}
u(t)=u(0)+\int_0^t \;v(s,u(s))\;{\rm d}s \,,
\end{equation}
holds true for all $t\in I$. Reciprocally, any curve $u\in\mathscr C(I,\mathcal H)$ satisfying
 \eqref{eq:Duhap} is a strong solution of the initial value problem \eqref{eq:ivp:apx}.
 One of our goals is to establish the existence of global strong solutions to \eqref{eq:ivp:apx} under suitable assumptions. The method we describe here is quite different from Cauchy--Lipschitz and fixed point theories since it uses  instead measure theoretical techniques. Note that the vector field $v:\R\times \mathcal H\to\mathcal H$ is only assumed to be continuous, hence standard arguments of nonlinear analysis are not applicable.

\bigskip
\paragraph{\textit{Liouville equation:}}
Consider $\mathcal{H}$ as a real Hilbert space endowed with the real scalar product $\Ree \langle \cdot,\cdot \rangle_{\mathcal{H}}$. Define the space of smooth cylindrical functions on $\mathcal{H}$, denoted $\mathscr{C}_{0,cyl}^{\infty}(\mathcal{H})$, as the set of all functions of the form
$$
\phi = \psi \circ \pi
$$
where  $\psi\in \mathscr{C}_{0}^{\infty}(\mathbb{R}^{m})$, $m\in\N$, and $\pi : \mathcal{H} \to \mathbb{R}^{m}$  a projection of the form:
\begin{equation}
\label{eq:pi}
\pi:u \to \pi(u) = \big( \Ree\langle u,e_{1} \rangle_{\mathcal H},\dots, \Ree\langle u,e_{m} \rangle_{\mathcal H}\big) \,,
\end{equation}
with $(e_{1},\dots,e_{m})$ is an arbitrary orthonormal family of $\mathcal{H}$. Similarly,
one defines for a given open bounded interval $I$, the space of smooth cylindrical functions $ \mathscr{C}_{0,cyl}^{\infty}(I \times \mathcal{H})$ as the set of functions $\phi$ such that :
\[
 \forall (t,u) \in I \times \mathcal{H} ~,~~\phi(t,u) = \psi(t,\pi(u)) ~,~
 \]
for some  $\psi \in \mathscr{C}_{0}^{\infty}(I \times \mathbb{R}^{m})$, $m \in \mathbb{N}$, and $\pi$
as in \eqref{eq:pi}.

\bigskip
\begin{defn}
Let $I$ be  a bounded open interval containing the origin. A family of Borel probability measures $\{\mu_{t}\}_{ t \in I}$ on $\mathcal{H}$ is a measure-valued solution of the Liouville equation  associated to the vector field $v:\R\times \mathcal{H}\to \mathcal{H}$  if and only if for all $\phi \in \mathscr{C}_{0,cyl}^{\infty}(I \times \mathcal{H})$:
\begin{equation}
\label{eq:Liou}
     \int_{I} \int_{\mathcal H} \left\{
     \partial_{t} \phi(t,u) + \Ree \langle v(t,u), \nabla_{\mathcal H} ~\phi(t,u) \rangle_{\mathcal{H}}\right\} ~{\rm d}\mu_{t}(u){\rm d}t=0 \,,
\end{equation}
where  $\nabla_{\mathcal{H}}$ denotes the real differential in the Hilbert space $\mathcal{H}$.
\end{defn}

The following result establishes an equivalence between Liouville equations and characteristic equations.
A  proof can be deduced from \cite[Proposition 4.2]{rouffort2018general}.
\begin{prop}[Equivalence]
\label{prop:char-Liouv-CR}
Let $ v : \R \times \mathcal{H} \to \mathcal{H}$ be a continuous vector field such that it is bounded on bounded sets.
Let $I\ni t\to\mu_{t}$ a weakly narrowly continuous curve in $\mathfrak{P}(\mathcal{H})$ such that
the following bound holds true:
\begin{equation}
\label{eq:bdv}
 \int_{I} \int_{\mathcal{H}} || v(t,u) ||_{\mathcal{H}} \,{\rm d}\mu_{t}(u){\rm d}t < + \infty ~.
 \end{equation}
Then the two assertions below are equivalent:
\begin{itemize}
\item[(i)]  $ \lbrace \mu_{t} \rbrace_{t \in I}$ is a solution of the Liouville equation \eqref{eq:Liou}.
\item[(ii)]  $ \lbrace \mu_{t} \rbrace_{t \in I}$ solves the following characteristic equation:
$\forall t \in I$, $ \forall y \in \mathcal{H}$,
\begin{equation}
\label{2_eqchar2}
    \mu_{t}(e^{2\ci\pi \Ree\langle y,.\rangle_{\mathcal{H}}}) = \mu_{0}(e^{2\ci\pi \Ree\langle y,.\rangle_{\mathcal{H}}}) +
     2\ci\pi \int_{0}^{t} \mu_{s} \big(e^{2\ci\pi \Ree \langle y,u \rangle_{\mathcal H}} ~\Ree\langle v(s,u),y \rangle_{\mathcal H}\big) \,{\rm d}s ~.~
\end{equation}
\end{itemize}
\end{prop}

Next we state some convenient Borel properties of the ensemble of solutions and initial conditions for the
initial value problem \eqref{eq:ivp:apx}.  Define the space
$$
\mathcal{X} = \mathcal H \times \mathscr{C}(\overline{I},\mathcal H),
$$
endowed with the norm
\[
||(u_0,u(\cdot))||_{\mathcal{X}} = ||u_0||_{\mathcal H} + \sup_{t \in \overline{I}} ||u(t)||_{\mathcal H} \,.
\]

\begin{lemma}
\label{lem:Fset}
The ensemble of solutions
\begin{equation}
\label{eq:Fset}
\mathcal{F}_I=\left\{(u_0,u(\cdot))\in \mathcal X : ~u(\cdot) \text{ satisfies } \eqref{eq:ivp:apx} \text{ on } I \text{ with the initial condition } u_0\right\}\,
\end{equation}
is a Borel subset of $\mathcal X$.
\end{lemma}
\begin{proof}
The maps
\begin{eqnarray*}
T_t:\mathcal X  & \rightarrow & \C \\
(u_0,u(\cdot))  & \mapsto & u(t)-u_0-\int_0^t  v(s,u(s))\; {\rm d}s  \,,
\end{eqnarray*}
are continuous since the vector field $v$ is continuous and bounded on bounded sets. Thus, for all $t\in I$,
$$
T_t^{-1}(\{0\})
$$
is a Borel set. Take a dense countable set $\{t_i\}_{i\in\N}$ in $I$, one concludes that
$$
\mathcal F_I= \bigcap_{i\in\N} T_{t_i}^{-1}(\{0\})\,,
$$
since each $u(\cdot)$ is a continues curve over $\bar I$ and the Duhamel formula \eqref{eq:Duhap} holds true.
\end{proof}

The following result of measure theory is useful. A proof can be found in \cite[Theorem 3.9]{MR0226684}.

\begin{lemma}
\label{parthmes}
Let $X_{1}$ and $X_{2}$ be two complete separable metric spaces and $E_{1} \subset X_{1}$, $E_{2} \subset X_{2}$. Suppose that $E_{1}$ is a Borel set. Let $\psi$ be a measurable one-to-one map of $E_{1}$ into $X_{2}$ such that $\psi(E_{1})=E_{2}$.
Then $E_{2}$ is a Borel set of $X_{2}$.
\end{lemma}

\begin{lemma}
\label{lemma:Gi}
The ensemble of initial conditions
\begin{equation}
\label{eq:Gi}
\mathcal{G}_I = \lbrace u_0 \in \mathcal H : \,  \exists u \in \mathscr{C}(\bar I, \mathcal H)
\text{ such that  } (u_0,u)\in \mathcal F_I  \rbrace  \,
\end{equation}
is a Borel subset of $\mathcal H$.
\end{lemma}
\begin{proof}
Note that the space $\mathscr C(\bar I,\mathcal H)$ endowed with the sup norm is a complete and separable metric space  (separability follows by Stone--Weierstrass theorem and Dini's theorem).
Apply  Lemma \ref{parthmes}  with $X_{1} = (\mathcal{X},||.||_{\mathcal{X}})$,
$X_{2} = (\mathcal H, ||.||_{\mathcal H})$, the map
\[ \begin{array}{lrcl}
 \psi=\wp_0 :  & \mathcal{X} & \longrightarrow & \mathcal H \\
            &  (u_0,u(\cdot)) & \longmapsto     & u_0\,,

\end{array} \]
and  $E_{1} = \mathcal{F}_I$, $ E_{2} = \mathcal{G}_I$.   Then one concludes that
$$
\wp_0(\mathcal F_I)= \mathcal{G}_I
$$
is a Borel subset of $\mathcal H$.
\end{proof}

A crucial tool for constructing a generalized global flow is the following probabilistic representation of measure-valued solutions to  Liouville equations.  Such general result is proved in \cite[Proposition 4.1]{MR3721874} (see also \cite{MR2129498}) and  justifies the existence of a probability measure $\eta$ over $\mathcal H$ concentrated on strong solutions of the initial value problem \eqref{eq:ivp}. Define for  $t \in I$ the evaluation map $\Xi_{t}$ over $\mathcal{X}$ as,
\begin{equation}
\label{app.evmap}
\Xi_{t} : (u_0,u(\cdot)) \in \mathcal{X} \to u(t) \in \mathcal H \,.
\end{equation}

\begin{thm}[Probabilistic representation]
\label{thm.eta}
  Let $ t \in I \to \mu_{t} \in \mathfrak{P}( \mathcal{H})$ be a weakly narrowly continuous curve in $\mathfrak{P}( \mathcal{H})$ satisfying the bound \eqref{eq:bdv} and the Liouville equation
  \eqref{eq:Liou}. Then there exists $\eta$ a Borel probability measure over  the space $(\mathcal{X},||.||_{\mathcal{X}})$ satisfying:
\begin{itemize}

\item[(i)] The measure $\eta$ concentrates on the set $\mathcal F_I$,
$$
\eta(\mathcal F_I)=1.
$$
\item[(ii)] For all $t \in I$,
$$
\mu_{t} = (\Xi_{t})_\sharp \eta\,.
$$
\end{itemize}
\end{thm}

\begin{cor}
\label{Gmes}
Let $ t \in I \to \mu_{t} \in \mathfrak{P}( \mathcal{H})$ be a weakly narrowly continuous curve in $\mathfrak{P}( \mathcal{H})$ satisfying the bound \eqref{eq:bdv} and the Liouville equation
  \eqref{eq:Liou}. Suppose that for each $u_0\in\mathcal H$, the initial value problem \eqref{eq:ivp:apx} admits at most one strong solution on $I$. Then the ensemble of initial conditions $\mathcal{G}_I$ in \eqref{eq:Gi} is total with respect to $\mu_0$, \emph{i.e.}
$$
\mu_0(\mathcal{G}_I)=1\,.
$$
\end{cor}

\begin{proof}
Using Theorem  \ref{thm.eta} one concludes
$$
\mu_0(\mathcal{G}_I)=(\Xi_0)_\sharp\eta(\mathcal{G}_I)=\eta(\Xi_0^{-1}(\mathcal{G}_I))= \eta(\mathcal{F}_I)=1\,,
$$
since $\Xi_0^{-1}(\mathcal{G}_I)\supset \mathcal F_I$.
\end{proof}

The following result ensures the existence of global solutions for initial value problems under weak  assumptions. In particular, we will apply such result to the Newton--Maxwell equation while proving Theorem \ref{thm:main-1}.

\begin{thm}[Generalized global flow]
\label{thm:gflow}
Consider the initial value problem \eqref{eq:ivp:apx} with $\mathcal H$ a separable Hilbert space and  $ v : \R \times \mathcal{H} \to \mathcal{H}$ a continuous  vector field bounded on bounded sets. Suppose :
\begin{enumerate}[label=\textnormal{(\roman*)}]
\item  \label{cd:a} There exists  a weakly narrowly continuous curve
$ t \in \R \to \mu_{t} \in \mathfrak{P}( \mathcal{H})$  satisfying the estimate  \eqref{eq:bdv} and the Liouville equation \eqref{eq:Liou} on any bounded open interval $I$ containing the origin.
\item  \label{cd:b} There exists at most one solutions of the initial value problem \eqref{eq:ivp:apx} over any bounded open interval $I$ containing the origin.
\end{enumerate}
    Then for $\mu_0$-almost all initial conditions $u_0$ in $\mathcal H$ there exists a (unique) global strong solution to the initial value problem \eqref{eq:ivp:apx}. Moreover, the ensemble of initial conditions
\begin{equation}
\label{eq:G}
\mathcal G=\left\{u_0\in\mathcal H: \exists u(\cdot) \text{ a global strong solution of } \eqref{eq:ivp:apx} \text{ with the initial condition } u_0  \right\}
\end{equation}
is a Borel subset of $\mathcal H$ satisfying  $\mu_0(\mathcal{G})=1$ and for any time $t\in\R$ the  map

\[
\begin{array}{lrcl}
 \Phi_t :  & \mathcal{G} & \longrightarrow & \mathcal H \\
            &  u_0 & \longmapsto     & u(t),

\end{array}
\]
 is Borel measurable.
\end{thm}

\begin{proof}

 Consider the increasing  sequence of intervals $I_n=(-n,n)$. Then $\{\mathcal{G}_{I_n}\}_{n\in\N}$ is a decreasing sequence of measurable sets. So, using the uniqueness of strong solutions one concludes that
 $$
 \mathcal{G}=\bigcap_{n\in \N} \mathcal{G}_{I_n}\,.
 $$
Furthermore by Corollary \ref{Gmes},
$$
\mu_0(\mathcal{G})=\lim_n \mu_0(\mathcal{G}_n)=1\,.
$$
To prove the measurability of the maps $\Phi_t$, recall that the projection $\wp_0|_{\mathcal{F}_I}$ defined in the proof of Lemma \ref{lemma:Gi} is injective. So, $\wp_0$ is a bijection  from $\mathcal{F}_I$ onto $\mathcal{G}_I$ and its inverse map

\[ \begin{array}{lrcl}
 \wp_0^{- 1} :  & \mathcal{G}_I & \longrightarrow & \mathcal{F}_I \\
                  &  u_0                  & \longmapsto     & (u_0,u),

\end{array} \]
is  measurable thanks to Lemma \ref{parthmes}. Hence for $t\in\R$ and $I$ containing $t$, the composition

\[
\begin{array}{cccccccc}
 \Phi_t: & \mathcal{G}  & \overset{Id}{\longrightarrow}&  \mathcal{G}_I  & \overset{\wp_0^{-1}}{\longrightarrow} & \mathcal{F}_I
  &\overset{\Xi_t}{\longrightarrow} & \mathcal H\\
                & u_0 & \longrightarrow &       u_0    & \longrightarrow  & (u_0,u) & \longrightarrow  & u(t)
\end{array}
\]
yields a well-defined and measurable map.

\end{proof}

\medskip
\section{Semiclassical calculus}
\label{appy}
We state some specific results on convergence towards Wigner measures which are related to semiclassical analysis in finite and infinite dimensions. For a  related general discussion, we refer the reader to the work \cite{ammari2008ahp}. In particular, Lemmas \ref{lem:B1} and \ref{lem:B2} are key arguments in establishing the convergence towards the characteristic equations in Section \ref{sec:classical-limit} under optimal regularity conditions on the density matrices $(\varrho_\hslash)_{\hslash\in (0,1)}$.

\begin{lemma}
\label{lem:B1}
Let $(\hslash_j)_{j\in\N}$ be a sequence such that $\hslash_j\to 0$ and $(\varrho_{\hslash_j})_{j\in\N}$ a family of density matrices on $\mathscr H$ satisfying
 \emph{(}\hyperref[eq:Ass0]{$S_0^{(1)}$}\emph{)} and \emph{(}\hyperref[eq:Ass1]{$S_1^{(1)}$}\emph{)}. Suppose that
\begin{equation}
 \mathcal{M}\big(\varrho_{\hslash_j},j\in\N\big)=\{\mu\}
\end{equation}
for some Borel probability measure $\mu\in\mathfrak{P}(X^0)$.  Then the following assertions hold true for any $\xi=(z_0,\alpha_0)\in X^0$ and
for any $f\in\mathscr{C}^1(\R^{dn})$ such that $\nabla f$ is bounded,
\begin{subequations}
\begin{align}
\label{eq:B1}
\lim_j\Tr\left[ \varrho_{\hslash_j} \, \mathcal{W}(\xi) \, f(\hat p) \right]=\int_{X^0}
e^{2\ci\Imm\langle z,z_0\rangle} e^{\sqrt{2}\ci\Ree\langle \alpha_0,\alpha\rangle_{\mathfrak{H}^0}} \,f(p) \;{\rm d}\mu(u)\,,\\
\label{eq:B2}
\lim_j\Tr\left[ \varrho_{\hslash_j} \, \mathcal{W}(\xi) \, f(\hat q)\right]=\int_{X^0}
e^{2\ci\Imm\langle z,z_0\rangle} e^{\sqrt{2}\ci\Ree\langle \alpha_0,\alpha\rangle_{\mathfrak{H}^0}} \,f(q)\;{\rm d}\mu(u)\,,
\end{align}
\end{subequations}
where the variable $u=(z,\alpha)\in X^0$.
\end{lemma}
\begin{proof}
The arguments for \eqref{eq:B1} and \eqref{eq:B2} are similar. So, we consider for instance only the first limit.  Note that the quantity $\Tr\left[ \varrho_{\hslash_j} \, \mathcal{W}(\xi) \, f(\hat p) \right]$ is well-defined thanks to Lemma \ref{lem:Tsmo}, the assumptions  (\hyperref[eq:Ass0]{$S_0^{(1)}$})  and (\hyperref[eq:Ass1]{$S_1^{(1)}$}) and Taylor expansion for $f$ which yields that there exists $c>0$ such that
\begin{equation}
|f(x)|\leq c \, \langle x\rangle \,,\qquad \forall x\in \R^{dn}\,.
\end{equation}
Let $\chi\in\mathscr C_0^\infty(\R^{dn})$ be a smooth  function such that $0\leq \chi \leq 1 $, $\chi\equiv 1$ in a neighborhood of the origin and $\int_{\R^{dn}} \chi(x) dx=1$. Define respectively a mollifier and a cutoff function for all $R>0$ as
$$
\chi_R(\cdot)=R^{-dn}\chi(R \,\cdot) \qquad \text{ and } \qquad \chi^R(\cdot)=\chi(R^{-1}\,\cdot\,)\,.
$$
For simplicity, we denote
\begin{equation}
\label{simp.not.ap}
e^{2\ci\Imm\langle z,z_0\rangle} e^{\sqrt{2}\ci\Ree\langle \alpha_0,\alpha\rangle_{\mathfrak{H}^0}}= e^{\ci Q(\xi,u)}\,.
\end{equation}
We use an approximation argument based on the  inequality:
\begin{eqnarray}
\label{eq:B6}
\hspace{-.5in}\left|\Tr\left[ \varrho_{\hslash_j} \, \mathcal{W}(\xi) \, f(\hat p) \right]-
\int_{X^0} e^{\ci Q(\xi,u)}\,f(p) \;{\rm d}\mu\right| &\leq &
\left|\Tr\left[ \varrho_{\hslash_j} \, \mathcal{W}(\xi) \, (f-\chi_R *f \chi^R)(\hat p) \right]\right| \\
\label{eq:B7}
&& \hspace{-3in} +\left|\Tr\left[ \varrho_{\hslash_j} \, \mathcal{W}(\xi) \, \chi_R*f\chi^R(\hat p) \right]-
\int_{X^0} e^{\ci Q(\xi,u)} \,(\chi_R*f\chi^R)(p) \;{\rm d}\mu\right| \\
\label{eq:B8}
&& \hspace{-3in}
+\left|\int_{X^0} e^{\ci Q(\xi,u)} \, (f-\chi_R*f\chi^R)(p) \;{\rm d}\mu\right|\,.
\end{eqnarray}
Now, we estimate all the right hand sides. Using Lemma \ref{lem:Tsmo} and assumptions  (\hyperref[eq:Ass0]{$S_0^{(1)}$}) and
(\hyperref[eq:Ass1]{$S_1^{(1)}$}),  one obtains
\begin{eqnarray*}
r.h.s ~\eqref{eq:B6} &\leq& \|(\hat p^2+1)^{1/2} \varrho_\hslash \mathcal{W}(\xi) (\hat p^2+1)^{1/2}\|_{\mathscr{L}^1(\mathscr{H})} \,
\|(\hat p^2+1)^{-1} (f-\chi_R *f\chi^R)(\hat p) \|_{\mathscr{L}(L^2(\R^{dn}))} \\
&\lesssim& \sup_{x\in\R^{dn}}\left| \frac{f(x)-(\chi_R *f\chi^R)(x)}{|x|^2+1}\right|  \\
&\lesssim& \sup_{x\in\R^{dn}}\left| \frac{f(x)-f\chi^R(x)}{|x|^2+1} \right|+
\sup_{x\in\R^{dn}}\left| \frac{f(x)-\chi_R *f(x)}{|x|^2+1} \right|\underset{R\to 0}{\longrightarrow} 0 \,.
\end{eqnarray*}
On the other hand,  Lemma \ref{lem:apest} shows that the Wigner measure $\mu$ satisfies the bound
$$
\int_{X^0} (p^2+q^2) \, {\rm d}\mu(u)<\infty\,.
$$
Hence, one obtains
\begin{equation}
r.h.s ~\eqref{eq:B8}\leq \int_{X^0} (p^2+1){\rm d}\mu(u)\; \sup_{p\in\R^{dn}}\left| \frac{f(p)-(\chi_R *f\chi^R)(p)}{p^2+1} \right|
 \underset{R\to 0}{\longrightarrow} 0 \,.
\end{equation}
We estimate the right hand side of \eqref{eq:B7}. Since $g^R=\chi_R *f\chi^R \in\mathscr{C}^\infty_0(\R^{dn})$, then one can write using the Fourier Transform and the Heisenberg--Weyl operators,
\begin{align*}
g^R(\hat p)=\int_{\R^{dn}} \mathscr{F}(g^R)(k) \;e^{2\pi\ci k\cdot \hat p} \,{\rm d}k &=\int_{\R^{dn}} \mathscr{F}
(g^R)(k) \;\hat T(-2\pi k,0)\,{\rm d}k\\ &
=\int_{\R^{dn}} \mathscr{F}(g^R)(k) \;\mathcal{W}(-2\pi k,0,0)\,{\rm d}k  \,,
\end{align*}
with absolutely convergent (Bochner) integrals in $\mathscr{L}(\mathscr H)$. Hence, by Fubini's theorem for  all $\xi=(p,q,\alpha)\in X^0$,
\begin{align*}
\Tr\left[ \varrho_{\hslash_j} \, \mathcal{W}(\xi) \, g^R(\hat p) \right] &=\int_{\R^{dn}}
\mathscr{F}(g^R)(k) \,
\Tr\left[ \varrho_{\hslash_j} \, \mathcal{W}(\xi)  \mathcal{W}(-2\pi k,0,0)\right]\, {\rm d}k \\
&=\int_{\R^{dn}} \mathscr{F}(g^R)(k) \,
\Tr\left[ \varrho_{\hslash_j} \, \mathcal{W}(-2\pi k+p,q,\alpha)\right]\,e^{-\ci\pi \hslash k\cdot p}\,  {\rm d}k\,.
\end{align*}
So, dominated convergence, Definition \ref{def:wigner} and Fubini's theorem  yield
\begin{eqnarray*}
\lim_j \Tr\left[ \varrho_{\hslash_j} \, \mathcal{W}(\xi) \, g^R(\hat p) \right] &=&\int_{\R^{dn}}
\mathscr{F}(g^R)(k) \,\int_{X^0}
e^{\ci Q(\xi,u)} \,e^{ 2\pi\ci k \cdot p} \,{\rm d}\mu(u) \,{\rm d}k \\
&=&\int_{X^0} g^R(p) \,e^{\ci Q(\xi,u)}  \,{\rm d}\mu(u)\,.
\end{eqnarray*}
Therefore, the left hand side of \eqref{eq:B6} is arbitrary small when choosing $R$ sufficiently large and $j\to\infty$.
\end{proof}

\begin{lemma}
\label{lem:B2}
Let $(\hslash_j)_{j\in\N}$ be a sequence such that $\hslash_j\to 0$ and $(\varrho_{\hslash_j})_{j\in\N}$ a family of density matrices on $\mathscr H$ satisfying
 \emph{(}\hyperref[eq:Ass0]{$S_0^{(1)}$}\emph{)}, \emph{(}\hyperref[eq:Ass1]{$S_1^{(1)}$}\emph{)}. Suppose that
\begin{equation}
 \mathcal{M}\big(\varrho_{\hslash_j},j\in\N\big)=\{\mu\}
\end{equation}
for some Borel probability measure $\mu\in\mathfrak{P}(X^0)$.  Then for all $\xi=(z_0,\alpha_0)\in X^0$ and for all $f\in \mathfrak{H}^0$,
\begin{subequations}
\begin{align}
\label{eq:B3}
\lim_j\Tr\left[ \varrho_{\hslash_j} \, \mathcal{W}(\xi) \, \hat a^{*}(f)\right]&=\int_{X^0}  e^{2\ci\Imm\langle z,z_0\rangle}
e^{\sqrt{2}\ci\Ree\langle \alpha_0,\alpha\rangle_{\mathfrak{H}^0}} \,
\langle \alpha,f\rangle_{\mathfrak{H}^0} \;{\rm d}\mu(u)\,,\\
\label{eq:B5}
\lim_j\Tr\left[ \varrho_{\hslash_j} \, \mathcal{W}(\xi) \, \hat a(f)\right]&=\int_{X^0}
e^{2\ci\Imm\langle z,z_0\rangle} e^{\sqrt{2}\ci\Ree\langle \alpha_0,\alpha\rangle_{\mathfrak{H}^0}} \,
\langle f,\alpha\rangle_{\mathfrak{H}^0} \;{\rm d}\mu(u)\,,
\end{align}
\end{subequations}
where the variable $u=(p,q,\alpha)\in X^0$.
\end{lemma}
\begin{proof}
We use the same notation as \eqref{simp.not.ap}.
By linearity, it suffices to prove the following limits,
\begin{subequations}
\begin{align}
\label{eq:B9}
\lim_j\Tr\left[ \varrho_{\hslash_j} \, \mathcal{W}(\xi) \, \hat \phi(f)\right]=2\int_{X^0}  e^{\ci Q(\xi,u)} \,
\Ree\langle \alpha,f\rangle_{\mathfrak{H}^0} \;{\rm d}\mu(u)\,,\\
\label{eq:B10}
\lim_j\Tr\left[ \varrho_{\hslash_j} \, \mathcal{W}(\xi) \, \hat \pi(f)\right]=2\ci\int_{X^0}  e^{\ci Q(\xi,u)} \,
\Imm\langle f,\alpha\rangle_{\mathfrak{H}^0} \;{\rm d}\mu(u)\,,
\end{align}
\end{subequations}
where $\hat\phi$ and $\hat\pi$ are the fields operators defined as
\begin{equation}
\hat\phi(f)=\hat a^*(f)+\hat a(f)\,,\qquad \hat \pi(f)=\ci \hat a^*(f)-\ci \hat a(f)\,.
\end{equation}
The arguments for \eqref{eq:B9} and \eqref{eq:B10} are similar; so we will consider only the first one. Recall that $\hat\phi(\alpha_0)$ is a self-adjoint operator
satisfying  $W(f)=e^{\frac{\ci}{\sqrt{2}}\hat\phi(f)}$.
We will use an approximation argument based on the  inequality:
\begin{eqnarray}
\label{eq:B11}
\hspace{-.5in}\left|\Tr\left[ \varrho_{\hslash_j} \, \mathcal{W}(\xi) \, \hat\phi(f) \right]-
2\int_{X^0} e^{\ci Q(\xi,u)} \,\Ree\langle \alpha,f\rangle_{\mathfrak{H}^0} \;{\rm d}\mu\right| &\leq & \nonumber \\
&&\hspace{-3in}\left|\Tr\left[ \varrho_{\hslash_j} \, \mathcal{W}(\xi) \, \left(\hat\phi(f)-\frac{W(\sqrt{2}s f)-Id}{\ci s} \right)\right]\right| \\
\label{eq:B12}
&& \hspace{-3in} +\left|\Tr\left[ \varrho_{\hslash_j} \, \mathcal{W}(\xi) \, \frac{W(\sqrt{2}s f)-Id}{\ci s} \right]-
\int_{X^0} e^{\ci Q(\xi,u)} \,\frac{e^{2s\ci\Ree\langle\alpha,f\rangle_{\mathfrak{H}^0}}-1}{\ci s} \;{\rm d}\mu\right| \\
\label{eq:B13}
&& \hspace{-3in}
+\left|\int_{X^0} e^{\ci Q(\xi,u)} \, \left(\frac{e^{2s\ci\Ree\langle\alpha,f\rangle_{\mathfrak{H}^0}}-1}{\ci s}-2\Ree\langle \alpha,f\rangle_{\mathfrak{H}^0}
\right)\;{\rm d}\mu\right|\,.
\end{eqnarray}
Dominated convergence and Lemma \ref{lem:apest} imply that the right hand side of \eqref{eq:B13} tends to zero as $s\to 0$ since
the assumptions (\hyperref[eq:Ass0]{$S_0^{(1)}$})  and (\hyperref[eq:Ass1]{$S_1^{(1)}$}) lead the bound
$$
\int_{X^0} \|\alpha\|_{\mathfrak{H}^0}^2 \,{\rm d}\mu<\infty.
$$
On the other hand,  for any $s\in\R^*$ the right hand side of \eqref{eq:B12} vanishes when $\hslash_j\to 0$ as a consequence of Definition \ref{def:wigner}.
So, it remains to prove that the right hand side of  \eqref{eq:B11} can be made arbitrary small uniformly in $\hslash_j$ by choosing $s$ sufficiently close to $0$.
In fact, one has
\begin{align}
\label{eq:B14}
\left|\Tr\left[ \varrho_{\hslash_j} \, \mathcal{W}(\xi) \, \left(\hat\phi(f)-\frac{W(\sqrt{2}sf)-Id}{\ci s} \right)\right]\right|
& \leq \|(\hat N_\hslash+1)^{1/2} \varrho_\hslash \mathcal{W}(\xi) (\hat N_\hslash+1)^{1/2}\|_{\mathscr{L}^1(\mathscr{H})} \\ \label{eq:B15}
&\hspace{-1in}\times \|(\hat N_\hslash+1)^{-1/2} \left(\hat\phi(f)-\frac{W(\sqrt{2}sf)-Id}{\ci s} \right) (\hat N_\hslash+1)^{-1/2}\|_{\mathscr{L}(\mathscr H)}\,.
\end{align}
So, according to Lemma \ref{lem:estN} and assumption (\hyperref[eq:Ass1]{$S_1^{(1)}$}), the right hand side of \eqref{eq:B14} is uniformly bounded with respect to $\hslash_j$.
Moreover,  Taylor expansion yields
$$
\hat\phi(f)-\frac{W(\sqrt{2}sf)-Id}{\ci s} =\frac{1}{s}\int_0^s \hat\phi(f) \, (Id-W(\sqrt{2}\tau f)) \,{\rm d}\tau\,,
$$
and on the other hand the uniform inequality \cite[Lemma 3.1]{ammari2008ahp} gives
$$
\|(Id-W(\sqrt{2}\tau f)) (\hat N_\hslash+1)^{-1/2}\|_{\mathscr{L}(\mathscr H)} \lesssim |\tau| \;\|f\|_{\mathfrak{H}^0} \,.
$$
Hence, by Lemma \ref{lemma:2} and the above inequalities one obtains
\begin{eqnarray*}
r.h.s ~\eqref{eq:B15} \lesssim  \left|\frac{1}{s} \int_0^s |\tau| \; {\rm d}s\right| \underset{s\to 0}{\longrightarrow} 0\,,
\end{eqnarray*}
uniformly with respect to $\hslash_j$.

\end{proof}

\begin{lemma}
\label{lemma:rhoti}
Let $(\varrho_\hslash)_{\hslash\in(0,1)}$ be a family of density matrices on $\mathcal H$ satisfying
\emph{(}\hyperref[eq:Ass0]{$S_0^{(1)}$}\emph{)}  and \emph{(}\hyperref[eq:Ass1]{$S_1^{(1)}$}\emph{)}. Define
\begin{equation}
  \tilde{\varrho}_\hslash =
e^{-\ci \frac{t}{\hslash} \hat{H}^0_\hslash}  \, \varrho_\hslash \,  e^{\ci \frac{t}{\hslash} \hat{H}^0_\hslash}\,,
\end{equation}
then the family $(\tilde\varrho_\hslash)_{\hslash\in(0,1)}$ satisfies the assumption \eqref{eq:wigcd} with $\delta=1$ and
$$
\mathcal{M}({\varrho}_\hslash, \hslash\in(0,1))= \left\{(\Phi_t^0)_{\sharp}\tilde\mu, \;\tilde \mu\in\mathcal{M}(\tilde{\varrho}_\hslash, \hslash\in(0,1))\right\}\,,
$$
where $\Phi_t^0$ is the free flow defined on $X^0$  as
\begin{equation}
\label{eq:freeflowB}
\Phi_{t}^0(p,q,\alpha)=\,\left((p_i)_{i=1,\cdots,n}\,;(q_i+t \frac{p_i}{m_i})_{i=1,\cdots,n}\,; (e^{-\ci t|\cdot|}\alpha_\lambda)_{\lambda=1,\cdots,d-1}\right)\,.
\end{equation}
\end{lemma}
\begin{proof}
First notice  that
$$
\Tr[\tilde{\varrho}_\hslash \,(\hat N_\hslash+\hat H_\hslash^0)  ]=\Tr[{\varrho}_\hslash \,(\hat N_\hslash+\hat H_\hslash^0)]\,,
$$
since $\hat{H}^0_\hslash$  strongly commutes with the operator  $\hat N_\hslash$. So, one only needs to check that
$$
\Tr[\tilde{\varrho}_\hslash \,\hat q^2  ]<\infty\,.
$$
The last statement follows by Lemma \ref{lemma:q2} and the spectral decomposition of $\varrho_\hslash$.  Hence, this
ensures that the family  $(\tilde\varrho_\hslash)_{\hslash\in(0,1)}$ satisfies the assumption \eqref{eq:wigcd} with $\delta=1$.

Now take $\mu\in \mathcal{M}({\varrho}_\hslash, \hslash\in(0,1))$. Then there exists a sequence $\hslash_k$ such that $\hslash_k\to 0$ and
$$
\mathcal{M}({\varrho}_{\hslash_k}, k\in\N)=\{\mu\}\,.
$$
In particular, Definition \ref{def:wigner} implies that for any $\xi=(p_0,q_0,\alpha_0)\equiv(z_0,\alpha_0)\in X^0$,
\begin{equation}
\lim_{k} \Tr\big[ \varrho_{\hslash_k} \;\mathcal{W}\big(\xi\big)\big] =
\int_{X^0}e^{2\ci\Imm\langle z,z_0\rangle} e^{\sqrt{2}\ci\Ree\langle \alpha_0,\alpha\rangle_{\mathfrak{H}^0}}\; \mathrm{d}\mu(u)\,,
\end{equation}
where $u=(z,\alpha)$. So, the latter identity yields for  $\tilde\xi=(\pi q,-\pi p,\sqrt{2}\pi \alpha)$,
 \begin{align*}
 \Tr\big[ \tilde\varrho_{\hslash_k} \;\mathcal{W}\big(\xi\big)\big] &=
\Tr\big[ \varrho_{\hslash_k} \;e^{\ci \frac{t}{\hslash} \hat{H}^0_\hslash}\mathcal{W}\big(\xi\big) \,e^{-\ci \frac{t}{\hslash} \hat{H}^0_\hslash}\big]\\
 &= \Tr\big[ \varrho_{\hslash_k} \;\mathcal{W}\big(\xi_t\big) \big]\,,
\end{align*}
with $\xi_t=(p_0,q_0+tp_0, e^{-\ci t|\cdot|}\alpha_0)\equiv ((q_0+tp_0)+\ci p_0, e^{-\ci t|\cdot|}\alpha_0) \in X^0$. Hence, one obtains
\begin{align*}
\lim_{k} \Tr\big[ \tilde\varrho_{\hslash_k} \;\mathcal{W}\big(\xi\big)\big]&=\int_{X^0}
e^{2\ci\Imm\langle z,z_0(t)\rangle} e^{\sqrt{2}\ci\Ree\langle \alpha_0(t),\alpha\rangle_{\mathfrak{H}^0}}\; \mathrm{d}\mu(u)\\
&=\int_{X^0} e^{2\ci\Imm\langle z,z_0\rangle} e^{\sqrt{2}\ci\Ree\langle \alpha_0,\alpha\rangle_{\mathfrak{H}^0}}\; \mathrm{d}(\Phi_{-t}^0)_\sharp\mu(u)\,,
\end{align*}
where $u=(z,\alpha)$, $z_0(t)=(q_0+tp_0)+\ci p_0$ and $\alpha_0(t)=e^{-\ci t|\cdot|}\alpha_0$. Note that we have used for all $u=(z,\alpha)\equiv(p,q,\alpha)\in X^0$ and $t\in\R$,
$$
\Imm\langle q+ip,z_0(t) \rangle=\Imm\langle q-tp+ip,z_0 \rangle \,, \qquad \Ree\langle  \alpha,\alpha_0(t) \rangle_{\mathfrak{H}^0}=\Ree\langle  e^{it|\cdot|}\alpha,\alpha_0 \rangle_{\mathfrak{H}^0}\,.
$$
Therefore, the following statement holds true
$$
 \mathcal{M}({\tilde\varrho}_{\hslash_k}, k\in\N)=\{\tilde\mu\}\, \qquad \text{ and } \qquad \tilde\mu=(\Phi_{-t}^0)_\sharp\mu \,.
$$
This proves
$$
\mathcal{M}({\varrho}_\hslash, \hslash\in(0,1))\subset\left\{(\Phi_t^0)_{\sharp}\tilde\mu, \;\tilde \mu\in\mathcal{M}(\tilde{\varrho}_\hslash, \hslash\in(0,1))\right\}\,.
$$
On the other hand, starting from $\tilde\mu\in \mathcal{M}({\tilde\varrho}_\hslash, \hslash\in(0,1))$ and mimicking the same argument one shows the  opposite
inclusion.
\end{proof}

\bibliographystyle{plain}

\end{document}